\documentclass[letterpaper,11pt]{amsart}
\usepackage{amsmath,amssymb,amsthm,pinlabel,tikz,hyperref,mathrsfs,color, thmtools}
\usepackage{verbatim}
\usepackage{float}
\usepackage{caption}
\usepackage{subcaption}
\usepackage{enumitem}

\newcommand{\nc}{\newcommand}
\nc{\dmo}{\DeclareMathOperator}
\dmo{\ra}{\rightarrow}
\dmo{\Prob}{\mathbb{P}}
\dmo{\E}{\mathbb{E}}
\dmo{\N}{\mathbb{N}}
\dmo{\Z}{\mathbb{Z}}
\dmo{\Q}{\mathbb{Q}}
\dmo{\R}{\mathbb{R}}
\dmo{\C}{\mathcal{C}}
\dmo{\X}{\mathcal{X}}
\dmo{\U}{\mathcal{U}}
\dmo{\T}{\mathcal{T}}
\dmo{\F}{\mathcal{F}}
\dmo{\AC}{\mathcal{AC}}
\dmo{\w}{\omega}
\dmo{\MIN}{\mathcal{MIN}}
\dmo{\Mod}{Mod}
\dmo{\PMod}{PMod}
\dmo{\PMF}{\mathcal{PMF}}
\dmo{\Mat}{Mat}
\dmo{\supp}{supp}
\dmo{\UE}{\mathcal{UE}}
\dmo{\vol}{vol}
\dmo{\B}{B}
\dmo{\PB}{PB}
\dmo{\PR}{PSL(2,\mathbb{R})}
\dmo{\GL}{GL(k, \mathbb{C})}
\dmo{\SL}{SL(2, \mathbb{Z})}
\dmo{\Isom}{Isom}
\dmo{\RP}{\mathbb{R} \mathrm{P}}
\dmo{\I}{\mathcal{I}}
\dmo{\el}{\ell_{\C}}
\dmo{\NN}{\mathcal{N}}
\dmo{\rk}{rank}
\dmo{\tr}{tr}
\dmo{\llangle}{\langle\langle}
\dmo{\rrangle}{\rangle\rangle}
\dmo{\Unif}{Unif}
\dmo{\Out}{Out}

\usetikzlibrary{decorations.markings, patterns}
\tikzset{->-/.style={decoration={
  markings,
  mark=at position #1 with {\arrow{>}}},postaction={decorate}}}

\nc{\nt}{\newtheorem}

\nt{theorem}{Theorem}

\newtheorem{thm}{{\bf Theorem}}[section]
\newtheorem{lem}[thm]{{\bf Lemma}}
\newtheorem{cor}[thm]{{\bf Corollary}}
\newtheorem{prop}[thm]{{\bf Proposition}}

\newtheorem{property}[thm]{Property}
\newtheorem{claim}[thm]{Claim} 
\newtheorem{remark}[thm]{Remark}

\newtheorem{definition}[thm]{Definition}
\numberwithin{equation}{section}

\title[Translation lengths of random hyperbolic isometries]{Linear growth of translation lengths of random isometries on Gromov hyperbolic spaces and Teichm{\"u}ller spaces}

\date{\today}
\author{Hyungryul Baik}
\address{%
		Department of Mathematical Sciences, KAIST\\
		291 Daehak-ro Yuseong-gu, Daejeon, 34141, South Korea 
}
\email{%
        hrbaik@kaist.ac.kr
}

\author{Inhyeok Choi}

\address{%
		Department of Mathematical Sciences, KAIST\\
		291 Daehak-ro Yuseong-gu, Daejeon, 34141, South Korea
}
\email{%
        inhyeokchoi48@gmail.com
        }

\author{Dongryul M. Kim}

\address{%
	Department of Mathematics, Yale University\\
	219 Prospect Street, New Haven, CT 06511, USA
}
\email{%
	dongryul.kim@yale.edu
}

\begin{document}
\begin{abstract}
	We investigate the translation lengths of group elements that arise in random walks on the isometry groups of Gromov hyperbolic spaces. In particular, without any moment condition, we prove that non-elementary random walks exhibit at least linear growth of translation lengths.  As a corollary, almost every random walk on mapping class groups eventually becomes pseudo-Anosov and almost every random walk on $\mathrm{Out}(F_n)$ eventually becomes fully irreducible. If the underlying measure further has finite first moment, then the growth rate of translation lengths is equal to the drift, the escape rate of the random walk.
	
	We then apply our technique to investigate the random walks induced by the action of mapping class groups on Teichm{\"u}ller spaces. In particular, we prove the spectral theorem under finite first moment condition, generalizing a result of Dahmani and Horbez.

\noindent{\bf Keywords.} Random walk, Gromov hyperbolic space, spectral theorem, Teichm{\"u}ller space, weakly hyperbolic group

\noindent{\bf MSC classes:} 20F67, 30F60, 57M60, 60G50
\end{abstract}

\maketitle

%
%

\section{Introduction} \label{sec:introduction}

Throughout, $X$ denotes either a separable geodesic Gromov hyperbolic space or the Teichm{\"u}ller space of a closed orientable surface of genus at least two, which is equipped with the Teichm\"uller metric. $G$ denotes a countable subgroup of $\Isom(X)$ containing a pair of independent hyperbolic isometries\footnote{A pair of independent hyperbolic isometries mean two hyperbolic isometries (pseudo-Anosovs on the surface when $X$ is the Teichm\"uller space) with disjoint set of fixed points on $\partial X$.}. When $X$ is a Gromov hyperbolic space, following Maher and Tiozzo \cite{maher2018random}, we call $G$ a \emph{weakly hyperbolic group}. $\mu$ denotes a non-elementary probability measure on $G$, i.e., a measure on $G$ such that the semigroup generated by the support of $\mu$ contains independent hyperbolic isometries. $\Prob$ stands for the probability measure for random walks induced by $\mu$. See Section \ref{sec:RW} for details.

Random walks on the isometry group of Gromov hyperbolic spaces or Teichm{\"u}ller spaces have been studied in depth for several decades. For example, Kaimanovich investigated random walks on hyperbolic groups and semi-simple Lie groups in \cite{Kaimanovich2000hyp}, identifying the Poisson boundary with other natural boundaries such as the Gromov boundary or the Furstenberg boundary, under some moment and entropy conditions on the measure. See also \cite{ledrappier2001free} for the behavior of random walks on free groups with general conditions. A generalization to relatively hyperbolic groups was considered by Gautero and Math{\'e}us in \cite{gautero2012tree}. This was further generalized to weakly hyperbolic groups by Maher and Tiozzo in \cite{maher2018random}.

In the course of characterizing the Poisson boundary, Maher and Tiozzo observed the following phenomenon.

\begin{thm}[{\cite[Theorem 1.2, 1.4]{maher2018random}}] \label{thm:maherTiozzo}
Let $\omega$ be the random walk generated by $\mu$. Then there exists a constant $L>0$ such that \[
\liminf_{n \rightarrow \infty} \frac{d_{X}(x_{0}, \w_{n}x_{0})}{n} \ge L
\]
for $\Prob$-almost every (a.e. in short) sample path $(\w_{n})$. Moreover, the translation length $\tau(\w_{n})$ of $\w_{n}$ grows at least linearly, i.e., \[
\Prob(\tau(\w_{n}) \le L'n) \rightarrow 0 \quad \mbox{as} \quad n \rightarrow \infty
\]
for some constant $L'>0$.
\end{thm}

Maher and Tiozzo proved these results for weakly hyperbolic groups, which implies the results for Teichm{\"u}ller space thanks to the coarsely Lipschitz systole map from Teichm{\"u}ller space to the curve complex. One can also refer to the earlier work \cite{ledrappier2001free} of Ledrappier that proves similar results with considerations on harmonic measures in the case of free groups.

If $\mu$ further has finite first moment, Kingman's subadditive ergodic theorem provides a constant $\lambda$ such that $\lim_{n \rightarrow \infty} d_{X}(x_{0}, \w_{n}x_{0})/n = \lambda$ for $\Prob$-a.e. $(\w_{n})$. Here $\lambda$ is called the \emph{escape rate} or the \emph{drift} of the random walk. One consequence of Theorem \ref{thm:maherTiozzo} is that the drift of a non-elementary random walk on weakly hyperbolic groups is strictly positive.

As presented in Theorem \ref{thm:maherTiozzo}, due to the lack of subadditivity, translation length is more difficult to investigate than displacement. Hence, while the growth of displacements is given \emph{almost surely}, the growth of translation lengths is given \emph{in probability}. Maher and Tiozzo also proved in \cite{maher2018random} \emph{almost sure linear growth}, i.e., \[
\liminf_{n\rightarrow \infty} \tau(\w_{n})\ge L'n \quad \mbox{for} \quad \Prob\mbox{-a.e.} (\w_n),
\]
when the support of $\mu$ is bounded. This relies on the exponential decay of shadows explained in \cite{maher2012exp}. 

Amongst weakly hyperbolic groups are mapping class groups and $\Out(F_{n})$ that act on the corresponding curve complexes or the complexes of free factors, respectively. We remark that before the work \cite{maher2018random} of Maher and Tiozzo, Maher proved in \cite{maher2011random} that random element of the mapping class group becomes pseudo-Anosov \emph{in probability}. 

Previously known results about the growth of translation lengths rely on moment conditions that require $\mu$ to have bounded support, finite exponential moment or finite second moment. One of our purposes is to study the growth of translation lengths without any moment condition. With the finite first moment condition (so that the drift is available), one can describe this growth with greater precision. Our first main theorem is as follows.

\begin{restatable}{theorem}{mainthm} \label{thm:A}
Let $G$ be a weakly hyperbolic group acting on a Gromov hyperbolic space $X$, and $\mu$ be a non-elementary probability measure on $G$. Then $\Prob$-almost every sample path shows at least linear growth of translation lengths. More precisely, there exists a constant $\mathcal{L}>0$ such that for $\Prob$-a.e. $(\w_{n})$, $\tau(\w_{n}) \ge \mathcal{L}n$ for sufficiently large $n$.

Moreover, if $\mu$ further has finite first moment, then for $\Prob$-a.e. $(\w_{n})$ we have \[
\lim_{n \rightarrow \infty} \frac{\tau(\w_{n})}{n} = \lambda,
\]
where $\lambda$ is the drift of the random walk.
\end{restatable}

As a corollary, \emph{almost every} sample path of non-elementary random walks on mapping class groups or $\Out(F_{n})$ become pseudo-Anosov or fully irreducible, respectively. Here, mapping class groups are acting on the curve complexes and $\Out(F_{n})$ are acting on the complexes of free factors; see \cite{masur1999curve}, \cite{masur2000curve}, \cite{bestvina2014hyperbolicity}. In addition to fully irreducibility, another notion that captures the loxodromic property of an outer automorphism is that of atoroidality. Using the action of $\Out(F_{n})$ on hyperbolic spaces such as Dowdall-Taylor's co-surface graph \cite{dowdall2017the-co-surface} or Brian Mann's intersection graph \cite{mann2014some}, one can even deduce the genericity of fully irreducible atoroidal outer automorphisms in $\Out(F_{n})$.

\begin{cor}[Eventually pseudo-Anosov behavior]
	Let $S$ be a closed orientable surface of genus at least 2 and $\Mod(S)$ be its mapping class group. Let $\mu$ be a non-elementary measure on $\Mod(S)$. Then for $\Prob$-a.e. sample path $(\w_n)$ of the random walk generated by $\mu$, there exists $N > 0$ such that $\w_n$ is pseudo-Anosov for all $n > N$.
\end{cor}

\begin{cor}[Spectral theorem for $\Mod(S)$ on the curve complex]
	Let $S$ be a closed orientable surface of genus at least 2 and $\Mod(S)$ be its mapping class group. Let $\mu$ be a non-elementary measure on $\Mod(S)$ having finite first moment on the curve complex and let $\lambda > 0$ be its drift. Then for $\Prob$-a.e. $(\w_n)$, we have$$\lim_{n \to \infty} {1 \over n} \tau(\w_n) = \lambda.$$
\end{cor}

We record a relevant work \cite{erlandsson2020pA} regarding the genericity of pseudo-Anosov elements in $\Mod(S)$.

Recently, there has been a remarkable progress on the study of the large deviation principle for random walks. In \cite{boulanger2022large}, Boulanger, Mathieu, Sert, and Sisto showed the large deviation property of the displacement and translation length of non-elementary random walks on weakly hyperbolic groups, under finite exponential moment condition. Also investigated are the properties of the deviation function. Their result is related to the large deviation principle of spectral radii of some random matrix products; see \cite{aoun2021law-of-large} for this perspective.

Recently, Gou{\"e}zel proved the exponential error bounds for non-elementary random walks on Gromov hyperbolic spaces without moment condition in \cite{gouezel2022exponential}. His strategy shares the philosophy of pivoting described in the present article, but their aim differs from ours. A similar technique appears in \cite{haissinsky2018renewal}, where the authors aim to deal with surface groups and central limit theorems.


Meanwhile, mapping class groups also act on Teichm{\"u}ller spaces equipped with the Teichm{\"u}ller metric, which are not Gromov hyperbolic in general (\cite{masur1995teichmuller}, \cite{lenzhen2008teichmuller}, \cite{masur1975teich}, \cite{mccarthy1989dynamics}, \cite{ivanov2002short}). Thus, random walks on Teichm{\"u}ller spaces are of independent interest. Kaimanovich and Masur \cite{kaimanovich1996poisson} proposed a way to investigate random walks on Teichm{\"u}ller spaces and related them with their Poisson boundary. Meanwhile, Dahmani and Horbez proved in \cite{dahmani2018spectral} the spectral theorem for mapping class groups with respect to the Teichm{\"u}ller metric under finite second moment condition. Their strategy is to lift the deviation of random paths on the curve complex to Teichm{\"u}ller setting. In \cite{baik2022topological}, the authors applied these techniques to establish the spectral theorems for free subgroups of $\Mod(S)$ generated by two multitwists, under finite second moment condition.


Our next main theorem generalizes the result of Dahmani and Horbez in  \cite{dahmani2018spectral}. 

\begin{restatable}{theorem}{mainthmTeich} \label{thm:B}
Let $S$ be a closed orientable surface of genus at least 2. Let $G=\Mod(S)$ be its mapping class group, $X=\T(S)$ be its Teichm{\"u}ller space equipped with the Teichm{\"u}ller metric, and $\mu$ be a non-elementary probability measure on $\Mod(S)$. Then $\Prob$-a.e. sample path shows at least linear growth of translation lengths. More precisely, there exists a constant $\mathcal{L}>0$ such that for $\Prob$-a.e. $(\w_{n})$, $\tau(\w_{n}) \ge \mathcal{L}n$ for sufficiently large $n$.

Moreover, if $\mu$ further has finite first moment with respect to the Teichm{\"u}ller metric, then for $\Prob$-a.e. $(\w_{n})$ we have \[
\lim_{n \rightarrow \infty} \frac{\tau(\w_{n})}{n} = \lambda
\] where $\lambda>0$ is the drift.
\end{restatable}

We remark that Theorem \ref{thm:B} implies the spectral theorem for Teichm{\"u}ller spaces equipped with Thurston's asymmetric Lipschitz metric. This is due to the result of Choi and Rafi in \cite{choi2007comparison}: among the (marked) surfaces of the same injectivity radius, the distances with respect to the Teichm{\"u}ller metric and the Thurston metric differ by a uniformly bounded amount.

Our methods are influenced by \cite{maher2018random} where Maher and Tiozzo defined the notion of persistent joints that records the permanent depart of the random walk from the origin. We also make use of the boundary convergence of the random walk and the non-atomness of the limiting measure, which are established in \cite{maher2018random}. Another approach to study the deviation of random walks was suggested by \cite{mathieu2020deviation}. These traditional methods require finite second moment condition to deduce a deviation inequality that leads to the summable decay of shadows. A related result is the central limit theorem on hyperbolic groups established by Benoist and Quint in \cite{benoist2016central}.

In contrast, apart from the boundary convergence, pivoting and probability estimation appearing in our methods rely purely on the elementary properties of Gromov products and fellow-traveling geodesics. As long as pivots are present in the sample path, the pivoting method works regardless of deviation from the escape rate. (See \cite{gouezel2022exponential} for a similar idea.) In order to implement this idea on Teichm{\"u}ller spaces, we utilize the fellow-traveling phenomena of certain Teichm{\"u}ller geodesics due to Rafi \cite{rafi2014hyperbolicity}. We remark that Duchin also investigated `thin triangles' in Teichm{\"u}ller spaces in \cite{duchin2005thin} to study the dynamics of random walks on Teichm{\"u}ller spaces.

\subsection*{Acknowledgments}
We truly appreciate \c Ca\u gri Sert for fruitful conversations. We also thank the anonymous referee for valuable comments.

The first and second authors were partially supported by Samsung Science \& Technology Foundation grant No. SSTF-BA1702-01.

%
%

\section{Preliminaries}
\label{section:prelim}

\subsection{Geometry of a Gromov hyperbolic space}

Let $(X, d)$ be a metric space. Throughout, we fix a basepoint $x_{0}$ in $X$. The following notion is crucial to defining Gromov hyperbolicity. 
\begin{definition}[Gromov product]
	For $x, y, z \in X$, the \emph{Gromov product} $(x, y)_{z}$ is defined by $$(x, y)_{z} := {1 \over 2} \left[ d(x, z) + d(y, z) - d(x, y)\right].$$
\end{definition}

\begin{definition}[Gromov hyperbolic space]
A metric space $(X, d)$ is said to be \emph{Gromov hyperbolic} if it satisfies the following property for some $\delta > 0$:

\begin{property}\label{property:Gromov}
	 For any $x, y, z,w\in X$, we have $$(x, y)_{w} \ge \min \{(x, z)_{w}, (y, z)_{w}\} - \delta.$$ 
\end{property}
\end{definition}

For details on the properties of Gromov hyperbolic spaces, see \cite{gromov1987hyperbolic} and \cite{bridson1999metric}. In Sections \ref{section:prelim} and \ref{section:translation}, we assume that $(X, d)$ is a separable, geodesic, Gromov hyperbolic space. 

We now consider the \emph{Gromov boundary} $\partial X$ of $X$ in terms of the Gromov product. We regard that a sequence $(x_n)_{n \in \N}$ in $X$ is converging to a point at infinity if $(x_n, x_m)_{x_0} \to \infty$ as $\min \{m, n\} \to \infty$. Furthermore, two such sequences $(x_n), (y_n)$ will be considered as converging to the same point at infinity if $(x_n, y_n)_{x_0} \to \infty$ as $n \to \infty$. In this case, we regard $(x_n)$ and $(y_n)$ to be equivalent and denote $(x_n) \sim (y_n)$. In this point of view, the \emph{Gromov boundary} is defined as follows. $$\partial X := \left\lbrace (x_n) \in X^{\N} : \lim_{\min\{m, n\} \to \infty}(x_n, x_m)_{x_0} \to \infty \right\rbrace / \sim$$

The Gromov product defined above can be extended to the Gromov boundary by setting $$(x, y)_{x_0} := \sup_{(x_n) = x, (y_n) = y} \liminf_{m, n \to \infty} (x_m, y_n)_{x_0}.$$

One can interpret the Gromov product $(x, y)_{z}$ as a crude distance from $z$ to the geodesic connecting $x$ and $y$. Having this in mind, we have the following useful lemma.

\begin{lem}\label{lem:almostAdditive}
Let $n \ge 1$ and $x_{0}, \ldots, x_{n}$ be points in $X$. Suppose that \begin{equation}\label{eqn:almostAdditiveAssump}
(x_{i-1}, x_{i+1})_{x_{i}} +(x_{i}, x_{i+2})_{x_{i+1}} < d(x_{i}, x_{i+1}) - 3\delta
\end{equation}
for $i = 1, \ldots, n-2$. Then: \begin{enumerate}
\item $|(x_{i}, x_{k})_{x_{j}} - (x_{j-1}, x_{j+1})_{x_{j}}| \le 2\delta$ for $0 \le i < j < k \le n$, and 
\item \[
	\left|  \left(\sum_{i=0}^{n-1} d(x_{i}, x_{i+1}) - 2 \sum_{i=1}^{n-1} (x_{i-1}, x_{i+1})_{x_{i}} \right) - d(x_{0}, x_{n}) \right| \le 2(n-1) \delta.
\]
\end{enumerate}
\end{lem}

\begin{proof}
We prove (1) and (2) by induction on $n$. For $n = 1$, (1) is void and (2) holds automatically.

Let us now assume (1) and (2) for $n =m\ge 1$ and prove them for $n=m+1$. 
We claim that \begin{align}\label{eqn:almostAdditive1}
|(x_{i}, x_{k})_{x_{j}} - (x_{j-1}, x_{k})_{x_{j}}| \le \delta,\\ \label{eqn:almostAdditive2}
|(x_{i}, x_{k})_{x_{j}} - (x_{i}, x_{j+1})_{x_{j}}| \le \delta
\end{align}
for $0 \le i < j < k \le m+1$ concludes (1) for $n=m+1$. Indeed, given a triple $(i, j, k)$, Inequality \ref{eqn:almostAdditive1} and Inequality \ref{eqn:almostAdditive2} (with $i$ replaced with $j-1$) lead to the inequality in (1).

The nontrivial case for Inequality \ref{eqn:almostAdditive1} is when $i < j-1$. Since $0 \le i < j-1 <  j \le m$ in this case, Inequality \ref{eqn:almostAdditive1} for $n=m$ implies \begin{equation}\label{eqn:almostAdditive3}
|(x_{i}, x_{j})_{x_{j-1}} - (x_{j-2}, x_{j})_{x_{j-1}}| \le \delta.
\end{equation}
Moreover, since $1 \le j-1 < j < k \le m+1$, Inequality \ref{eqn:almostAdditive2} for $n=m$ implies \begin{equation}\label{eqn:almostAdditive3.5}
|(x_{j-1}, x_{k})_{x_{j}} - (x_{j-1}, x_{j+1})_{x_{j}}| \le \delta.
\end{equation}
Combining them, we have \begin{equation}\label{eqn:almostAdditive4}\begin{aligned}
(x_{i}, x_{j-1})_{x_{j}} &= d(x_{j-1}, x_{j}) - (x_{i}, x_{j})_{x_{j-1}} &\\
&\ge d(x_{j-1}, x_{j}) - (x_{j-2}, x_{j})_{x_{j-1}} - \delta &\textrm{($\because$ Inequality \ref{eqn:almostAdditive3})} \\
&>(x_{j-1}, x_{j+1})_{x_{j}} +2 \delta &\textrm{($\because$ Inequality \ref{eqn:almostAdditiveAssump})} \\
&\ge (x_{j-1}, x_{k})_{x_{j}} + \delta. & \textrm{($\because$ Inequality \ref{eqn:almostAdditive3.5})}
\end{aligned}\end{equation}
Now Property \ref{property:Gromov} reads \begin{equation}\label{eqn:almostAdditive5}
(x_{i}, x_{k})_{x_{j}} \ge \min \{(x_{i}, x_{j-1})_{x_{j}}, (x_{j-1}, x_{k})_{x_{j}}\} - \delta \ge (x_{j-1}, x_{k})_{x_{j}} - \delta
\end{equation}
and \begin{equation}\label{eqn:almostAdditive6}
(x_{j-1}, x_{k})_{x_{j}} \ge \min\{(x_{j-1}, x_{i})_{x_{j}}, (x_{i}, x_{k})_{x_{j}}\} - \delta.
\end{equation}
If $\min\{(x_{j-1}, x_{i})_{x_{j}}, (x_{i}, x_{k})_{x_{j}}\} = (x_{j-1}, x_{i})_{x_{j}}$ then Inequality \ref{eqn:almostAdditive4} and \ref{eqn:almostAdditive6} imply \[
(x_{i}, x_{j-1})_{x_{j}} > (x_{j-1}, x_{k})_{x_{j}} + \delta \ge (x_{j-1}, x_{i})_{x_{j}} + 1.5\delta - \delta,
\]
a contradiction. Hence, $\min\{(x_{j-1}, x_{i})_{x_{j}}, (x_{i}, x_{j})_{x_{j}}\} = (x_{i}, x_{k})_{x_{j}}$ and Inequality \ref{eqn:almostAdditive6} reads \begin{equation}\label{eqn:almostAdditive7}
(x_{j-1}, x_{k})_{x_{j}} \ge (x_{i}, x_{k})_{x_{j}} - \delta.
\end{equation}
Inequality \ref{eqn:almostAdditive5} and \ref{eqn:almostAdditive7} lead to Inequality \ref{eqn:almostAdditive1}. Inequality \ref{eqn:almostAdditive2} is deduced in a similar manner.

For (2), we have \[\begin{aligned}
	&\left|  \left(\sum_{i=0}^{m} d(x_{i}, x_{i+1}) - 2 \sum_{i=1}^{m} (x_{i-1}, x_{i+1})_{x_{i}} \right) - d(x_{0}, x_{m+1}) \right| \\
	&\le \left|  \left(\sum_{i=0}^{m-1} d(x_{i}, x_{i+1}) - 2 \sum_{i=1}^{m-1} (x_{i-1}, x_{i+1})_{x_{i}} \right) - d(x_{0}, x_{m}) \right|\\
	& \quad + \left| d(x_{0}, x_{m}) + d(x_{m}, x_{m+1}) - d(x_{0}, x_{m+1}) - 2(x_{m-1}, x_{m+1})_{x_{m}}\right| \\
	&\le 2(m-1) \delta + 2 |(x_{0} ,x_{m+1})_{x_{m}} - (x_{m-1}, x_{m+1})_{x_{m}}| \le 2m \delta. \qedhere
	\end{aligned}
\]
\end{proof}

We now define the shadows of a point in terms of the Gromov product.

\begin{definition}[Shadow]
	For $x_0, x \in X$ and $R > 0$, the \emph{shadow} $S_{x_0}(x, R)$ is defined by $$S_{x_0}(x, R) := \{y \in X : (x_0, y)_x \le R\}.$$
\end{definition}

Intuitively, the shadow is a set of points $y \in X$ that the geodesic segment connecting $x_0$ and $y$ is of distance at most $R$ from $x$ up to an additive constant. See Figure \ref{fig:shadow}.


\begin{figure}[h]
	\begin{tikzpicture}[scale=0.7, every node/.style={scale=1}]

	\draw[fill=gray!10] (2.65, 3) .. controls (0, 2) and (0, -2) .. (2.65, -3);
	\draw[draw=white, fill=gray!10, domain=311.455233559:408.544766441] plot ({4*cos(\x)}, {4*sin(\x)});
	
	\draw (0, 0) circle(4);
	
	\filldraw (-2, 0) circle(2pt);
	\draw (-2, 0) node[left] {$x_0$};
	
	\filldraw (2, 0) circle(2pt);
	\draw (2, 0) node[right] {$x$};
	
	\filldraw (2.8, 2) circle(2pt);
	\draw (2.8, 2) node[right] {$y$};

	\draw [domain=282:303] plot ({-4.99 + 14.2442*cos(\x)}, {13.93 + 14.2442*sin(\x)});
	
	\draw[dashed, <->] (1.45, 1.22) -- (1.95, 0.05);
	\draw (1.75, 0.7) node[right] {$\le R$};
	
	\draw (2.5, -1) node[below] {$S_{x_0}(x, R)$};

	\end{tikzpicture}
	\caption{Shadow of $x$ with respect to $x_{0}$.} \label{fig:shadow}
\end{figure}
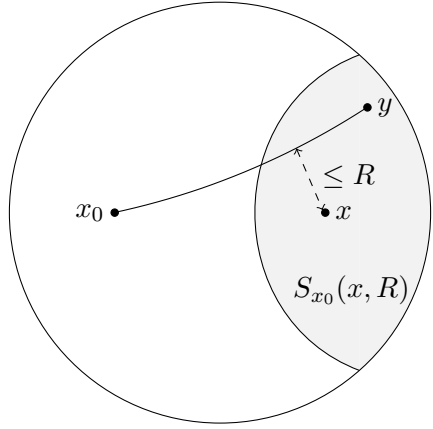

\subsection{Isometries of a Gromov hyperbolic space}

We mainly consider the group $\Isom(X)$ of isometries $X \to X$. Isometries of Gromov hyperbolic spaces are classified into the following categories.

\begin{prop}[Classification of isometries] For $g \in \Isom(X)$, one of the following holds. \begin{enumerate}
	\item $g$ is \emph{elliptic}, i.e., $\{g^nx : n \in \Z\}$ is bounded for any $x \in X$;
	\item $g$ is \emph{parabolic}, i.e., $g$ is not elliptic and has exactly one fixed point in $\partial X$; or
	\item $g$ is \emph{hyperbolic}, i.e., $g$ is not elliptic and has exactly two fixed points in $\partial X$.
\end{enumerate} In particular, when $g$ is hyperbolic, its action shows a source-sink dynamics; one fixed point in $\partial X$ is the attracting point while the other is the repelling point.
\end{prop}

One quantity representing the dynamics of an isometry is its \emph{translation length}.

\begin{definition}[Translation length]
	For $g \in \Isom(X)$, its \emph{translation length} $\tau(g)$ is $$\tau(g) := \liminf_{n \to \infty} {d(x, g^n x) \over n}$$ for any $x \in X$.
\end{definition}

Note that Lemma \ref{lem:almostAdditive} gives the following corollary. (cf. \cite[Proposition 5.8]{maher2018random})

\begin{cor} \label{cor:translationlength}
	For $g \in \Isom(X)$, if $g$ and $x_0 \in X$ satisfy $$d(x_0, gx_0)> 2(g x_0, g^{-1}x_0)_{x_0} + 3\delta$$ then we have $$\left| \tau(g) - \left( d(x_0, g x_0) - 2 (g^{-1}x_0, gx_0)_{x_0} \right) \right| \le 2\delta.$$
\end{cor}

%
%

\subsection{Random walks} \label{sec:RW}

We define a bi-infinite random walk on a group by adopting the convention in \cite{maher2018random}.

For a countable group $G$, let $\mu : G \to [0, 1]$ be a probability measure on $G$. Then the product space $(G^{\Z}, \mu^{\Z})$ forms the step space consisting of bi-infinite step paths. To obtain random walks, we consider the map $G^{\Z} \to G^{\Z}$, $(g_n) \mapsto (\w_n)$ so that $\w_0$ is the identity and $\w_{n-1}^{-1}\w_n = g_n$ for all $n$. In other words, we set $$\w_n = \begin{cases}
g_1 \cdots g_n &, n > 0\\
id &, n = 0\\
g_0^{-1}g_{-1}^{-1} \cdots g_{n+1}^{-1} &, n < 0.
\end{cases}$$ Via this map, $(G^{\Z}, \mu^{\Z})$ induces the probability space $(\Omega, \Prob)$, where $\Omega$ denotes the space of sample paths for random walks.

We often need to estimate the distance or the Gromov product among the translates of $x_{0}$ by isometries. In this situation, bringing a particular point to the basepoint can ease notation and computation. For example, consider isometries $w$, $g$, and $ h$. Then $(wg x_{0}, wh x_{0})_{wx_{0}}$ can also be computed by $(gx_{0}, hx_{0})_{x_{0}}$, which does not depend on $w$. In particular, when $w$ are words at step $n$ and $g$, $h$ are the next steps, then this reduction helps unify the cases.

In this philosophy, we keep the following notations throughout the paper: for a path $\vec{w} = (w_{1}, \ldots, w_{n})$, we write $$x_{n}:=w_{n} x_{0}, \quad x_{n \rightarrow m} := w_{n}^{-1} w_{m} x_{0}.$$
Morally, this amounts to shifting the basepoint to $x_{n}$ and observe the phenomena relative to step $n$. One can readily observe the following equalities:

\begin{enumerate}
\item $x_{0 \rightarrow n} = x_{n} = w_{n}x_{0}$, $x_{n \rightarrow 0} = w_{n}^{-1} x_{0}$, $x_{n \rightarrow n} = x_{0}$.
\item $d(x_{n}, x_{m}) = d(x_{k \rightarrow n}, x_{k \rightarrow m}) = d(x_{0}, x_{n \rightarrow m}) = d(x_{0}, x_{m \rightarrow n})$.
\item $(x_{n}, x_{m})_{x_{k}} = (x_{k \rightarrow n}, x_{k \rightarrow m})_{x_{0}}$.
\end{enumerate}

%
%

\section{Random walks on hyperbolic spaces}
\label{section:translation}

\subsection{Ingredients for persistent joints}\label{subsec:prelim}

In this section, $(X, d)$ is a separable geodesic Gromov hyperbolic space, $G \le \Isom(X)$ is a weakly hyperbolic group, and $\mu : G \to [0, 1]$ is a non-elementary probability measure. As in Section \ref{sec:RW}, $\mu$ induces the probability space $(\Omega, \Prob)$ for random walks.

We recall some facts proved in \cite{maher2018random}. First of all, Maher and Tiozzo proved in \cite[Theorem 1.1]{maher2018random} that almost every sample path $\w_n x_0$ converges to a point $\w_+$ in $\partial X$. This induces the hitting measure $\nu$ on $\partial X$ by \[
\nu(S) := \Prob(\w : \w_{n} o \rightarrow \xi \in S\,\, \textrm{as}\,\, n \rightarrow \infty )
\]
for Borel sets $S \subseteq \partial X$. Maher and Tiozzo also showed that $\nu$ is the unique $\mu$-stationary probability measure on $\partial X$ and is non-atomic. We also borrow the notation \[
Sh(x_{0}, r) := \{S_{x_{0}}(gx_{0}, R) : g \in G, \,d_{X}(x_{0}, gx_{0}) - R \ge r\}.
\]

\begin{prop}[{\cite[Proposition 5.1]{maher2018random}}]\label{prop:MT5.1}
\[
\lim_{r \rightarrow \infty} \sup_{S \in Sh(x_{0}, r)} \nu(\bar{S}) = 0
\]
\end{prop}

For each subset $U$ of $X$, we also define \[
H_{x}^{+}(U) := \Prob(\w_{n}x \in U \mbox{ for some } n \ge 0).
\] Similarly we define $H_x^{-}(U)  := \Prob(\w_{n}x \in U \mbox{ for some } n  \le 0)$. Then the above proposition holds analogously:

\begin{prop}[{\cite[Proposition 5.2]{maher2018random}}]
We have \[
\lim_{r \to \infty} \sup_{S \in Sh(x_{0}, r)} H_{x_{0}}^{\pm}(S) = 0.
\]
\end{prop}
Let $R_{1}>0$ be such that $ \sup_{S \in Sh(x_{0}, R_{1})} H_{x_{0}}^{\pm}(S) < 0.01$.

Meanwhile, since $\mu$ is assumed to be non-elementary, there exist two independent hyperbolic isometries $w_{+}, w_{-}$ in $\llangle \supp \mu \rrangle$, the subsemigroup generated by the support of $\mu$. Their independence implies the following: $\{(w_{1}^{n} x_{0}, w_{2}^{m} x_{0})_{x_{0}}\}_{m, n > 0}$ is bounded, say by $R_{2}>0$, for \[
(w_{1}, w_{2}) \in \{ (w_{+}, w_{+}^{-1}), (w_{-}, w_{-}^{-1}), (w_{+}, w_{-}), (w_{+}, w_{-}^{-1}), (w_{+}^{-1}, w_{-}), (w_{+}^{-1}, w_{-}^{-1})\}.
\]
See Figure \ref{fig:indepiso}. We now fix $R = 1000(R_{1} + R_{2} + \delta+1)$. Note that the bound $R_{2}$ for $\{(w_{1}^{n} x_{0}, w_{2}^{m} x_{0})_{x_{0}}\}_{m, n > 0}$ still works if we replace $w_{+}$ and $w_{-}$ with their positive powers. By taking suitable powers of $w_{\pm}$ if necessary, we can assume that: \begin{enumerate}
\item $w_{+}, w_{-} \in \supp \mu^{L}$ for the same power $L$, i.e., $w_{+} = a_{1} \cdots a_{L}$ and $w_{-} = b_{1} \cdots b_{L}$ for some isometries $a_{i}, b_{i} \in \supp \mu$, and
\item $d(x_{0}, w_{+} x_{0}), d(x_{0}, w_{-}x_{0}) > 100R$.
\end{enumerate}

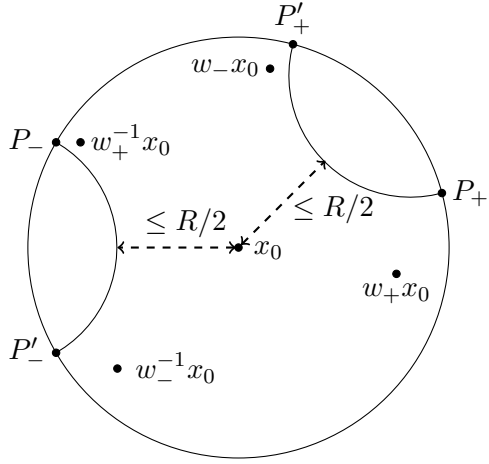
\begin{figure}[h]
	\begin{tikzpicture}[scale=0.7, every node/.style={scale=1}]
	
	\filldraw (0, 0) circle(2pt);
	\draw (0.1, -0.05) node[right] {$x_0$};
	
	\draw (0, 0) circle(4);

	\draw [domain=165:285] plot ({3.265986 + 2.309401*cos(\x)}, {3.265986 + 2.309401*sin(\x)});
	
	\filldraw (1.035276, 3.863703) circle(2pt);
	\draw (1.035276, 3.863703) node[above] {$P_{+}'$};
	
	\filldraw (0.6, 3.4) circle(2pt);
	\draw (0.6, 3.4) node[left] {$w_- x_0$};
	
	\filldraw (3.863703, 1.035276) circle(2pt);
	\draw (3.863703, 1.035276) node[right] {$P_{+}$};
	
	\filldraw (3, -0.5) circle(2pt);
	\draw (3, -0.5) node[below] {$w_+ x_0$};

	\draw [domain=300:420] plot ({-4.618802 + 2.309401*cos(\x)}, {2.309401*sin(\x)});
	
	\filldraw (-3.464101, 2) circle(2pt);
	\draw (-3.464101, 2) node[left] {$P_{-}$};
	
	\filldraw (-3, 2) circle(2pt);
	\draw (-3, 2) node[right] {$w_+^{-1} x_0$};
	
	\filldraw (-3.464101, -2) circle(2pt);
	\draw (-3.464101, -2) node[left] {$P_{-}'$};
	
	\filldraw (-2.3, -2.3) circle(2pt);
	\draw (-2.15, -2.33) node[right] {$w_-^{-1} x_0$};
	
	\draw[dashed, thick, <->] (0.05, 0.05) -- (4 - 2.35, 4 - 2.35);
	\draw (1.8, 1.2) node[below] {$\le R_{2}$};
	
	\draw[dashed, thick, <->] (-0.05, 0) -- (-2.3, 0);
	\draw (-1, 0) node[above] {$\le R_{2}$};
	
	\end{tikzpicture}
	\caption{Choice of $R$ and $w_{\pm}$. Here, $P_+$ and $P_+'$ are attracting points and $P_-$ and $P_-'$ are repelling points of $w$ and $w'$, respectively.} \label{fig:indepiso}
\end{figure}

For convenience, we fix the following notations: \[\begin{aligned}
p_{+} &=  \mu(a_{1}) \cdots \mu(a_{L}) , \\ p_{-} &= \mu(b_{1}) \cdots \mu(b_{L}), \\ 
P &= \max\left(\frac{p_{+}}{p_{+}+ p_{-}}, \frac{p_{-}}{p_{+} + p_{-}}\right). \end{aligned}
\]

\subsection{Persistent joints}
We define a random variable $\chi_{k}(\w)$ that witnesses \emph{persistent joints} at position $3kL$, which is a slight variation of the one defined by Maher and Tiozzo in \cite[Section 5.2]{maher2018random}. See also Figure \ref{fig:joint}.
\begin{definition}[Persistent joint]
	
	For a sample path $\w = (\w_n)$ with the step sequence $(g_n) = (\w_{n-1}^{-1}\w_n)$, we define a random variable $\chi_k(\w)$ as follows.
	
	$\chi_{k}(\w) = 1$ if:
	\begin{enumerate}
		\item \[
		(g_{3(k-1)L + 1}, \ldots, g_{3kL})= \left\{ \begin{array}{c} (b_{1}, \ldots, b_{L}, a_{1}, \ldots, a_{L}, b_{1}, \ldots, b_{L}) \\ \textrm{or} \\  (b_{1}, \ldots, b_{L}, b_{1}, \ldots, b_{L}, b_{1}, \ldots, b_{L})\end{array}\right.,
		\]
		\item $x_{n} \in S_{x_{(3k - 2)L}} (x_{3(k-1)L}, 0.9R)$ for all integers $n \le 3(k-1)L$, and 
		\item $x_{n} \in S_{x_{(3k-1)L} }(x_{3kL}, 0.9R)$ for all integers $n \ge 3kL$
	\end{enumerate}
	where $x_n = \w_n x_0$. Otherwise, $\chi_{k}(\w)=0$.
	
\end{definition}

\begin{figure}[h]
	\begin{tikzpicture}[scale=0.85, every node/.style={scale=1}]

	\draw[fill=gray!10] (2.65, 3) .. controls (1, 2) and (1, -2) .. (2.65, -3);
	\draw[draw=white, fill=gray!10, domain=311.455233559:408.544766441] plot ({4*cos(\x)}, {4*sin(\x)});
	
	\draw[fill=gray!10] (-2.65, 3) .. controls (-1, 2) and (-1, -2) .. (-2.65, -3);
	\draw[draw=white, fill=gray!10, domain=311.455233559:408.544766441] plot ({-4*cos(\x)}, {4*sin(\x)});
	
	\draw (0, 0) circle(4);
	
	\filldraw (-2, 0) circle(2pt);
	\draw (-1.7, 0.3) node[left] {$x_{3(k-1)L}$};
	
	\filldraw (2, 0) circle(2pt);
	\draw (2, -0.3) node[right] {$x_{3kL}$};
	
	\filldraw (-0.5, 0.5) circle(2pt);
	\draw (-0.5, 0.5) node[above] {$x_{(3k-2)L}$};
	
	\filldraw (0.5, -0.5) circle(2pt);
	\draw (0.5, -0.5) node[below] {$x_{(3k-1)L}$};
	
	\draw (-2.5, -1.5) -- (-2.3, -1.3) -- (-2.5, -1) -- (-2.3, -0.7) -- (-2.5, -0.5) -- (-2, -0.5) -- (-2, 0);
	\draw[thick] (-2, 0) -- (-0.5, 0.5) -- (0.5, -0.5) -- (2, 0);
	\draw (2, 0) -- (2.5, 0.5) -- (2.5, 1) -- (3, 0.5) -- (3.5, 1) -- (3, 1.2) -- (3, 1.5);
	\draw (3.2, 1.3) node[above] {$\scalebox{-1}[1]{$\ddots$}$};
	\draw (-2.3, -1.6) node[left] {$\scalebox{-1}[1]{$\ddots$}$};
	
	\draw (-0.2, 0.1) node[right] {$w_{\pm}$};
	\draw (-1, -0.3) node[above] {$w_-$};
	\draw (1.05, -0.4) node[above] {$w_-$};
	
	\draw (2, -4) node[right] {$S_{x_{(3k-1)L}}(x_{3kL}, 0.9R)$};
	\draw (-2, 4) node[left] {$S_{x_{(3k-2)L}}(x_{3(k-1)L}, 0.9R)$};
	
	\draw[->] (4, -3.5) .. controls (4, -3) and (3, -3) .. (3, -2);
	\draw[->] (-4, 3.5) .. controls (-4, 3) and (-3, 3) .. (-3, 2);

	\end{tikzpicture}
	\caption{Description of a persistent joint} \label{fig:joint}
\end{figure}
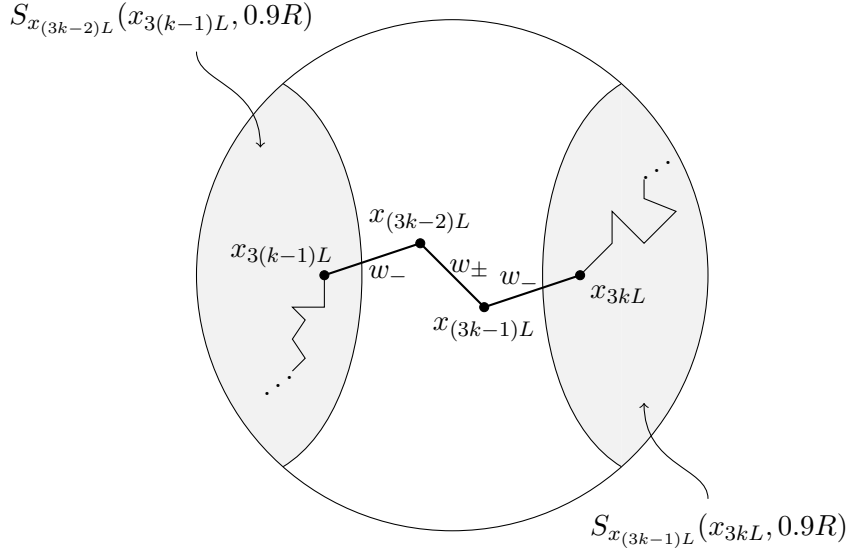

We first observe that $\E(\chi_{1}(\w)) > 0$. The probability for condition (1) is $(p_{+} + p_{-}) p_{-}^{2} \neq 0$. Given (1) as the prior condition, (2) and (3) become independent events. (2) holds if the shifted random walk $T^{3(k-1)L}\w$ does not hit $S_{x_0}(w_-x_0, d(x_{0}, w_{-}x_{0}) - 0.9R)$ in negative time. Indeed, we have 
\[\begin{aligned}
&g_{3(k-1) L}^{-1} \cdots g_{n+1}^{-1} x_{0} \notin S_{x_{0}} (w_{-}x_{0}, d(x_{0}, w_{-} x_{0}) - 0.9R) \\
\Leftrightarrow &\w_{n} x_{0} \notin S_{\w_{3(k-1) L} x_{0}} (\w_{3(k-1)L} w_{-} x_{0}, d(x_{0}, w_{-} x_{0}) -0.9R) \\
\Leftrightarrow & \w_{n} x_{0} \in S_{\w_{(3k-2) L} x_{0}} (\w_{3(k-1)L} x_{0}, 0.9R)
\end{aligned}
\]
for $n < 3(k-1) L$. Here, the last inequality follows from the equality \[
(p, y)_{z} + (p, z)_{y} = d(y, z)
\] for $p \in X$ and $y = \w_{3(k-1)L} x_{0}$, $z = \w_{3(k-1)L} w_{-} x_{0} = \w_{(3k-2)L} x_{0}$. Thus, the probability for condition (2) is at least \[
1-H^{-}_{x_{0}}(S_{x_{0}}(w_{-}x_{0}, d(x_{0}, w_{-}x_{0}) -R)) \ge 1-H^{-}_{x_{0}}(S_{x_{0}}(w_{-}x_{0}, 99R))  \ge 0.99.
\] Similarly, the probability for condition (3) is at least \[
1-H^{+}_{x_{0}} (S_{x_{0}}(w_{-}^{-1} x_{0}, d(x_{0}, w_{-} x_{0}) - R)) \ge 0.99.
\] Overall, we have \[
\eta := \Prob(\chi_{1}(\w) =1) \ge (p_{+} + p_{-}) p_{-}^{2} \cdot (0.99)^{2} > 0.
\]
Note also that $\chi_{k}(\w) = \chi_{1}(T^{3(k-1)L} \w)$. We invoke a variant of Kingman's subadditive ergodic theorem.

\begin{thm}[{\cite[(8.10) Theorem]{woess2000random}}]\label{thm:Woess}
Let $(\Omega, \Prob)$ be a probability space and $U : \Omega \rightarrow \Omega$ be a measure-preserving transformation. If $W_{n}$ is a non-negative real-valued random variable on $\Omega$ satisfying the subadditivity $W_{n+m} \le W_{n} + W_{m} \circ U^{n}$ for all $m, n \in \N$, and $W_{1}$ has finite first moment, then there is a $U$-invariant random variable $W_{\infty}$ such that \[
\lim_{n \rightarrow \infty} \frac{1}{n} W_{n} = W_{\infty}
\]
almost surely and in $L^{1}(\Omega, \Prob)$. If $U$ is ergodic in addition, then $W_{\infty}$ is constant a.e.
\end{thm}

We define $W_{n} = \sum_{k=1}^{n} \chi_{k}(\w)$. Then $W_{n+m} = W_{n} + W_{m} \circ T^{3Ln}$ holds. Since $W_{1}$ is bounded, it has finite first moment. Applying Theorem \ref{thm:Woess}, we get almost everywhere convergence of $\frac{1}{n} W_{n}$ to an a.e. constant $W_{\infty}$. Since $\mathbb{E}(W_{1}) = \eta > 0$, we have $W_{\infty} = \eta$ a.e.

We now consider a modified version of $W_{n}$. Given positive integers $m \le n$, we say that $\mathcal{N} = \{n_{1}< \cdots< n_{k}\}\subseteq 3L\Z$ is an \emph{$(m, n)$-set of pivots} for a finite path $\vec{w} = (w_{1}, \ldots, w_{n})$ if the following hold: \begin{enumerate}
\item $\mathcal{N} \subseteq \{1, \ldots, m\}$;
\item for each $i=1, \ldots, k$, \[
	(g_{n_{i} - 3L + 1}, \ldots, g_{n_{i}})= \left\{ \begin{array}{c} (b_{1}, \ldots, b_{L}, a_{1}, \ldots, a_{L}, b_{1}, \ldots, b_{L}) \\ \textrm{or} \\  (b_{1}, \ldots, b_{L}, b_{1}, \ldots, b_{L}, b_{1}, \ldots, b_{L})\end{array}\right.,
	\]
	\item for each $i=1, \ldots, k$, $x_{j} \in S_{x_{n_{i} - 2L}} (x_{n_{i} - 3L}, 0.9R)$ for  $n_{i-1} - L \le j \le n_{i} - 3L$, and 
	\item for each $i=1, \ldots, k$, $x_{j} \in S_{x_{n_{i} - L}}(x_{n_{i}}, 0.9R)$ for $n_{i} \le j \le n_{i+1}-2L$.
\end{enumerate}
(For convenience, we set $n_{0} = L$ and $n_{k+1} = n+2L$.) Note that if $\mathcal{N}$ and $\mathcal{N}'$ are $(m, n)$-sets of pivots for $\vec{w} = (w_{1}, \ldots, w_{n})$, then so is their union. Thus, we can associate each finite path $\vec{w} = (w_{1}, \ldots, w_{n})$ with its maximal $(m, n)$-set of pivots $\mathcal{N}(\vec{w}) = \mathcal{N}_{m, n}(\vec{w})$. We also define \[
F_{n} :=\left\{\vec{w} = (w_{1}, \ldots, w_{n}) : \# \mathcal{N}_{n, n}(\vec{w}) \ge \frac{\eta n}{6L}+1\right\}.
\]
Note that $\#\mathcal{N}_{n, n}(\w_1, \ldots, \w_n) \ge W_{\lfloor n/3L\rfloor}(\w_1, \ldots, \w_n)$ for each $n$ and $W_{\lfloor n/3L \rfloor} (\w)\ge \eta \lfloor n/3L\rfloor \ge \frac{\eta n}{ 6L}+1$ eventually holds for a.e. sample path $\w$. Consequently, $(\w_{1}, \ldots, \w_{n}) \in F_{n}$ eventually holds for a.e. $\w$.

Let us now fix a finite path $\vec{w} = (w_{1}, \ldots, w_{n})$ with $\mathcal{N}_{n, n}(\vec{w}) = \{n_{1} < \ldots < n_{k}\}$. For convenience, we define the following for $i=1, \ldots, k$: \begin{equation}\label{eqn:Aialphai}\begin{aligned}
A_{i}'(\vec{w})&:= n_i(\vec{w}) - 3L, \\
\alpha'_{i}(\vec{w}) &:= n_i(\vec{w}) - 2L,\\
\beta'_{i}(\vec{w}) &:= n_i(\vec{w}) - L,\\
B'_{i}(\vec{w}) &:= n_{i}(\vec{w}).
\end{aligned}
\end{equation}
We also let $\beta'_{0}(\vec{w}) = B'_{0}(\vec{w}) := 0$ and $\alpha'_{k+1}(\vec{w}) = A'_{k+1}(\vec{w}) := n$. The following lemma allows us to calculate the distances among $x_{j}$.

\begin{lem}\label{lem:wellSevered}
	We have the following:
\begin{enumerate}
\item $d(x_{\beta'_{i-1}}, x_{\alpha'_{i}}) > 99R$ for each $i=1, \ldots, k+1$;
\item $(x_{\alpha'_{i} \rightarrow \beta'_{i-1}}, w_{\pm}x_0)_{x_{0}} < 0.6R$ for each $i=1, \ldots, k$, and
\item $(x_{\beta'_{i} \rightarrow \alpha'_{i+1}}, w_{\pm}^{-1}x_0)_{x_{0}} < 0.6R$ for each $i=1, \ldots, k$.
\end{enumerate}

\end{lem}

\begin{proof}

Let us discuss (1). For $2 \le i \le k+1$, we have\[
(x_{\alpha'_{i}}, x_{B'_{i-1}})_{x_{\beta'_{i-1}}} =  d(x_{\beta'_{i-1}}, x_{B'_{i-1}}) - (x_{\alpha'_{i}}, x_{\beta'_{i-1}})_{x_{B'_{i-1}}} \ge 100R - 0.9R > 99R. 
\]
Since $(x_{\alpha'_{i}}, x_{B'_{i-1}})_{x_{\beta'_{i-1}}} \le d(x_{\alpha'_{i}}, x_{\beta'_{i-1}})$, we deduce the desired conclusion. Similar discussion on $(x_{\beta'_{i-1}}, x_{A'_{i}})_{x_{\alpha'_{i}}}$ for $1 \le i \le k$ handles the remaining case: \begin{equation}\label{eqn:betaAAlpha}
(x_{\beta'_{i-1}}, x_{A'_{i}})_{x_{\alpha'_{i}}} = d(x_{\alpha'_{i}}, x_{A'_{i}}) - (x_{\beta'_{i-1}}, x_{\alpha'_{i}})_{x_{A'_{i}}} \ge 100R - 0.9R > 99R.
\end{equation}

For (2), we invoke Property \ref{property:Gromov}: \[
\min \left\{(w_{-}^{-1} x_{0}, x_{\alpha'_{i} \rightarrow \beta'_{i-1}})_{x_{0}},(x_{\alpha'_{i} \rightarrow \beta'_{i-1}}, w_{\pm} x_{0})_{x_{0}}\right\} - \delta \le (w_{-}^{-1} x_{0}, w_{\pm} x_{0})_{x_{0}} < 0.5R.
\]
Here $(w_{-}^{-1} x_{0}, x_{\alpha'_{i} \rightarrow \beta'_{i-1}})_{x_{0}} = (x_{\alpha'_{i} \rightarrow A'_{i}}, x_{\alpha'_{i} \rightarrow \beta'_{i-1}})_{x_{0}} > 99 R$ was proven in Inequality \ref{eqn:betaAAlpha}. Hence, we deduce $(x_{\alpha'_{i} \rightarrow \beta'_{i-1}}, w_{\pm} x_{0})_{x_{0}} < 0.5R+\delta< 0.6R$. (3) is argued similarly.
\end{proof}

\begin{cor}\label{cor:progress}
Let \[
(y_{2i-1}, y_{2i}) := (x_{\alpha_{i}'}, x_{\beta_{i}'})
\]
for $i = 1, \ldots, k$ and $y_{0} = x_{0}$, $y_{2k+1} = x_{n}$. Then we have \[
(y_{i}, y_{l})_{y_{j}} \le 0.8R,\quad d(y_{i}, y_{j}) \le d(y_{i}, y_{l}) - 95(l-j)R
\] for all $0 \le i \le j \le l\le 2k+1$.
\end{cor}

\begin{proof}
The first item of Lemma \ref{lem:wellSevered} tells us that $d(y_{2i}, y_{2i+1}) > 99R$ for $i = 0, \ldots, k$. Moreover, $d(y_{2i-1}, y_{2i}) = d(x_{0}, w_{\pm} x_{0}) \ge 100R$ for $i = 1, \ldots, k$. Finally, (2) and (3) of Lemma \ref{lem:wellSevered} read that $(y_{i-1}, y_{i+1})_{y_{i}} < 0.6R$ for $i = 1, \ldots, 2k$ and \[
(y_{i-1}, y_{i+1})_{y_{i}} + (y_{i}, y_{i+2})_{y_{i+1}} < 1.2R < 99R - 3 \delta \le d(y_{i}, y_{i+1}) - 3\delta
\]
for $i=1, \ldots, 2k-1$. We can then apply Lemma \ref{lem:almostAdditive} and conclude \[
(y_{i}, y_{l})_{y_{j}} \le (y_{j-1}, y_{j+1})_{y_{j}} + 2\delta \le 0.8R
\]
for $0 \le i < j < l \le 2k+1$. Moreover, this implies \[
d(y_{i}, y_{j}) \le d(y_{i}, y_{j+1}) - d(y_{j}, y_{j+1}) + 2 (y_{i}, y_{j+1})_{y_{j}} \le d(y_{i}, y_{j+1}) - 99R + 1.6 R
\]
for $0 \le i \le j < 2k+1$, which leads to the second conclusion.
\end{proof}

We now define constants $D, M$ and set $G_{n}$ such that\[\begin{aligned}
D &:= R\eta / L, \\
M &> 1+10R + 2d(x_{0}, w_{+} x_{0}) + 2d(x_{0}, w_{-}x_{0})+4L,\\
G_{n} &:= \left\{ \vec{w} \in F_n : \tau(w_{n}) \le \left(2D-\frac{2\eta}{M}\right)n\right\}.
\end{aligned}
\]

Furthermore, for $\vec{w} \in F_{n}$ and $Q > 0$ we define \[
\begin{aligned} \mathcal{N}_{f}(\vec{w}; Q) &= \left\{ n_{i} \in \mathcal{N}(\vec{w}) : d(x_{0}, x_{\beta_{i}'}) \le \frac{1}{2} d(x_{0}, x_{n}) -(D+Q\eta/M)n \right\},\\ 
\mathcal{N}_{b}(\vec{w};Q) &= \left\{ n_{i} \in \mathcal{N}(\vec{w}): d(x_{n}, x_{\alpha_{i}'}) \le \frac{1}{2} d(x_{0}, x_{n}) - (D+Q\eta/M)n\right\},\\
\mathcal{N}_{0}(\vec{w};Q) &= \mathcal{N}(\vec{w})\setminus \left[\mathcal{N}_{f}(\vec{w};Q) \cup \mathcal{N}_{b}(\vec{w};Q)\right].
\end{aligned}
\]
We observe the following estimation of $\mathcal{N}_{f}$ and $\mathcal{N}_{b}$.

\begin{lem} \label{lem:setsize}
For $n>10L/\eta$, if $\vec{w} \in F_{n}$ and $Q \le 1$ then
	$$|\mathcal{N}_{f}(\vec{w}; Q)| \ge \frac{\eta n}{20LM^{2}} \quad \mbox{or} \quad |\mathcal{N}_{b}(\vec{w}; Q)| \ge \frac{\eta n}{20LM^{2}}$$
	holds.
\end{lem}

\begin{proof}
	Suppose not. Then $\mathcal{N}_{0}(\vec{w};Q)$ contains at least $\frac{\eta n}{15L}+1$ indices. By setting $i = 1$ and increasing $j$, Corollary \ref{cor:progress} implies that $d(x_{0}, x_{\beta_{j}'})$ increases as $j$ increases. Similarly, by setting $j = N+1$ and decreasing $i$, we realize that $d(x_{\alpha_{i}}, x_{n})$ increases as $i$ decreases. Consequently, we may take $t$, $t'$ such that \[
	\mathcal{N}_{0}(\vec{w};Q) = \{n_{t}, n_{t+1}, \ldots,  n_{t'}\}.
	\]
	Again, using Corollary \ref{cor:progress}, we deduce that \[
	d(x_{\beta_{t}'}, x_{\alpha_{t'}'}) \ge 90R \cdot \frac{\eta n}{15L} \ge 6 R\eta n/L.
	\]
	Recall that $R >1000>Q$, $M>L$, $D ={R \eta \over L}$ and $\frac{2\eta n}{L}\ge20$. Then Corollary \ref{cor:progress} implies that \[\begin{aligned}
	d(x_{0}, x_{n}) & \ge d(x_{0}, x_{\beta_{t}'}) + d(x_{\beta_{t}'}, x_{\alpha_{t'}'}) + d(x_{\alpha_{t'}'}, x_{n}) - 4.4R\\
	& \ge d(x_{0}, x_{n}) - 2Dn - \frac{2Q\eta n}{M} + \frac{6 R \eta n}{L}  - 4.4R\\
	& \ge d(x_{0}, x_{n}) - \frac{2R \eta n}{L} -\frac{2R \eta n}{L} + \frac{6R \eta n}{L}- 4.4R > d(x_{0}, x_{n}),
	\end{aligned}
	\]wich is a contradiction.
\end{proof}

From this lemma, it follows that \[\begin{aligned}
F_{n, f}(Q) & := \left\{ \vec{w} \in F_{n} : |\mathcal{N}_{f}(\vec{w};Q)| \ge \frac{\eta n}{20 LM^{2}} \right\},\\
F_{n, b}(Q) & := \left\{ \vec{w} \in F_n  : |\mathcal{N}_{b}(\vec{w};Q)| \ge \frac{\eta n}{20LM^{2}} \right\}
\end{aligned}
\]
cover entire $F_{n}$ for $Q\le 1$ and large enough $n$. We also define $$G_{n, f}(Q) := F_{n, f}(Q) \cap G_{n} \quad \mbox{and} \quad G_{n, b}(Q) := F_{n, b}(Q) \cap G_{n}.$$

For each $\vec{w} \in F_{n, f}(Q)$ ($\vec{w} \in F_{n, b}(Q)$, resp.), we fix an integer $N=N(\vec{w})$ between $\frac{\eta n}{50 M^{2}L}$ and $\frac{\eta n}{20 M^{2} L}$. Then we pick \emph{pivot indices} $p_{1}(\vec{w})<\ldots< p_{N}(\vec{w})$ from $\mathcal{N}_{f}$ ($\mathcal{N}_{b}$, resp.). We introduce a notation \[
\begin{aligned}
A_{i}(\vec{w})&:= p_i(\vec{w}) - 3L, \\
\alpha_{i}(\vec{w}) &:= p_i(\vec{w}) - 2L,\\
\beta_{i}(\vec{w}) &:= p_i(\vec{w}) - L,\\
B_{i}(\vec{w}) &:= p_{i}(\vec{w}).
\end{aligned}
\]
For convenience, we also let $\beta_{0}(\vec{w}) = B_{0}(\vec{w}) := 0$ and $\alpha_{N+1}(\vec{w}) = A_{N+1}(\vec{w}) := n$. Recalling Figure \ref{fig:joint}, $A_i$, $\alpha_i$, $\beta_i$, $B_i$ are described as in Figure \ref{fig:AB}

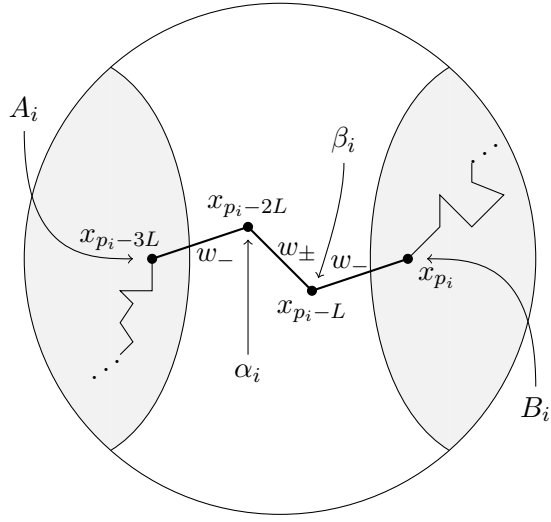
\begin{figure}[h]
	\begin{tikzpicture}[scale=0.85, every node/.style={scale=1}]

	\draw[fill=gray!10] (2.65, 3) .. controls (1, 2) and (1, -2) .. (2.65, -3);
	\draw[draw=white, fill=gray!10, domain=311.455233559:408.544766441] plot ({4*cos(\x)}, {4*sin(\x)});
	
	\draw[fill=gray!10] (-2.65, 3) .. controls (-1, 2) and (-1, -2) .. (-2.65, -3);
	\draw[draw=white, fill=gray!10, domain=311.455233559:408.544766441] plot ({-4*cos(\x)}, {4*sin(\x)});
	
	\draw (0, 0) circle(4);
	
	\filldraw (-2, 0) circle(2pt);
	\draw (-1.7, 0.3) node[left] {$x_{p_i - 3L}$};
	
	\filldraw (2, 0) circle(2pt);
	\draw (2, -0.3) node[right] {$x_{p_i}$};
	
	\filldraw (-0.5, 0.5) circle(2pt);
	\draw (-0.5, 0.5) node[above] {$x_{p_i - 2L}$};
	
	\filldraw (0.5, -0.5) circle(2pt);
	\draw (0.5, -0.5) node[below] {$x_{p_i - L}$};
	
	\draw (-2.5, -1.5) -- (-2.3, -1.3) -- (-2.5, -1) -- (-2.3, -0.7) -- (-2.5, -0.5) -- (-2, -0.5) -- (-2, 0);
	\draw[thick] (-2, 0) -- (-0.5, 0.5) -- (0.5, -0.5) -- (2, 0);
	\draw (2, 0) -- (2.5, 0.5) -- (2.5, 1) -- (3, 0.5) -- (3.5, 1) -- (3, 1.2) -- (3, 1.5);
	\draw (3.2, 1.3) node[above] {$\scalebox{-1}[1]{$\ddots$}$};
	\draw (-2.3, -1.6) node[left] {$\scalebox{-1}[1]{$\ddots$}$};
	
	\draw (-0.2, 0.1) node[right] {$w_{\pm}$};
	\draw (-1, -0.3) node[above] {$w_-$};
	\draw (1.1, -0.4) node[above] {$w_-$};
	
	\draw (-4, 2) node[above] {$A_i$};
	\draw[->] (-4, 2) .. controls (-4, 0) and (-3, 0) .. (-2.3, 0);
	
	\draw (-0.5, -1.5) node[below] {$\alpha_i$};
	\draw[->] (-0.5, -1.5) -- (-0.5, 0.3);
	
	\draw (1, 1.5) node[above] {$\beta_i$};
	\draw[->] (1, 1.5) .. controls (1, 1) and (0.8, 0.3).. (0.6, -0.3);
	
	\draw (4, -2) node[below] {$B_i$};
	\draw[->] (4, -2) .. controls (4, 0) and (3, 0) .. (2.3, 0);

	\end{tikzpicture}
	\caption{Choice of $A_i$, $\alpha_i$, $\beta_i$, $B_i$} \label{fig:AB}
\end{figure}

For each choice of $\sigma \in \{0, 1\}^{N}$, we now define the pivoted word $\vec{w}^{\sigma}$ by declaring the steps $\{g_{i}^{\sigma}\}$ as follows. In plain words, we modify the type of joints that are marked by $\sigma$ only. For $\sigma(i) = 1$ we set \[
(g^{\sigma}_{\alpha_i + 1}, \ldots, g^{\sigma}_{\beta_i}) := \begin{cases}

(a_1, \ldots, a_L) &\mbox{if } (g_{\alpha_i + 1}, \ldots, g_{\beta_i}) = (b_1, \ldots, b_L)\\
(b_1, \ldots, b_L) &\mbox{if } (g_{\alpha_i + 1}, \ldots, g_{\beta_i}) = (a_1, \ldots, a_L)

\end{cases}
\]
Other steps remain unchanged. For $\vec{w}^{\sigma} = (w_i^{\sigma})$, we similarly denote $w_{n}^{\sigma} x_{0}$ by $x_{n}^{\sigma}$ and $(w_{n}^{\sigma})^{-1} w_{m}^{\sigma} x_{0}$ by $x_{n \rightarrow m}^{\sigma}$. Then we observe \[
x_{\alpha_{i} \rightarrow \beta_{i-1}} = x_{\alpha_{i} \rightarrow \beta_{i-1}}^{\sigma}, \quad x_{\beta_{i-1} \rightarrow \alpha_{i}} = x_{\beta_{i-1} \rightarrow \alpha_{i}}^{\sigma}.
\]
By definition, we have \[
x_{\alpha_{i} \rightarrow \beta_{i}}^{\sigma} \in \{w_{+}, w_{-}\},\quad x_{\alpha_{i} \rightarrow \beta_{i}}^{\sigma} = x_{\alpha_{i} \rightarrow \beta_{i}} \quad \textrm{iff}\quad \sigma(i) = 0.
\]
Clearly, $\vec{w}^{\sigma} \neq \vec{w}^{\sigma'}$ for $\sigma \neq \sigma'$. Finally, note that $\mathcal{N}(\vec{w}) = \mathcal{N}(\vec{w}^{\sigma})$ for $ \sigma \in \{0, 1\}^{N}$, while \textit{a priori} $\mathcal{N}_{f}(\vec{w})$ and $\mathcal{N}_{f}(\vec{w}^{\sigma})$ may not coincide.

The proof of Lemma \ref{lem:wellSevered} and Corollary \ref{cor:progress} also apply here, so we omit the proof.

\begin{lem}\label{lem:wellSevered2}
	We have the following.
\begin{enumerate}
\item $d(x_{\beta_{i-1}'}^{\sigma}, x_{\alpha_{i}'}^{\sigma}) > 99R$ for each $1 \le i\le k+1$.
\item $(x_{\alpha_{i}' \rightarrow \beta_{i-1}'}^{\sigma}, w_{\pm}x_0)_{x_{0}} < 0.6R$ for each $i$.
\item $(x_{\beta_{i}' \rightarrow \alpha_{i+1}'}^{\sigma}, w_{\pm}^{-1}x_0)_{x_{0}} < 0.6R$ for each $i$.
\end{enumerate}
\end{lem}

\begin{cor}\label{cor:progress2}
Let \[
(y_{2i-1}^{\sigma}, y_{2i}^{\sigma}) := (x_{\alpha_{i}}^{\sigma}, x_{\beta_{i}}^{\sigma})
\]
for $i = 1, \ldots, k$ and $y_{0} = x_{0}$, $y_{2k+1} = x_{n}$. Then $(y_{i}, y_{l})_{y_{j}} \le 0.8R$ and $d(y_{i}, y_{j}) \le d(y_{i}, y_{l}) - 95(l-j)R$ for all $0 \le i < j < l\le 2N+1$.
\end{cor}

Now we have all ingredients ready.

\begin{lem}\label{claim:step1}
Let $Q \ge 0.9$ and $n \ge 40 M R/\eta$. For each $\vec{w} \in F_{n, f}(Q)$ and $\sigma \neq \kappa \in \{0, 1\}^{N}$, if $\tau(w_{n}^{\kappa}) \le \left(2D -2\eta/M\right)n$, then $\tau(w_{n}^{\sigma}) \ge (2D +\eta/M)n$.
\end{lem}

\begin{proof}
	
We first establish a bound on $d(x_{0}, x_{n}^{\sigma})$. Note that this part does not require any assumption on $\tau(w_{n}^{\kappa})$.

\begin{claim} \label{claim:absolutediff}
	We have
	$$|d(x_0, x_n^{\kappa}) - d(x_0, x_n^{\sigma})| \le {\eta n \over 40M}$$
	and for each $i = 1, \cdots, N$, \[
	\begin{aligned}
	|d(x_n^{\kappa}, x_{\alpha_{i}'}^{\kappa}) - d(x_n^{\sigma}, x_{\alpha_{i}'}^{\sigma})|& \le {\eta n \over 40M}, \\
	|d(x_0, x_{\beta_{i}'}^{\kappa}) - d(x_0, x_{\beta_{i}'}^{\sigma})|& \le {\eta n \over 40M}.
	\end{aligned}
	\]
\end{claim}

\begin{proof}[Proof of the Claim]
	By Lemma \ref{lem:almostAdditive}, we have
\[\begin{aligned}
d(x_{0}, x_{n}^{\sigma}) &\ge \sum_{i=1}^{N+1} d(x_{\beta_{i-1}}^{\sigma}, x_{\alpha_{i}}^{\sigma}) + \sum_{i=1}^{N} d(x_{\alpha_{i}}^{\sigma}, x_{\beta_{i}}^{\sigma}) \\ 
& \quad - 2 \sum_{i=1}^{N} (x_{\beta_{i}}^{\sigma}, x_{\beta_{i-1}}^{\sigma})_{x_{\alpha_{i}}^{\sigma}} - 2 \sum_{i=1}^{N} (x_{\alpha_{i}}^{\sigma}, x_{\alpha_{i+1}}^{\sigma})_{x_{\beta_{i}}^{\sigma}} - 2 \cdot 2N \cdot \delta \\
& \ge \sum_{i=1}^{N+1} d(x_{\beta_{i-1}}^{\sigma}, x_{\alpha_{i}}^{\sigma}) = \sum_{i=1}^{N+1} d(x_{\beta_{i-1}}^{\kappa}, x_{\alpha_{i}}^{\kappa})
\end{aligned}
\]

Recall also that $d(x_{\alpha_{i}}^{\kappa}, x_{\beta_{i}}^{\kappa})$ is either $d(x_0, w_{+} x_0)$ or $d(x_0, w_{-} x_0)$, which are both smaller than $M/2$. 
Since $N$ is chosen to be less than $\frac{\eta n}{20M^2L}$, we have

\[\begin{aligned}
d(x_0, x_n^{\sigma}) &\ge \sum_{i=1}^{N+1} d(x_{\beta_{i-1}}^{\kappa}, x_{\alpha_{i}}^{\kappa}) + \sum_{i=1}^{N} d(x_{\alpha_{i}}^{\kappa}, x_{\beta_{i}}^{\kappa}) - 0.5M \cdot \frac{\eta n}{20 M^{2} L}\\
& \ge d(x_{0}, x_{n}^{\kappa}) - \frac{\eta n}{40ML} \\
&\ge d(x_{0}, x_{n}^{\kappa}) - \frac{\eta n}{40M}.
\end{aligned}
\]
A symmetric argument shows that $d(x_{0}, x_{n}^{\sigma}) \le d(x_{0}, x_{n}^{\kappa}) + \eta n/40M$. Similar arguments involving partial sums give the remaining results.
\end{proof}

Now suppose $\tau(w_n^{\kappa}) \le (2D - 2\eta/M)n$. We observe:

\begin{claim} \label{claim:Dnestimate}
	\begin{equation}\label{eqn:Dnestimate}
	(x_{n}^{\kappa}, x_{n \rightarrow 0}^{\kappa})_{x_{0}} \ge \frac{1}{2} d(x_{0}, x_{n})- Dn.
	\end{equation}
\end{claim}

\begin{proof}[Proof of the Claim]
Recall that $M > L$, $R>1+3\delta$ and $n > 40 M R/\eta$. This implies $1/40 M < R/L$, and consequently,
\begin{equation}\label{eqn:3delta}
3 \delta < R < \frac{\eta n}{40M} <  Dn.
\end{equation}
Now suppose that Inequality \ref{eqn:Dnestimate} does not hold. This implies 
  \[\begin{aligned}
d(x_{0}, x_{n}^{\kappa}) - 2(x_{n}^{\kappa}, x_{n \rightarrow 0}^{\kappa})_{x_{0}}  &\ge \left[d(x_{0}, x_{n}) - \frac{\eta n}{40 M}\right]- 2(x_{n}^{\kappa}, x_{n \rightarrow 0}^{\kappa})_{x_{0}}  & (\textrm{$\because$ Claim \ref{claim:absolutediff}})\\
&> 2Dn - \frac{\eta n}{40 M}  > \frac{\eta n}{40 M}> 3\delta.
\end{aligned}
\]
Then from Corollary \ref{cor:translationlength}, we have \begin{equation}\label{eqn:DnestimateEqn}
\tau(w_{n}^{\kappa}) \ge d(x_{0}, x_{n}^{\kappa}) - 2(x_{n}^{\kappa}, x_{n\rightarrow0}^{\kappa})_{x_0} -2 \delta
\ge \left(2D - \frac{\eta}{40 M} - \frac{\eta}{40 M}\right)n.
\end{equation}
This contradicts the assumption that $\tau(w_n^{\kappa}) \le (2D - 2\eta/M)n$.
\end{proof}

Fixing $\sigma \in \{0, 1\}^{N}$ other than $\kappa$, we now estimate the Gromov products among four points $x_n^{\kappa}, x_n^{\sigma}, x_{n \to 0}^{\kappa},$ and $x_{n \to 0}^{\sigma}$ based at $x_0$. See Figure \ref{fig:pivoting}.

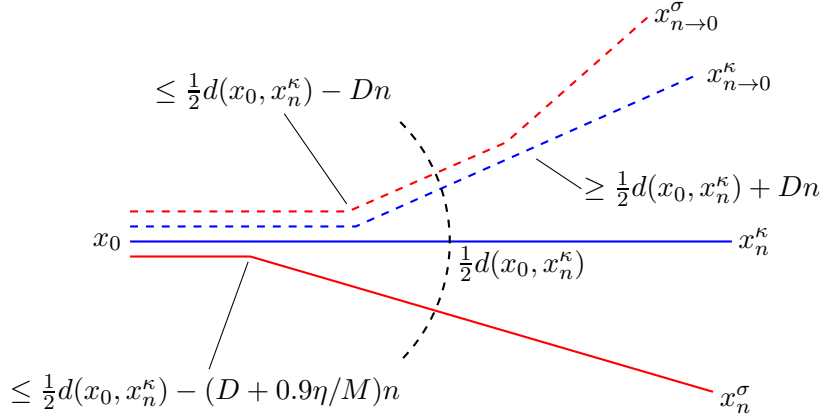
\begin{figure}[h]
	\begin{tikzpicture}
	\draw[blue, thick] (1, 0) -- (9, 0);
	\draw[blue, thick, dashed] (1, 0.2) -- (4, 0.2) -- (8.5, 2.2);
	
	\draw[red, thick] (1, -0.2) -- (2.6, -0.2) -- (8.75, -2);
	\draw[red, thick, dashed] (1, 0.4) -- (3.9, 0.4) -- (6, 1.32) -- (7.9, 3);
	
	\draw[thick, dashed] (5.2-2.2+1.590990257669732, 1.590990257669732) arc (45:-45:2.25);
	
	\draw (9.3, 0) node {$x_{n}^{\kappa}$};
	\draw (9.1, 2.2) node {$x_{n\rightarrow 0}^{\kappa}$};
	\draw (8.4, 3) node {$x_{n \rightarrow 0}^{\sigma}$};
	\draw (9.05, -2.1) node {$x_{n}^{\sigma}$};
	\draw (0.7, 0) node {$x_{0}$};
	\draw (6.2, -0.3) node {$\frac{1}{2} d(x_{0}, x_{n})$};
	\draw (8.6, 0.7) node {$\ge \frac{1}{2} d(x_{0}, x_{n}) + Dn$};
	\draw (2.9, 2) node {$\ge \frac{1}{2} d(x_{0}, x_{n}) - Dn$};
	\draw (2, -2) node {$\le \frac{1}{2} d(x_{0}, x_{n}) - (D + 0.8 \eta/M)n$};
	
	\draw (2.05, -1.7) -- (2.57, -0.27);
	\draw (3.05, 1.7) -- (3.88, 0.5);
	\draw (7, 0.7) -- (6.4, 1.1);

	\end{tikzpicture}
	\caption{4 segments in the pivoting process and the Gromov products.}
	\label{fig:pivoting}
\end{figure}

\begin{claim} \label{claim:halfestimate}
	We have \begin{equation}\label{eqn:halfestimate}(x_{n\rightarrow 0}^{\kappa}, x_{n\rightarrow 0}^{\sigma})_{x_0} \ge \frac{1}{2} d(x_{0}, x_{n}) + Dn.
	\end{equation}
\end{claim}

\begin{proof}[Proof of the Claim]
Since $\vec{w} \in F_{n, f}$, $B_{N}=p_{N}$ belongs to $\mathcal{N}_{f}(\vec{w})$. Note that the steps $(g_{i})_{i}$ of $\vec{w}$ after $\beta_{N}$ are not altered by the pivoting; thus, we have \[
x_{n \rightarrow \beta_{N}}^{\kappa} =x_{n \rightarrow \beta_{N}}^{\sigma} = x_{n \rightarrow \beta_{N}}.
\]
Moreover, Claim \ref{claim:absolutediff} and $p_{N} \in \mathcal{N}_{f}(\vec{w})$ imply\[\begin{aligned}
d(x_{n \rightarrow 0}^{\kappa}, x_{n \rightarrow \beta_{N}}^{\kappa}) &= d(x_{0}^{\kappa}, x_{\beta_{N}}^{\kappa} )\le d(x_{0}, x_{\beta_{N}}) + \frac{\eta n}{40 M} \\
&\le \frac{1}{2} d(x_{0}, x_{n}) - (D + Q \eta / M) n + \frac{\eta n}{40 M}.
\end{aligned}
\]
The same upper bound also applies to $d(x_{n \rightarrow 0}^{\sigma}, x_{n \rightarrow \beta_{N}}^{\sigma})$. Hence, \[
\begin{aligned}
d(x_{n \rightarrow 0}^{\kappa}, x_{n\rightarrow 0}^{\sigma}) &\le d(x_{n \rightarrow 0}^{\kappa}, x_{n \rightarrow \beta_{N}}^{\kappa}) + d(x_{n \rightarrow \beta_{N}}^{\kappa}, x_{n \rightarrow \beta_{N}}^{\sigma}) + d(x_{n \rightarrow \beta_{N}}^{\sigma}, x_{n \rightarrow 0}^{\sigma}) \\
&\le 2 \cdot \left(\frac{1}{2} d(x_{0}, x_{n})  - (D + Q \eta /M) n + \frac{\eta n}{40M}\right).
\end{aligned}
\]
This inequality and Claim \ref{claim:absolutediff} together yield \[ \begin{aligned}
(x_{n \rightarrow 0}^{\kappa}, x_{n \rightarrow 0}^{\sigma})_{x_{0}} &\ge \frac{1}{2} [d(x_{0}, x_{n}^{\kappa}) + d(x_{0}, x_{n}^{\sigma})] - \left(\frac{1}{2} d(x_{0}, x_{n})  - (D + Q \eta /M) n + \frac{\eta n}{40M}\right) \\
&\ge d(x_{0}, x_{n}) - \frac{\eta n}{40 M} - \left(\frac{1}{2} d(x_{0}, x_{n})  - (D + Q \eta /M) n + \frac{\eta n}{40M}\right) \\
&\ge \frac{1}{2} d(x_{0}, x_{n}) + Dn  + \left( \frac{0.9 \eta}{M} - \frac{\eta}{20 M} \right) n. \qedhere
\end{aligned}
\]
\end{proof}

In contrast, we get the following inequality. 

\begin{claim} \label{claim:0.9estimate}
	We have
	$$(x_{n}^{\kappa}, x_{n}^{\sigma})_{x_{0}} \le \frac{1}{2} d(x_{0}, x_{n}) -(D + 0.8\eta/M)n.$$
\end{claim}
\begin{proof}[Proof of the Claim]
To see this, let $l$ be the minimum among $\{1, \ldots, N\}$ such that $\sigma(l) \neq \kappa(l)$. We let: \[
y_{0} = x_{n}^{\kappa}, \quad y_{1} = x_{\beta_{l}}^{\kappa}, \quad y_{2} = x_{\alpha_{l}}^{\kappa} = x_{\alpha_{l}}^{\sigma}, \quad y_{3} = x_{\beta_{l}}^{\sigma}, \quad y_{4} = x_{n}^{\sigma}.
\]
Then Corollary \ref{cor:progress2} tells us that $d(y_{i-1}, y_{i}) > 99R$ for $i = 1, \ldots, 4$ and $(y_{0}, y_{2})_{y_{1}}, (y_{2}, y_{4})_{y_{3}} < 0.8R$. We also have $(y_{1}, y_{3})_{y_{2}} = (w_{+} x_{0}, w_{-}x_{0})_{x_{0}} \le R_{1} <0.5R$. 
Hence, we have\[
(y_{i-1}, y_{i+1})_{y_{i}} + (y_{i}, y_{i+2})_{y_{i+1}} \le 1.3R < 99R - 3\delta \le d(y_{i}, y_{i+1})
\]
for $i=1, 2$. We then apply Lemma \ref{lem:almostAdditive} and deduce $(y_{0}, y_{4})_{y_{2}} < 0.5R + 2\delta < 0.6R$. This implies  \begin{equation}\label{eqn:claim0.9estimate1}\begin{aligned}
2(x^{\kappa}_{n}, x_{n}^{\sigma})_{x_{0}} &= d(x_{0}, y_{0}) + d(x_{0}, y_{4}) - d(y_{0}, y_{4}) \\
&\le \big(d(x_{0}, y_{2}) + d(y_{2}, y_{0}) \big)+ \big(d(x_{0}, y_{2}) + d(y_{2}, y_{4}) \big) \\
& \quad - \big(d(y_{0}, y_{2}) + d(y_{2}, y_{4}) - 2(y_{0}, y_{4})_{y_{2}}\big) \\
&\le 2d(x_{0}, y_{2}) + 2 \cdot 0.6R.
\end{aligned}
\end{equation}

Meanwhile, the distance from $x_{0}$ and $y_{2} = x_{\alpha_{l}}^{\kappa}$ is estimated as follows:\[\begin{aligned}
d(x_{0}, x_{\alpha_{l}}^{\kappa}) &\le d(x_{0}, x_{\beta_{N}}^{\kappa})  & (\textrm{$\because$ Corollary \ref{cor:progress2}})\\
&\le d(x_{0}, x_{\beta_{N}}) + \frac{\eta n}{40 M} & (\textrm{$\because$ Claim \ref{claim:absolutediff}})\\
&\le \frac{1}{2} d(x_{0}, x_{n}) - \left( D + \frac{Q \eta}{M} \right) n + \frac{\eta n}{40M} & \left(p_{N} \in \mathcal{N}_{f}(\vec{w})\right).
\end{aligned}
\]
By plugging this into Inequality \ref{eqn:claim0.9estimate1}, we obtain
 \[\begin{aligned}
(x^{\kappa}_{n}, x_{n}^{\sigma})_{x_{0}}& \le \frac{1}{2} d(x_{0}, x_{n}) - \left( D + \frac{Q \eta}{M} \right)n + \frac{ \eta n}{40 M} + 0.6 R \\
&\le \frac{1}{2}  d(x_{0}, x_{n}) - \left( D + \frac{0.9 \eta}{M} \right) n + \frac{0.05 \eta n}{M} + \frac{0.02n\eta}{M} \\
&\le \frac{1}{2} d(x_{0}, x_{n}) - \left( D + \frac{0.8 \eta}{M} \right) n
\end{aligned}
\]
for $n > 40R M/\eta$ as desired.
\end{proof}

Now let us finish the proof of Lemma \ref{claim:step1}. If we have \begin{equation}\label{eqn:claimStep1Final}
(x_{n}^{\sigma}, x_{n \rightarrow 0}^{\sigma})_{x_{0}} \ge \frac{1}{2} d(x_{0}, x_{n}) - \left(D + \frac{0.7 \eta }{M} \right)n,
\end{equation} then we have \[\begin{aligned}
(x_{n}^{\kappa}, x_{n}^{\sigma})_{x_{0}} &\ge \min \{ (x_{n}^{\kappa}, x_{n \rightarrow 0}^{\kappa})_{x_{0}}, (x_{n\rightarrow 0}^{\kappa}, x_{n}^{\sigma})_{x_{0}} \} - \delta & (\textrm{$\because$ Property \ref{property:Gromov}}) \\
&\ge \min \left\{ \begin{array}{c} (x_{n}^{\kappa}, x_{n \rightarrow 0}^{\kappa})_{x_{0}}, (x_{n \rightarrow 0}^{\kappa}, x_{n \rightarrow 0}^{\sigma})_{x_{0}}, \\(x_{n \rightarrow 0}^{\sigma}, x_{n}^{\sigma})_{x_{0}} \end{array}\right\} - 2\delta & (\textrm{$\because$ Property \ref{property:Gromov}})\\
&\ge \frac{1}{2} d(x_{0}, x_{n}) - \left(D + \frac{0.7 \eta}{M} \right)n - 2\delta & \left(\because \begin{array}{c} \textrm{Claim \ref{claim:Dnestimate},}\\ \textrm{ Claim \ref{claim:halfestimate}} \end{array}\right)\\
&> \frac{1}{2} d(x_{0}, x_{n}) -  \left(D + \frac{0.8 \eta }{M} \right)n . & (\textrm{$\because$ Inequality \ref{eqn:3delta}})
\end{aligned}
\]
This contradicts Claim \ref{claim:0.9estimate}. Therefore, Inequality \ref{eqn:claimStep1Final} does not hold and \[\begin{aligned}
d(x_{0}, x_{n}^{\sigma}) - 2(x_{n}^{\sigma}, x_{n \rightarrow 0}^{\sigma})_{x_{0}} &\ge \left[d(x_{0}, x_{n}) - \frac{\eta n}{40 M}\right] -  2(x_{n}^{\sigma}, x_{n \rightarrow 0}^{\sigma})_{x_{0}}\\
&\ge 2\left( D + \frac{0.7 \eta}{M} \right)n - \frac{\eta n}{40 M} \\
&\ge \left(2D + \eta/M\right) n + 3\delta > 3\delta. & (\textrm{$\because$ Inequality \ref{eqn:3delta}})
\end{aligned}
\]
Then Corollary \ref{cor:translationlength} tells us that $\tau(\w_{n}^{\sigma}) \ge \left(2D + \eta/M\right) n$.
\end{proof}

In particular, for $\vec{w} \in G_{n, f}(Q)$, $\vec{w}^{\sigma} \notin G_{n}$ for any nontrivial $\sigma$. Similar discussion holds for $F_{n, b}$ and $G_{n, b}$.

In the following crucial observation, we finally set the optimal value for $Q$ and $N(\vec{w})$.

\begin{lem}\label{claim:step2}
Let  $Q=1$ and $n \ge 40 RM/\eta$. Suppose that $\vec{w}, \vec{w}' \in G_{n, f}(Q=1)$ and the numbers of pivots $N(\vec{w})$, $N(\vec{w}')$ are $\lfloor \frac{\eta n}{40 M^{2}L}\rfloor$. Then for $\sigma, \sigma' \in \{0, 1\}^{N}$, $\vec{w}^{\sigma} = \vec{w}'^{\sigma'}$ if and only if $\vec{w}= \vec{w}'$ and $\sigma = \sigma'$.
\end{lem}

\begin{proof}

Let $\vec{v} = \vec{w}^{\sigma} = \vec{w}'^{\sigma'}$. The key idea here is to change the roles of $\vec{w}$, $\vec{w}'$ and $\vec{v}$. First note that $$\mathcal{N}_{f}(\vec{w}; Q=1) \subseteq \mathcal{N}(\vec{w}) \subseteq \mathcal{N}(\vec{v}=\vec{w}^{\sigma}).$$ 
Moreover, for each $n_{i} \in \mathscr{N}_{f} (\vec{w}; Q=1)$ we have \[\begin{aligned}
d(x_{0}, x_{\beta_{i}}^{\sigma})&\le d(x_{0}, x_{\beta_{i}}) + \frac{\eta n}{40M} & (\textrm{$\because$ Claim \ref{claim:absolutediff}}) \\
&\le \frac{1}{2} d(x_{0}, x_{n})+ \frac{\eta n}{40M}- \left(D + \frac{\eta}{M}\right)n & \left(\because p_{i} \in \mathcal{N}_{f}(\vec{w}; Q = 1)\right) \\
&\le \frac{1}{2} d(x_{0}, x_{n}^{\sigma}) + \frac{3\eta n}{80M} - \left(D + \frac{\eta}{M}\right)n & (\textrm{$\because$ Claim \ref{claim:absolutediff}})\\
&\le \frac{1}{2} d(x_{0}, x_{n}^{\sigma})- \left(D + \frac{0.9\eta}{M}\right)n.
\end{aligned}
\]
This implies $n_{i} \in \mathcal{N}_{f}(\vec{v}; Q = 0.9)$. It follows that $$\mathcal{N}_{f}(\vec{w}; Q=1) \subseteq \mathcal{N}_{f}(\vec{v}; Q = 0.9)\quad \mbox{and} \quad  \vec{v} \in F_{n, f}(Q=0.9).$$ Similarly, $\mathcal{N}_{f}(\vec{w}'; Q=1) \subseteq \mathcal{N}_{f} (\vec{v}; Q=0.9)$.

Thus, we are able to pick forward pivots $p_{i}(\vec{w})$ and $p'_{i}(\vec{w}')$ of $\vec{w}$ and $\vec{w}'$ altogether for $\vec{v}$. (This will give $N(\vec{v}) \le \frac{\eta n}{20 M^{2}L}$ which is legitimate.) Then Lemma \ref{claim:step1} applied to $\vec{v} \in F_{n, f}(Q=0.9)$ yields a contradiction with $\vec{w}, \vec{w}' \in G_{n}$ unless $\sigma = \sigma'$.
\end{proof}

Similar discussion also holds for $G_{n, b}$.

\subsection{Translation lengths of random isometries}

We now prove the first main theorem:

\mainthm*

\begin{proof}
Let $n \ge 40 RM / \eta$. By Lemma \ref{lem:setsize}, $F_{n, f}(Q=1)$ and $F_{n, b}(Q=1)$ cover entire $F_{n}$. Consequently, $G_{n, f}(Q=1)$ and $G_{n, b}(Q=1)$ cover $G_{n}$.

As in Lemma \ref{claim:step2}, we let $N(\vec{w})=\lfloor \frac{\eta n}{40 M^{2} L} \rfloor$ for each $\vec{w} \in G_{n, f}(Q=1)$. Let $\mathcal{C}(\vec{w}) := \{ \vec{w}^{\sigma} : \sigma \in \{0, 1\}^{N}\}$. Lemma \ref{claim:step2} asserts that $\mathcal{C}(\vec{w})$ and $\mathcal{C}(\vec{w}')$ are disjoint for distinct elements $\vec{w}, \vec{w}' \in G_{n, f}(Q=1)$. Moreover, for each $\vec{w} \in G_{n, f}(Q=1)$, the conditional probability of $\vec{w}$ in $\mathcal{C}(\vec{w})$ is bounded by $P^{N}$. Indeed, elements $\vec{v} = (v_{i})_{i=1}^{n}$ in $\mathcal{C}(\vec{w})$ are determined by $N$ independent choices at pivotal times, with probability $p_{+}$ for $(v_{\alpha_{i} + (j-1)}^{-1} v_{\alpha_{i} + j})_{j=1}^{L} = (a_{1}, \ldots , a_{n})$ and with probability $p_{-}$ for $(v_{\alpha_{i} + (j-1)}^{-1} v_{\alpha_{i} + j})_{j=1}^{L} = (b_{1}, \ldots, b_{n})$. Since $\vec{w}$ corresponds to a single outcome, its probability is at most $P^{N}$.

This implies that $$\begin{aligned}
\Prob((\omega_{i})_{i=1}^{n} \in G_{n, f}(Q=1)) & = \sum_{w \in G_{n, f}(Q=1)} \Prob((\omega_{i})_{i=1}^{n} = \vec{w})\\
& \le P^{N} \sum_{w \in G_{n, f}(Q=1)} \Prob((\omega_{i})_{i=1}^{n} \in \mathcal{C}(\vec{w})) \\
& \le P^{N},
\end{aligned}$$
where the disjointness of $\mathcal{C}(\vec{w})$ was used at the end. Similarly, we have $$\Prob((\omega_{i})_{i=1}^{n} \in G_{n, b}(Q=1)) \le P^N.$$ Then the Borel-Cantelli lemma guarantees that $\mathbb{P}$-a.e. sample path $\w$ avoids $G_{n}$ eventually. Pictorially, we have Figure \ref{fig:avoid}.

\begin{figure}[h]
	\def\c{2.8}
	\begin{tikzpicture}[scale=1, every node/.style={scale=0.8}]

	\draw[thick] (4*\c - 0.5, 3) -- (0, 3) -- (0, -3) -- (4*\c - 0.5, -3);
	\foreach \i in {1, 2, 3}{
		\draw[thick, dashed] (\c*\i, 3.5) -- (\c*\i, -3.5);
	}
	\draw[thick] (0, 1.5) -- (\c, 1.5) -- (\c, 2) -- (2*\c, 2) -- (2*\c, 2.3) -- (3*\c, 2.3) -- (3*\c, 2.5) -- (4*\c - 0.5, 2.5);

	\fill[blue!30] (0, 0.2) -- (\c, 0.2) -- (\c, 1.2) -- (0, 1.2) -- cycle;
	\fill[pattern=north west lines, thick, pattern color = blue!45] (0, -0.8) -- (\c, -0.8) -- (\c, -1.5) -- (0, -1.5) -- cycle;

	\fill[blue!30] (\c, 1) -- (2*\c, 1) -- (2*\c, 1.7) -- (\c, 1.7) -- cycle;
	\fill[pattern=north west lines, thick, pattern color = blue!45] (\c, 0.3) -- (2*\c, 0.3) -- (2*\c, 0.7) -- (\c, 0.7) -- cycle;
	
	\fill[red!22] (\c, -2.8) -- (\c*2, -2.8) -- (\c*2, -2.3) -- (\c, -2.3);
	\fill[pattern=north west lines, thick, pattern color = red!45] (\c, -1.2) -- (2*\c, -1.2) -- (2*\c, -0.9) -- (\c, -0.9);
	
	\fill[black!20!green!35] (\c, -0.7) -- (\c*2, -0.7) -- (\c*2, 0.1) -- (\c, 0.1);
	\fill[pattern=north west lines, thick, pattern color = black!20!green] (\c, -1.4) -- (2*\c, -1.4) -- (2*\c, -2) -- (\c, -2) -- cycle;

	\fill[blue!30] (2*\c, 1.8) -- (3*\c, 1.8) -- (3*\c, 2.15) -- (2*\c, 2.15);
	\fill[pattern=north west lines, thick, pattern color = blue!45] (2*\c, 0.5) -- (3*\c, 0.5) -- (3*\c, 0.8) -- (2*\c, 0.8) -- cycle;
	\fill[pattern=north west lines, thick, pattern color = blue!45] (2*\c, 0.2) -- (3*\c, 0.2) -- (3*\c, 0.35) -- (2*\c, 0.35) -- cycle;
	\fill[pattern=north west lines, thick, pattern color = blue!45] (2*\c, -1.5) -- (3*\c, -1.5) -- (3*\c, -2) -- (2*\c, -2) -- cycle;
	
	\fill[red!22] (2*\c, 1) -- (3*\c, 1) -- (3*\c, 1.2) -- (2*\c, 1.2);
	\fill[pattern=north west lines, thick, pattern color = red!45] (2*\c, 1.5) -- (3*\c, 1.5) -- (3*\c, 1.6) -- (2*\c, 1.6) -- cycle;
	\fill[pattern=north west lines, thick, pattern color = red!45] (2*\c, -0.1) -- (3*\c, -0.1) -- (3*\c, -0.35) -- (2*\c, -0.3) -- cycle;
	\fill[pattern=north west lines, thick, pattern color = red!45] (2*\c, -2.15) -- (3*\c, -2.15) -- (3*\c, -2.3) -- (2*\c, -2.3) -- cycle;
	
	\fill[black!20!green!35] (2*\c, -0.5) -- (3*\c, -0.5) -- (3*\c, -0.9) -- (2*\c, -0.9);
	\fill[pattern=north west lines, thick, pattern color = black!20!green] (2*\c, -1) -- (3*\c, -1) -- (3*\c, -1.1) -- (2*\c, -1.1) -- cycle;
	\fill[pattern=north west lines, thick, pattern color = black!20!green] (2*\c, -2.35) -- (3*\c, -2.35) -- (3*\c, -2.5) -- (2*\c, -2.5) -- cycle;
	\fill[pattern=north west lines, thick, pattern color = black!20!green] (2*\c, -2.8) -- (3*\c, -2.8) -- (3*\c, -2.9) -- (2*\c, -2.9) -- cycle;

	\fill[red!20!blue!20!yellow!35] (2*\c, -1.2) -- (3*\c, -1.2) -- (3*\c, -1.4) -- (2*\c, -1.4);
	\fill[pattern=north west lines, thick, pattern color = black!20!yellow] (2*\c, 1.3) -- (3*\c, 1.3) -- (3*\c, 1.45) -- (2*\c, 1.45) -- cycle;
	\fill[pattern=north west lines, thick, pattern color = black!20!yellow] (2*\c, 0) -- (3*\c, 0) -- (3*\c, 0.15) -- (2*\c, 0.15) -- cycle;
	\fill[pattern=north west lines, thick, pattern color = black!20!yellow] (2*\c, -2.7) -- (3*\c, -2.7) -- (3*\c, -2.55) -- (2*\c, -2.55) -- cycle;
	
	\draw[thick, blue!70, ->] (3*\c + 0.15, 1.85) -- (3*\c + 0.35, 1.85) -- (3*\c + 0.35, 0.65) -- (3*\c + 0.15, 0.65);
	\draw[thick, blue!70, ->] (3*\c + 0.15, 1.95) -- (3*\c + 0.45, 1.95) -- (3*\c + 0.45, 0.275) -- (3*\c + 0.15, 0.275);
	\draw[thick, blue!70, ->] (3*\c + 0.15, 2.05) -- (3*\c + 0.55, 2.05) -- (3*\c + 0.55, -1.75) -- (3*\c + 0.15, -1.75);
	\foreach \i in {1,2, 3}{
		\fill (3*\c + 0.5+0.27*\i, 0) circle (0.04);
	}

	\draw (2.5*\c, 2.2) node[below] {$\vec{w} \in G_{3}$};
	\draw (0.5*\c, 2.25) node {$\exists k \ge 1\, [(\w_{i}) \notin F_{k}]$};
	\draw (1.5*\c, 2.5) node {$\exists k \ge 2 \,[(\w_{i})\notin F_{k}]$};
	\draw (2.5*\c, 2.62) node {$\exists k \ge 3\, [(\w_{i}) \notin F_{k}]$};
	\draw (2.5*\c, -1.75) node {$\vec{w}^{\sigma} \notin G_{3}$};
	\end{tikzpicture}
	\caption{Schematic for $F_{n}$ and $G_{n}$. Each colored region corresponds to $\{\w : (\w_{i})_{i=1}^{k} \in F_{k}$ for $k \ge n$, $(\w_{i})_{i=1}^{n} = \vec{w} \in G_{n}\}$, and is copied inside $\{\w : (\w_{i})_{i=1}^{k} \in F_{k}$ for $k \ge n\}$ by pivoting (hatched regions). Copies of distinct words in $G_{n}$ (colored in distinct colors) are disjoint. The sum of the measures of colored regions is bounded. Note that $\{\w : (\w_{i})_{i=1}^{k} \notin F_{k}$ for some $k \ge n\}$ decreases to a measure zero set.} \label{fig:avoid}
\end{figure}
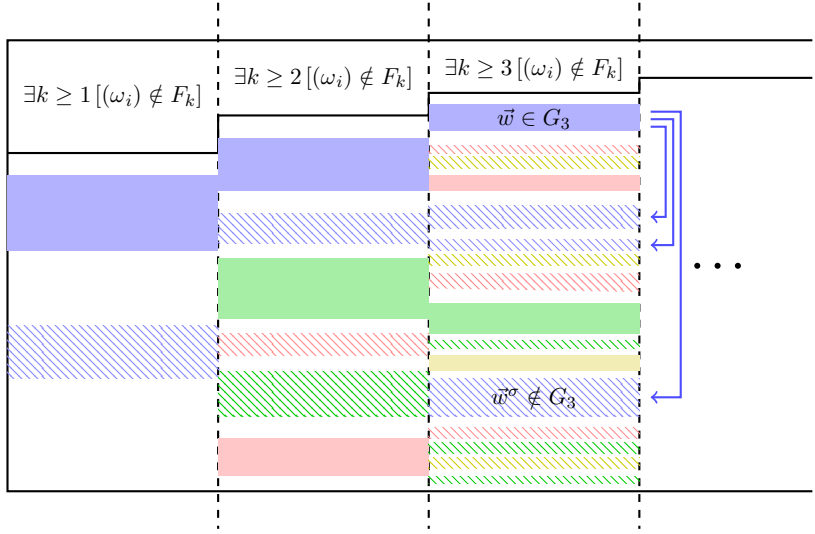

For those paths, if $\tau(\w_n) \le {1.5R \eta n \over L} < (2D - 2\eta/M)n$ happens infinitely often, then $(\w_1, \ldots, \w_{n})$ should avoid $F_n$ infinitely often. However, as mentioned before, almost every sample path avoids $F_{n}$ eventually. Consequently, $\tau(\w_n) \ge {1.5R \eta n \over L}$ eventually holds for $\Prob$-a.e. $(\w_n)$.

Now let us prove the second assertion. When $\mu$ has finite first moment, there exists a constant $0<\lambda < +\infty$, called \emph{drift}, satisfying that $\lambda = \lim_{k \to \infty} {1 \over k}d(x_0, \w_k x_0)$ for $\Prob$-a.e. $(\w_k)$. Set a threshold $0 < \epsilon<1$. We slightly modify the definition of $F_{n}$ and $G_{n}$ as follows: $$\begin{aligned}
F_{n} & := \left\{ \vec{w} : \begin{matrix}
\# \mathcal{N}_{\lfloor \epsilon n/7 \rfloor, n}(\vec{w}) \ge {\eta \epsilon n \over 24L}, \,\,d(x_{0}, w_{n} x_{0}) \ge \left(1-\frac{\epsilon}{1000}\right) \lambda n,\\
d(x_{0}, w_{\lfloor \epsilon n / 7\rfloor} x_{0}) \le \frac{\epsilon \lambda n}{6} \end{matrix} \right\}\\
G_{n} & := \left\{ \vec{w} \in F_{n}: \tau(w_{n}) \le \left(2D - \frac{2\eta}{M} \right) n \right\}.
\end{aligned}$$

We again invoke the subadditive ergodic theorem. Theorem \ref{thm:Woess} asserts that for a.e. path $\w$, there exists some $t(\w)$ such that for all $i > t(\w)$,\begin{enumerate}
\item $W_{i} \ge 0.999\eta i$,
\item $d(x_{0}, \w_{i} x_{0}) \ge ( 1- \epsilon/1000) \lambda i$,
\item $d(x_{0}, \w_{i} x_{0}) \le 1.001 \lambda i$. 
\end{enumerate}
Note also that $\#N_{i, j}(\w) \ge W_{\lfloor i/ 3L \rfloor}$ for any $i \le j$. Considering this, such path $\w$ will be contained in $F_{n}$ for $n \ge 1000 (L+1)t(\w)/\epsilon$. 

This time, we set \[\begin{aligned}
D &:= 0.5(1 - \epsilon/2), \\
M &> 1 + 10R + 2d(x_{0}, w_{+} x_{0}) + 2d(x_{0}, w_{-} x_{0}) + 4L + \frac{2000 \eta}{\lambda \epsilon} + {2000\eta \over \epsilon} + {2000 \over \epsilon}.
\end{aligned}
\]
Note that Inequality \ref{eqn:3delta} still holds. Considering this, Lemma \ref{lem:wellSevered2}, Lemma \ref{claim:step1} and Lemma \ref{claim:step2} also hold with the same proofs. The only missing step is Lemma \ref{lem:setsize}. We now claim that $F_{n, f}(D, Q=1) = F_{n}$. If $\vec{w} \in F_{n}$, then \begin{equation} \label{eqn:revision2}
	\begin{aligned}
 d(x_{0}, w_{\lfloor \epsilon n / 7 \rfloor}x_{0} ) & \le \frac{\epsilon \lambda n}{6} \\
 & \le \frac{1}{2}\left(1-\frac{\epsilon}{1000}\right) \lambda n - \left( \frac{1}{2} (1-\epsilon/2) \lambda + \eta \cdot \frac{\epsilon\lambda}{2000 \eta}\right)n \\
 & \le\frac{1}{2} d(x_{0}, w_{n}x_{0}) -  \left(D + \frac{\eta}{M} \right)n.
 \end{aligned}
\end{equation}
Let us denote $\mathcal{N}_{\lfloor \epsilon n/7 \rfloor, n} (\vec{w})= \{n_{1} < \ldots < n_{k}\}$. We have $k \ge \eta \epsilon n / 24 L$ (since $\vec{w} \in F_{n}$) and $\mathcal{N}_{\lfloor \epsilon n/7 \rfloor, n}(\vec{w}) \subseteq \mathcal{N}_{\lfloor \epsilon n/7 \rfloor, \lfloor \epsilon n/7 \rfloor}(\vec{w})$. Let\[
y_{2i-1} = x_{n_{i} - 2L}, y_{2i} = x_{n_{i} - L}
\]
for $i=1, \ldots, k$ and $y_{2k+1} = x_{\lfloor \epsilon n / 7 \rfloor}$. Then Corollary \ref{cor:progress} tells us that  \[d(x_{0}, x_{n_{i} - L}) \le d(x_{0}, w_{\lfloor \epsilon n / 7 \rfloor}x_{0}) \le \frac{1}{2} d(x_{0}, w_{n}x_{0}) -  \left(D + \frac{\eta}{M} \right)n
\] for each $n_{i} \in \mathcal{N}_{\lfloor \epsilon n/7 \rfloor, n}$. Hence, these joints contribute to $\mathcal{N}_f(\vec{w})$. We then have $\vec{w} \in F_{n, f}$ because \[
\frac{\eta n}{20 L M^{2}} \le \frac{\eta n}{20 L M} \le \frac{\eta n}{20 L} \cdot \frac{\epsilon}{2000}\le \frac{\eta \epsilon n}{24 L}.
\]

Given that the lemmata still holds true, the previous argument to deduce that almost every path $\w$ does not fall into $G_{n}$ infinitely often. Note also that \[
2D - \frac{2 \eta}{M} \ge( 1-\epsilon) \lambda.
\]
Together with the observation that a.e. path does not avoid $F_{n}$ infinitely often, we deduce that $\liminf \frac{1}{n}\tau(\w_{n}) \ge (1-\epsilon) \lambda$ almost surely. By setting $\epsilon = 1/k$ and take intersection of those events for $k \in \Z_{>1}$, we deduce that $\liminf \frac{1}{n} \tau(\w_{n}) \ge \lambda$ almost surely. Since $\frac{1}{n} \tau(\w_{n}) \le \frac{1}{n}d(x_{0}, \w_{n} x_{0})$ and the latter one converges to $\lambda$ a.s., we finally conclude that $\frac{1}{n} \tau(\w_{n}) \rightarrow \lambda$ a.s.
\end{proof}

\section{Teichm\"uller space and its geometry}

\subsection{Basic notions}

In this section, we investigate the geometry of the Teichm{\"u}ller space $(X, d)$ equipped with the Teichm{\"u}ller metric. Let $S$ be a closed orientable surface of genus at least 2. Its \emph{Teichm\"uller space} is the space of equivalence classes $[(f, \Sigma)]$ of an orientation preserving homeomorphism $f : S \to \Sigma$ to a hyperbolic surface $\Sigma$, where $(f, \Sigma)$ and $(g, \Sigma')$ are equivalent if there is an isometry $i : \Sigma \to \Sigma'$ so that $i \circ f$ is homotopic to $g$. The Teichm\"uller space admits a metric $d$ called the \emph{Teichm\"uller metric} defined by $d([(f, \Sigma)], [(g, \Sigma')]) = {1 \over 2} \inf_{\varphi} \log K_{\varphi}$, where the infimum is taken over quasiconformal homeomorphisms $\varphi$ with $f = \varphi \circ g$ up to homotopy and $K_{\varphi}$ is its quasiconformal constant. Then the mapping class group $\Mod(S)$ of the surface acts isometrically on its Teichm\"uller space $X = \mathcal{T}(S)$. We denote $G = \Mod(S)$ for the rest of the paper. For details on the Teichm\"uller geometry, see \cite{hubbard2006teich}, \cite{papadopoulos2007handbook}, \cite{imayoshi1992introduction}.

We recall the notion of the extremal length of a curve.

\begin{definition}[Extremal length]\label{dfn:extLength}
	For a point $x \in X$ in the Teichm\"uller space and an isotopy class $\alpha$ of simple closed curves on the underlying surface $S$, the \emph{extremal length} is defined as $$Ext_x(\alpha):= \sup_{\sigma \in [x]} {l_{\sigma}^2(\alpha) \over \mbox{area}(\sigma)}.$$ Here, $\sigma$ is a Riemannian metric in the conformal class $x$ and $l_{\sigma}(\alpha)$ is the length of $\alpha$ measured by $\sigma$.
\end{definition}

Kerckhoff \cite{kerckhoff1980asymptotic} characterized the Teichm\"uller metric in terms of the extremal length: $$d(x, y) = {1 \over 2} \log \sup_{\alpha \in \mathcal{C}_0(S)} {Ext_y (\alpha) \over Ext_x(\alpha)}$$ where $\mathcal{C}_0(S)$ is a space of isotopy classes of essential simple closed curves on $S$.

We denote the $\epsilon$-thick part of the Teichm\"uller space $X$ by $X_{\ge \epsilon}$. That is, $X_{\ge \epsilon}$ consists of surfaces whose shortest extremal length is at least $\epsilon$. From Kerckhoff's formula, $x \in X_{\ge \epsilon}$ implies $y \in X_{\ge \epsilon'}$ where $\epsilon' = \epsilon e^{-2d(x, y)}$. 

A \emph{geodesic segment} on $X$ refers to an isometric embedding $\Gamma : I=[a,b] \rightarrow X$ of a closed interval $[a, b]$ into $X$, i.e., $d_{X}(\Gamma(t_{1}), \Gamma(t_{2})) =|t_{1}-t_{2}|$ for all $t_{1}, t_{2} \in I$. We make an abuse of notation: the image $\Gamma([a, b])$ of $\Gamma$ is also called the geodesic segment connecting $\Gamma(a)$ and $\Gamma(b)$, and denoted by $[\Gamma(a), \Gamma(b)]$. Note that for each $x, y \in X$, $[x, y]$ uniquely exists by Teichm\"uller's theorem. We say that a segment $\Gamma' : J=[c, d]\rightarrow X$ is a \emph{subsegment} of $\Gamma : I=[a, b] \rightarrow X$ if $\Gamma|_{J} = \Gamma'$. In this situation, we also say that $[\Gamma'(c), \Gamma'(d)]$ is a subsegment of $[\Gamma(a), \Gamma(b)]$. For $\epsilon > 0$, we also call a segment $\epsilon$-thick if it is contained in the $\epsilon$-thick part.

For subsets $A, B \subseteq X$ and $R > 0$, we define \[
\begin{aligned}
N_{R}(A) &:= \{y \in X : \textrm{there exists}\,\, x \in A \,\,\textrm{such that}\,\, d(x, y) < R\},\\
d(A, B) &:= \inf\{ d(x, y) : x \in A,\,\, y \in B\},\mbox{ and}\\
d_{H}(A, B) & := \inf\{R > 0: A \subseteq N_{R}(B)\,\,\textrm{and}\,\, B \subseteq N_{R}(A) \}.
\end{aligned}
\]\

\subsection{Curve complex} \label{subsec:teichandcurvecomplex}

We temporarily digress to the relationship between the Teichm{\"u}ller geometry and the curve complex. The payback for this is Theorem \ref{thm:nonAtom1}, which guarantees the existence of persistent joints.

The \emph{curve complex} $Y = Y(S)$ of a surface $S$ is a simplicial complex whose vertices are isotopy classes of essential simple closed curves and each $(k+1)$-simplex corresponds to $k+1$ vertices represented by disjoint simple closed curves. The curve complex was first introduced by Harvey \cite{harvey1981boundary}, and it endows with a natural metric that each edge has a length $1$.

There exists a projection $\pi : X \rightarrow Y$ from Teichm{\"u}ller space to the curve complex that associates each surface $x \in X$ with the curve $\pi(x)$ with the shortest extremal length. Masur and Minsky \cite{masur1999curve} showed that $\pi$ is coarsely $\Mod(S)$-equivariant and coarsely Lipschitz. That is, there exists a constant $A > 0$ such that $$d_{Y(S)}(g \pi(x), \pi(gx)) \le A, \quad d_{Y(S)}(\pi(x), \pi(y)) \le Ad_{X}(x, y) + A$$ for each $x, y \in X$ and $g \in \Mod(S)$. Furthermore, there exists a constant $K > 0$ such that every geodesic $\gamma$ in the Teichm\"uller space $X$ descends to a $K$-quasi-geodesic $\pi \circ \gamma$ in the curve complex $Y(S)$ up to reparametrization. Finally, since $Y(S)$ is Gromov hyperbolic, the following assertions hold for some $A' > 0$. \begin{enumerate}
\item If $\gamma$, $\eta$ are $K$-quasi-geodesics with the same endpoints, then $d_{H}(\gamma, \eta) < A'$. 
\item Given $x, y, z\in Y(S)$, $d(y, [x, z]) \ge d(x, y) - (y, z)_{x} - A'$.
\end{enumerate}

Thurston \cite{thurston1988classification} compactified the Teichm\"uller space by introducing $\mathcal{PMF}$, the projective space of measured foliations, as the boundary of the Teichm\"uller space. Pseudo-Anosov maps exhibit source-sink dynamics on this compactified Teichm{\"u}ller space as follows. Let $\mathcal{UE} \subseteq \mathcal{PMF}$ the set of uniquely ergodic measured foliations. Then for each pseudo-Anosov $\phi$ there exists distinct foliations $\phi^{+\infty}, \phi^{-\infty} \in \mathcal{UE}$ such that $\lim_{n \rightarrow +\infty} \phi^{ \pm n} x = \phi^{\pm \infty}$ for any $x \in X$.

Moreover, Kaimanovich and Masur \cite{kaimanovich1996poisson} proved the boundary convergence of random walks on the Teichm\"uller space. More precisely, for non-elementary random walks on mapping class groups, almost every orbit on the Teichm\"uller space converges to $\mathcal{UE} \subset \mathcal{PMF}$, the space of uniquely ergodic measured foliations.

Klarreich \cite{klarreich2022boundary} and Hamenst\"adt \cite{hamenstadt2006train} also proved that the boundary of the curve complex can be identified with $\widetilde{\mathcal{MIN}}$, the space of equivalence classes of minimal topological foliations. As such, we can define a map $\pi_{\infty} : \UE\to \widetilde{\mathcal{MIN}}$ which forgets the measure structure and quotients by topological equivalence. This map $\pi_{\infty}$ is injective, and further, if a sequence $x_{n} \in X$ converges to a point $x_{\infty} \in \UE$, then the corresponding $\pi(x_{n})$ converges to a point $\pi_{\infty}(x_{\infty})$.

We now observe that two sequences on Teichm{\"u}ller space diverge from each other if they are heading to distinct uniquely ergodic foliations.

\begin{prop}\label{thm:deviatesEventually}
Let $\phi$ be a pseudo-Anosov mapping class and $\{z_{k}\}_{k \ge 0}$ be a sequence in $X$ that converges to a uniquely ergodic foliation $\xi$ other than $\phi^{+\infty}$. Then for each $C>0$, there exists $F=F(C, \phi, \{z_{k}\}_{k}) >0$ such that $d(\phi^{m} x_{0}, [x_{0}, z_{k}]) > C$ for all $k \ge 0$ and $m \ge F$.
\end{prop}

\begin{proof}
Since $\{z_{k}\}_{k}$ and $\{\phi^{n}x_{0}\}_{n}$ are both heading to points in $\UE$, we can descend them to the curve complex; $\pi(z_{k}) \rightarrow \pi_{\infty}(\xi)$ and $\pi(\phi^{n}x_{0}) \rightarrow \pi_{\infty}(\phi^{+\infty})$. Since $\xi$ and $\phi^{+\infty}$ are distinct, we have
\[
P:= \sup_{k, n} \{(\pi(\phi^{n} x_{0}), \pi(z_{k}))_{\pi(x_{0})}\}< \infty.
\]

Take $F$ such that \[
d(\pi(x_{0}), \pi(\phi^{m}x_{0})) > P+2A' + A(C+1).
\]
for any $m \ge F$. Suppose to the contrary that $d(\phi^{m} x_{0}, [x_{0}, z_{k}]) \le C$ for some $m \ge F$ and $k \ge 0$. Let $x^{\dagger} \in [x_{0}, z_{k}]$ be the closest point to $\phi^{m}x_{0}$.

Note that $[\pi(x_{0}), \pi(z_{k})]$ and $\pi([x_{0}, z_{k}])$ are (unparametrized) $K$-quasi-geodesics on $Y(S)$ with the same endpoints. Hence, we have \[
d_{H}([\pi(x_{0}), \pi(z_{k})], \pi([x_{0}, z_{k}])) < A'.
\] Moreover, since the projection $\pi$ is coarsely Lipschitz, we have \[
d(\pi(x^{\dagger}), \pi(\phi^{m} x_{0})) \le Ad(x^{\dagger}, \phi^{m}x_{0}) + A \le A(C+ 1).
\] Combining these yields \[\begin{aligned}
d([\pi(x_{0}), \pi(z_{k})], \pi(\phi^{m}x_{0})) &\le d_{H}([\pi(x_{0}), \pi(z_{k})], \pi([x_{0}, z_{k}])) + d(\pi(x^{\dagger}), \pi(\phi^{m}x_{0})) \\
&\le A' + A (C+ 1).
\end{aligned}
\]
Meanwhile, we also have \[\begin{aligned}
d([\pi(x_{0}), \pi(z_{k})], \pi(\phi^{m}x_{0})) &\ge d(\pi(x_{0}), \pi(\phi^{m}x_{0})) - (\pi(\phi^{m}x_{0}), \pi(z_{k}))_{\pi(x_{0})} - A' \\
&> (P + 2A' + A (C + 1))- P - A'\\
&= A' + A(C+ 1),
\end{aligned}
\]
a contradiction.
\end{proof}

Proposition \ref{thm:deviatesEventually} leads to the following probabilistic observation. Recall that $\w$ is a non-elementary random walk on the mapping class group.

\begin{thm}\label{thm:nonAtom1}
Let $\phi$ be a pseudo-Anosov mapping class. Then for each $C>0$, there exists a constant $F = F(C, \phi)> 0$ such that \[
\Prob\left( (\w_{n})_{n \in \Z} :  \begin{array}{c} \textrm{$\exists m, n \in \Z$ such that $|m| \ge F$ and} \\ \textrm{$d(\phi^{m} x_{0}, [x_{0}, \w_{n} x_{0}]) <C $}\end{array}\right) < 1.
\]
\end{thm}

\begin{proof}
Suppose not. Then for each $N$, for almost every sample path $\w =(\w_{i})_{i}$ 
there exists $m, n \in \Z$ such that $|m| \ge N$ and $d(\phi^{m} x_{0}, [x_{0}, \w_{n} x_{0}]) <C$. Taking intersection of these events, we deduce that 
\begin{equation} 
\Prob\left( \mathcal{E} := \left\{(\w_{n})_{n \in \Z} : \begin{array}{c} \textrm{$\exists (n_{i}), (m_{i})$ such that $\lim_{i \rightarrow \infty} |m_{i}| = +\infty$ and} \\ \textrm{ $d(\phi^{m_{i}} x_{0}, [x_{0}, \w_{n_{i}}x_{0}]) < C$}\end{array}\right\}\right) =1.
\end{equation}
Meanwhile, recall the result of Kaimanovich and Masur that a.e. path $\w$ have boundary points $\xi^{+}(\w), \xi^{-}(\w) \in \UE$ such that $\lim_{n \rightarrow \pm \infty} \w_{n} x_{0} = \xi^{\pm}(\w)$. Since the forward and the backward stationary measure does not have atoms, we have \[
\Prob\left(\mathcal{E}' := \left\{\w : \{\xi^{+}(\w), \xi^{-}(\w)\} \cap \{\phi^{+\infty}(\w), \phi^{-\infty} (\w)\} = \emptyset\right\}\right) = 1.
\] Finally, Proposition \ref{thm:deviatesEventually} tells us that $\mathcal{E} \cap \mathcal{E}' = \emptyset$. This implies that $\Prob(\mathcal{E} \cup \mathcal{E}') = \Prob(\mathcal{E}) + \Prob(\mathcal{E}') = 2 > 1$, a contradiction.
\end{proof}

\subsection{Fellow traveling}

Let $\gamma : [0, L] \rightarrow X$ and $\gamma' : [0, L'] \rightarrow X$ be paths on $X$. We say that $\gamma$ and $\gamma'$ \emph{$\epsilon$-fellow travel} if $d(\gamma(kL), \gamma'(kL')) < \epsilon$ for each $0 \le k \le 1$. We remark that we always stick to the arclength parametrization when discussing geodesics on $X$. The following is a direct observation.

\begin{lem}\label{lem:fellowCopy}
Let $\gamma^{(i)} : [0, L_{i}] \rightarrow X$ be arcs on $X$ for $i = 1, 2, 3$. Suppose that $\gamma^{(1)}$ and $\gamma^{(2)}$ $\epsilon$-fellow travel, and $\gamma^{(2)}$ and $\gamma^{(3)}$ $\epsilon'$-fellow travel. Then:

\begin{enumerate}
\item  $\gamma^{(1)}|_{[kL_{1}, k'L_{1}]}$ and $\gamma^{(2)}|_{[kL_{2}, k'L_{2}]}$ $\epsilon$-fellow travel for each $0 \le k \le k'\le 1$. In particular, any initial (terminal, resp.) subarc of $\gamma_{1}$ $\epsilon$-fellow travel with an initial (terminal, resp.) subarc of $\gamma_{2}$.
\item $\gamma_{1}$ and $\gamma_{3}$ $(\epsilon+\epsilon')$-fellow travel.
\end{enumerate}
\end{lem}

In Gromov hyperbolic spaces, two geodesics fellow travel if their endpoints are pairwise near. In contrast, such geodesics in Teichm{\"u}ller spaces need not fellow travel since Teichm{\"u}ller spaces are not Gromov hyperbolic in general. Indeed, Rafi \cite{rafi2014hyperbolicity} presented examples of pairs of geodesics whose endpoints are pairwise near while they are not fellow traveling with a uniform constant. Nonetheless, Rafi also proved the fellow traveling phenomenon of geodesics with near endpoints, given that the endpoints are lying on an $\epsilon$-thick part. 

\begin{thm}[{\cite[Theorem 7.1]{rafi2014hyperbolicity}}]\label{thm:Rafi1Ori}
There exists a constant $\mathscr{B}_{0}(\epsilon)$ satisfying the following.  For $x, x', y, y'\in X_{\ge \epsilon}$ such that \[
d(x, x') \le 1 \quad \textrm{and}\quad d(y, y') \le 1,
\]
$[x, y]$ and $[x', y']$ $\mathscr{B}_{0}(\epsilon)$-fellow travel.
\end{thm}

From this, one can also see the following.

\begin{cor}\label{cor:Rafi1}
For each $C >0$, there exists a constant $\mathscr{B}(\epsilon, C) > C$ satisfying the following. For all $x, y \in X_{\ge \epsilon}$ and all $x', y' \in X$ such that \[
d(x, x') \le C \quad \textrm{and}\quad d(y, y') \le C,
\]
$[x, y]$ and $[x', y']$ $\mathscr{B}(\epsilon, C)$-fellow travel.
\end{cor}

\begin{proof}
Let $\{x_{t}\}$ ($\{y_{t}\}$, resp.) be a segment connecting $x$ and $x'$ ($y$ and $y'$, resp.) with speed less than 1 and $x_{0} = x$, $x_{C} = x'$, $y_{0}= y$ and $y_{C} = y'$. Then each of $x_{t}$, $y_{t}$ is $\epsilon e^{-2t}$-thick. Thus, $[x_{t}, y_{t}]$ and $[x_{t+1}, y_{t+1}]$ $\mathscr{B}_{0}(\epsilon e^{-2(t+1)})$-fellow travel. This implies that $[x, y]$ and $[x', y']$ $(\sum_{t=1}^{\lceil C \rceil} \mathscr{B}_{0}(\epsilon e^{-2t}))$-fellow travel.
\end{proof}

Rafi also proved that geodesic triangles in a thick part of the Teichm\"uller space are thin.
 
\begin{thm}[{\cite[Theorem 8.1]{rafi2014hyperbolicity}}]\label{thm:Rafi2}
There exist constants $\mathscr{C}_{0}(\epsilon)$ and $\mathscr{D}_{0}(\epsilon)$ such that the following holds. Let $x, y, z \in X_{\ge \epsilon}$ and suppose that the geodesic $[x, y]$ contains a segment $\gamma \subseteq X_{\ge \epsilon}$ of length at least $\mathscr{C}_{0}(\epsilon)$. Then there exists a point $w\in\gamma$ such that \[
\min \left\{ d(w, [x, z]), d(w, [z,y]) \right\} < \mathscr{D}_{0}(\epsilon).
\]
\end{thm}

From now on, we fix a point $x_0 \in X$ as the basepoint. Since the random walk $\w$ is generated by a non-elementary probability measure $\mu$, there exist two independent pseudo-Anosovs $\phi_{+}$, $\phi_{-}$ in $\llangle \supp \mu \rrangle$, the subsemigroup generated by the support of $\mu$. By taking suitable powers, we may assume that they are made of equal numbers of elements in $\supp \mu $.

Let $\Gamma(\phi_{\pm})$ be the invariant axis of $\phi_{\pm}$ on $X$, respectively. We fix points $y_{+} \in \Gamma(\phi_{+})$ and $y_{-} \in \Gamma(\phi_{-})$. We also let $\tau_{\pm}$ be the translation length of $\phi_{\pm}$ on $X$, respectively.

\begin{lem}\label{lem:thick}
There exists $\mathscr{M}, \epsilon > 0$ such that the following holds. Let $\phi \in \{\phi_{+}, \phi_{-}\}$ and $n \le m$. Then \begin{enumerate}
\item $[\phi^{n} x_{0}, \phi^{m} x_{0}]$ is $\epsilon$-thick, and 
\item $\{\phi^{i} x_{0} : n \le i \le m\}$ is contained in the $\mathscr{M}$-neighborhood of $[\phi^{n} x_{0}, \phi^{m} x_{0}]$.
\end{enumerate}
\end{lem}

\begin{proof}
Let us discuss the case for $\phi = \phi_{+}$. Let $\epsilon_{0} > 0$ be such that $x_{0}, y_{+}$ are $\epsilon_{0}$-thick. 

First, $\{\phi_{+}^{i} x_{0} : n \le i \le m \}$ and $\{\phi_{+}^{i} y_{+} : n \le i \le m\}$ are within Hausdorff distance $d(x_{0}, y_{+})$. Moreover, $\{\phi_{+}^{i} y_{+} : n \le i \le m\}$ periodically appear on the Teichm{\"u}ller geodesic $[\phi_{+}^{n} y_{+}, \phi_{+}^{m} y_{+}]$ with period $\tau_{+}$. Hence, they are within Hausdorff distance $\tau_{+}$, and $[\phi_{+}^{n} y_{+}, \phi_{+}^{m} y_{+}]$ is $\epsilon_{0} e^{-2\tau_{+}}$-thick.

Finally, note that  $\phi_{+}^{n} y_{+}, \phi_{+}^{m} y_{+}$ are $\epsilon_{0}$-thick and $d(\phi_{+}^{n} y_{+}, \phi_{+}^{n} x_{0}), d(\phi_{+}^{m} y_{+}, \phi_{+}^{m} x_{0}) < d(x_{0}, y_{+})$. Corollary \ref{cor:Rafi1} then tells us that the Hausdorff distance between $[\phi_{+}^{n} y_{+}, \phi_{+}^{m} y_{+}]$ and $[\phi_{+}^{n} x_{0}, \phi_{+}^{m} x_{0}]$ is bounded by $\mathscr{B}(\epsilon_{0}, d(x_{0}, y_{+}))$.
\end{proof}

\begin{remark}
We have actually proved something stronger, namely, that the Hausdorff distance between $[\phi^{n} x_{0}, \phi^{m} x_{0}]$ and $\{\phi^{i} x_{0} : n \le i \le m\}$ is bounded by $\mathscr{M}$. We stick to the weaker conclusion as in Lemma \ref{lem:thick} since it still holds true after replacing $\phi_{\pm}$ with their powers.
\end{remark}

\begin{lem}\label{lem:distantAxes}
For each $C> 0$, there exists a constant $\mathscr{G}(C)>0$ such that $d(\phi_{+}^{k} y_{+}, \phi_{-}^{m} y_{-}) \le C$ implies $\max (|k|, |m|) \le \mathscr{G}(C)$. 
\end{lem}

\begin{proof}
Suppose not. Without loss of generality, for each $i \in \Z_{> 0}$, let $k_{i}, m_{i} \in \Z$ be such that $|k_{i}|> i$ and $d(\phi_{+}^{k_{i}} y_{+}, \phi_{-}^{m_{i}} y_{-})\le C$. Note that \[\begin{aligned}
d(x_{0}, \phi_{+}^{k_{i}} y_{+}) &\ge d(y_{+}, \phi_{+}^{k_{i}} y_{+}) - d(y_{+}, x_{0}) \ge |k_i| \tau_{+}- d(y_{+}, x_{0})
\end{aligned}
\]
tends to infinity. Since $d(\phi_{+}^{k_{i}}y_{+}, \phi_{-}^{m_{i}}y_{-}) < C$, $d(x_{0}, \phi_{-}^{m_{i}} y_{-})$ also tends to infinity. Possibly after passing to subsequences, this implies that $\phi_{+}^{k_{i}} y_{+}$ approaches an endpoint of $\Gamma(\phi_{+})$ and $\phi_{-}^{m_{i}} y_{-}$ approaches an endpoint of $\Gamma(\phi_{-})$. Since $d(\phi_{+}^{k_{i}}y_+, \phi_{-}^{m_{i}}y_- )$ is bounded, this implies that those endpoints are identical (\cite[Lemma 1.4.2]{kaimanovich1996poisson}); this contradicts the independence of $\phi_{+}$ and $\phi_{-}$.
\end{proof}

\begin{lem}\label{lem:distantOrbits}
There exists $C_{Grom}>0$ such that the following holds. For any pair of distinct elements $\phi, \psi$ of $\{\phi_{+}, \phi_{+}^{-1}, \phi_{-}, \phi_{-}^{-1}\}$ and $m, n \ge 0$, we have $(\phi^{n} x_{0}, \psi^{m} x_{0})_{x_{0}} < C_{Grom}$.
\end{lem}

\begin{proof}
Let $\mathscr{M}, \epsilon > 0$ be as in Lemma \ref{lem:thick}, and let $\mathscr{C}_{0}(\epsilon), \mathscr{D}_{0}(\epsilon)$ be as in Theorem \ref{thm:Rafi2}. Finally, let $\mathscr{G} = \mathscr{G}(\mathscr{D}_{0}(\epsilon) + 2\mathscr{M})$ be as in Lemma \ref{lem:distantAxes}.

When $\phi = \psi^{-1}$, the conclusion follows from the fact that $\phi_{+}$, $\phi_{-}$ have positive translation lengths. In particular, we have \[
|d(\phi_{\pm}^{m} x_{0}, \phi_{\pm}^{n} x_{0}) - \tau_{\pm} |m-n| | \le 2 d(y_{\pm}, x_{0})
\]
for any integers $m, n \in \Z$. This leads to \[
|d(\phi_{\pm}^{m} x_{0}, \phi_{\pm}^{-n} x_{0}) - d(\phi_{\pm}^{m} x_{0}, x_{0}) - d(\phi_{\pm}^{-n} x_{0}, x_{0})| \le 6 d(y_{\pm}, x_{0})
\]
for $m, n \ge 0$.

Let us now fix $m, n \in \Z$ and estimate $(\phi_{+}^{m} x_{0}, \phi_{-}^{n} x_{0})_{x_{0}}$. If $d(x_{0}, \phi_{+}^{m} x_{0})$ is smaller than $\mathscr{G} \tau_{+} + 2d(x_{0}, y_{+}) + \mathscr{M} + 1 + \mathscr{C}_{0}(\epsilon)$, so is the Gromov product. If not, we take the subsegment $\eta$ of $[x_{0}, \phi_{+}^{m} x_{0}]$ with length $\mathscr{C}_{0}(\epsilon)$ such that $d(x_{0}, \eta) =  \mathscr{G} \tau_{+} + 2d(x_{0}, y_{+}) + \mathscr{M} + 1$. By Theorem \ref{thm:Rafi2}, either $[x_{0}, \phi_{-}^{n} x_{0}]$ or $[\phi_{+}^{m} x_{0}, \phi_{-}^{n} x_{0}]$ intersects the $\mathscr{D}_{0}(\epsilon)$-neighborhood of $\eta$.

 In the former case, let $p \in [x_{0}, \phi_{-}^{n} x_{0}]$ and $q \in \eta$ be such that $d(p, q) \le \mathscr{D}_{0}(\epsilon)$. By Lemma \ref{lem:thick}, there exists $m', n'$ such that $\phi_{+}^{m'} x_{0} \in N_{\mathscr{M}}(p)$ and $\phi_{-}^{n'} x_{0} \in N_{\mathscr{M}}(q)$. Then we have $d(\phi_{+}^{m'} x_{0}, \phi_{-}^{n'} x_{0}) \le \mathscr{D}_0(\epsilon) + 2 \mathscr{M}$ for some $m', n' \in \Z$ such that $d(x_{0}, \phi_{+}^{m'} x_{0}) \ge d(x_{0}, \eta) - \mathscr{M} > \mathscr{G} \tau_{+} + 2d(x_{0}, y_{+})$. This contradicts Lemma \ref{lem:distantAxes}. 

Hence, the latter case holds: there exists $p \in [\phi_{+}^{m} x_{0}, \phi_{-}^{n} x_{0}]$ that belongs to the $\mathscr{D}_{0}(\epsilon)$-neighborhood of $\eta$. This implies that \[\begin{aligned}
2(\phi_{+}^{m} x_{0}, \phi_{-}^{n} x_{0})_{x_{0}} &= [d(x_{0}, \phi_{+}^{m} x_{0}) - d(\phi_{+}^{m} x_{0}, p)] + [d(x_{0}, \phi_{-}^{n} x_{0}) - d(p, \phi_{-}^{n} x_{0})]\\
& \le 2d(x_{0}, p) \le 2d(x_{0}, \eta) + 2\operatorname{diam}(\eta) + 2\mathscr{D}_{0}(\epsilon).
\end{aligned}
\]
In conclusion, we have  $(\phi_{+}^{m} x_{0}, \phi_{-}^{n} x_{0})_{x_{0}} \le \mathscr{G} \tau_{+} + 2d(x_{0}, y_{+}) + \mathscr{C}_{0}(\epsilon) + \mathscr{D}_{0}(\epsilon) + \mathscr{M} + 1$.
\end{proof}

\subsection{Witnessing}

Recall that $G$ stands for the mapping class group of the underlying surface.

\begin{definition}
Let $D>0$, $\gamma$, $\gamma'$ be paths on $X$. We say that $\gamma$ is \emph{$D$-witnessed by $\gamma'$} if there exists a subsegment $\eta$ of $\gamma$ that $D$-fellow travels with $\gamma'$. Here, if $\gamma$ and $\gamma'$ share the beginning point (ending point, resp.), we additionally require $\eta$ to begin at (end at, resp.) that shared point.

For $w, w' \in G$, we say that: \begin{itemize}
\item $w$ is $D$-witnessed by $\gamma'$ if $[x_{0}, w x_{0}]$ is $D$-witnessed by $\gamma'$;
\item $\gamma$ is $D$-witnessed by $w'$ if $\gamma$ is $D$-witnessed by $[x_{0}, w'x_{0}]$, and 
\item $w$ is $D$-witnessed by $w'$ if $[x_{0}, w x_{0}]$ is $D$-witnessed by $[x_{0}, w'x_{0}]$.
\end{itemize}

We also define \[
\begin{aligned}
\mathcal{C}_{D}(\phi \rightarrow 0) &:= \{w \in G: w\,\,\textrm{is $D$-witnessed by}\,\,\phi \},\\
\mathcal{C}_{D}(0 \rightarrow \phi) &:= \{w \in G : w^{-1}\,\textrm{is $D$-witnessed by}\,\,\phi^{-1} \},\\
\mathcal{C}_{D}(\phi \rightarrow \psi) &:= \mathcal{C}_{D}(\phi \rightarrow 0) \cap \mathcal{C}_{D}(0 \rightarrow \psi).
\end{aligned}
\]
\end{definition}

The following lemma will tell us that sample paths are witnessed by $\phi_{\pm}$ for a definite probability.

\begin{lem}\label{lem:singleSegment}
For each $F>0$ there exists $\mathscr{F} = \mathscr{F}(F)$ such that the following holds.

Let $\phi \in \{\phi_{+}, \phi_{+}^{-1}, \phi_{-}, \phi_{-}^{-1}\}$ and $w \in G$. If $d(\phi^{F} x_{0}, [x_{0}, wx_{0}]) > \mathscr{C}_{0}(\epsilon) + \mathscr{D}_{0}(\epsilon) + \mathscr{M}$, then $\phi^{-m} w$ is $\mathscr{F}$-witnessed by $\phi^{-m}$ for every $m \ge 0$.
\end{lem}

\begin{proof}
Let $K_{1}>0$ be such that $d(x_{0}, \phi^{i} x_{0}) > \mathscr{M} + \mathscr{C}_{0}(\epsilon)$ for each $\phi \in \{\phi_{+}, \phi_{+}^{-1}, \phi_{-}, \phi_{-}^{-1}\}$ and $i \ge K_{1}$. Then, let $K_{2}>0$ be such that $d(x_{0}, \phi^{i} x_{0}) \le K_{2}$ for each $\phi \in \{\phi_{+}, \phi_{+}^{-1}, \phi_{-}, \phi_{-}^{-1}\}$ and $0 \le i \le F+K_{1}$. We then take \[
\mathscr{F} = K_{1} + K_{2} + \mathscr{B}(\epsilon, \mathscr{C}_{0}(\epsilon) +\mathscr{D}_{0}(\epsilon)+ \mathscr{M} + K_{2}) +F.
\]
\

Note that for any $x, y, z \in X$, $[x, y]$ is $d(x, z)$-witnessed by $[x, z]$. Since $\mathscr{F} \ge K_{2}$,  $\phi^{-m} w$ is $\mathscr{F}$-witnessed by $\phi^{-m}$ for $0 \le m \le F+K_{1}$.

Let us now consider the case of $m > F+K_{1}$. Let $y$ be a point on $[\phi^{-m} x_{0}, x_{0}]$ that lies in the $\mathscr{M}$-neighborhood of $\phi^{F-m} x_{0}$. Note that $d(y, x_{0}) \ge d(\phi^{m-F} x_{0}, x_{0}) - \mathscr{M} \ge \mathscr{C}_{0}(\epsilon)$. Hence, we can take a subsegment $\eta = [y, y']$ of $[\phi^{-m} x_{0}, x_{0}]$ such that $d(y, y') = \mathscr{C}_{0}(\epsilon)$.

By Theorem \ref{thm:Rafi2}, either $[\phi^{-m}w x_{0}, x_{0}]$ or $[\phi^{-m} w x_{0}, \phi^{-m} x_{0}]$ intersects the $\mathscr{D}_{0}(\epsilon)$-neighborhood of $\eta$. If $d([\phi^{-m} w x_{0}, \phi^{-m} x_{0}], \eta) < \mathscr{D}_{0}(\epsilon)$, then we deduce $d(\phi^{F-m} x_{0}, [\phi^{-m} w x_{0}, \phi^{-m} x_{0}]) < \mathscr{D}_{0}(\epsilon)+ \mathscr{M} + \mathscr{C}_{0}(\epsilon)$. This is equivalent to $d(\phi^{F} x_{0}, [x_{0}, wx_{0}]) < \mathscr{C}_{0}(\epsilon) + \mathscr{D}_{0}(\epsilon) + \mathscr{M}$, which contradicts the assumption.

Hence, $[\phi^{-m} w x_{0}, x_{0}]$ contains a point $p$ that lies in the $\mathscr{D}_{0}(\epsilon)$-neighborhood of $\eta$. Note that \[
\begin{aligned}
d(p, \phi^{-m} x_{0}) & \le d(p, \eta) + \operatorname{diam}(\eta) + \mathscr{M} + d(\phi^{F-m} x_{0}, \phi^{-m} x_{0}) \\
& \le \mathscr{C}_{0}(\epsilon) + \mathscr{D}_{0}(\epsilon) + \mathscr{M} + K_{2}.
\end{aligned}
\]

Then $[p, x_{0}]$ and $[\phi^{-1}x_{0}, x_{0}]$ $\mathscr{B}(\epsilon, \mathscr{C}_{0}(\epsilon) + \mathscr{D}_{0}(\epsilon) + \mathscr{M} + K_{2})$-fellow travel as desired.
\end{proof}

\begin{cor}\label{cor:eventualWitness}
There exists $\mathscr{F} > 0$ such that for $\phi \in \{\phi_{+}, \phi_{-}\}$, we have\[
\Prob\left( (\w_{n})_{n \in \Z} : \textrm{$\phi^{m}\w_{n}$ is $\mathscr{F}$-witnessed by $\phi^{m}$ for all $m, n \in \Z$} \right) > 0.
\]
\end{cor}

\begin{proof}
Let $F = F(\mathscr{C}_{0}(\epsilon) + \mathscr{D}_{0}(\epsilon) + \mathscr{M}, \phi)$ be as in Theorem \ref{thm:nonAtom1} and $\mathscr{F} = \mathscr{F}(F)$ be as in Lemma \ref{lem:singleSegment}. We consider the event $\mathcal{E}$ of sample paths $\w = (\w_{n})_{n}$ such that \[
d(\phi^{F} x_{0}, [x_{0}, \w_{n} x_{0}]), d(\phi^{-F} x_{0}, [x_{0},  \w_{n} x_{0}]) > \mathscr{C}_{0}(\epsilon) + \mathscr{D}_{0}(\epsilon) + \mathscr{M}
\]for all $n \in \Z$. Then $\Prob(\mathcal{E}) >0$ thanks to Theorem \ref{thm:nonAtom1}. Moreover, Lemma \ref{lem:singleSegment} immplies that for $\w \in \mathcal{E}$, $\phi^{m} \w_{n}$ is $\mathscr{F}$-witnessed by $\phi^{m}$ for each $m \in \Z$ as desired.
\end{proof}

We now fix some constants. As the base case, we set \[
\begin{aligned}
\mathscr{D}_{1} &:= \mathscr{D}_{0}(\epsilon) + \mathscr{B}(\epsilon, \mathscr{C}_{0}(\epsilon) + \mathscr{D}_{0}(\epsilon) + \mathscr{M}) + \mathscr{F}, \\
\epsilon_{1} &:= \epsilon e^{-2\mathscr{D}_{1}}.
\end{aligned}
\] Now given $\mathscr{D}_{j}$ and $\epsilon_{j}$, we define 
\begin{equation}\label{eqn:DEDfn}\begin{aligned}
\mathscr{D}_{j+1} & := \mathscr{D}_{j} + \mathscr{B}(\epsilon, 4\mathscr{D}_{j} + 2C_{Grom}  + 5\mathscr{C}_{0}(\epsilon_{j}) + 5\mathscr{D}_{0}(\epsilon_{j}) + 2), \\
\epsilon_{j+1} &:= \epsilon e^{-2\mathscr{D}_{j+1}}.
\end{aligned}
\end{equation}
Note that the constants $\epsilon, \mathscr{M}, \mathscr{F}, C_{Grom}$ remain the same even if we replace $\phi_{+}$, $\phi_{-}$ with their powers. Hence, by employing sufficiently large powers of $\phi_{\pm}$ if necessary, we may assume
\begin{equation}\label{eqn:largeJump}
d(x_{0}, \phi_{\pm} x_{0}) \ge 6\mathscr{D}_{j} + 2C_{Grom}  + 2\mathscr{C}_{0}(\epsilon_{j}) + 2\mathscr{D}_{0}(\epsilon_{j}) + 2
\end{equation}
for $j = 1, \ldots, 6$. Finally, we set \begin{equation}\label{eqn:ZDfn}
Z = \mathscr{C}_{0}(\epsilon_{3}) +\mathscr{D}_{0}(\epsilon_{3}) +2[\mathscr{D}_{6} + d(x_{0}, \phi_{+} x_{0}) +d(x_{0}, \phi_{-}x_{0})] + 1.
\end{equation}

\begin{lem}\label{lem:selfWitness}
For any $\phi \in \{\phi_{+}, \phi_{+}^{-1}, \phi_{-}, \phi_{-}^{-1}\}$ and $n>0$, $\phi^{n}$ belongs to $\mathcal{C}_{\mathscr{D}_{1}}(\phi \rightarrow \phi)$. 
\end{lem}

\begin{proof}
By Lemma \ref{lem:thick}, we have $p \in [x_{0}, \phi^{n} x_{0}]$ such that \[
d(p, \phi x_{0}) < \mathscr{M}< \mathscr{C}_{0}(\epsilon) + \mathscr{D}_{0}(\epsilon) + \mathscr{M}.
\] Since $x_{0}$ and $\phi x_{0}$ are $\epsilon$-thick,  Corollary \ref{cor:Rafi1} tells us that $[x_{0}, p]$ and $[x_{0}, \phi x_{0}]$ $\mathscr{B}(\epsilon, \mathscr{C}_{0}(\epsilon) + \mathscr{D}_{0}(\epsilon) + \mathscr{M})$-fellow travel. By a similar reason, there exists $q \in [x_{0}, \phi^{n} x_{0}]$ such that $[q, \phi^{n} x_{0}]$ and $[\phi^{n-1} x_{0}, \phi^{n} x_{0}]$ $\mathscr{B}(\epsilon, \mathscr{C}_{0}(\epsilon) + \mathscr{D}_{0}(\epsilon) + \mathscr{M})$-fellow travel.
\end{proof}

\begin{remark}
We will observe that if $\phi, \psi \in \{\phi_{+}, \phi_{+}^{-1}, \phi_{-}, \phi_{-}^{-1}\}$ are distinct, then $\mathcal{C}_{\mathscr{D}_{5}}(\phi \rightarrow 0) \cap \mathcal{C}_{\mathscr{D}_{5}}(\psi \rightarrow 0) = \emptyset$.
\end{remark}

We can now discuss the concatenation of witnessed mapping classes.

\begin{definition}
Let $\phi_{i}, \psi_i \in \{\phi_{+}, \phi_{+}^{-1}, \phi_{-}, \phi_{-}^{-1}\}$. We say that the sequences $(\phi_{i})_{i}$ and $(\psi_{i})_{i}$ are \emph{repulsive} if $\phi_{i+1} \neq \psi_{i}^{-1}$ for each $i$.

Given repulsive sequences $(\phi_{i})_{i}$ and $(\psi_{i})_{i}$, we say that a sequence $(g_{i})_{i=1}^{n} \subseteq G$ is \emph{$D$-marked with} $(\phi_{i})_{i=2}^{n}$ and $(\psi_{i})_{i=1}^{n-1}$ if \begin{enumerate}
\item $g_{1} \in \mathcal{C}_{D}(0 \rightarrow \psi_{1})$,
\item $g_{i} \in \mathcal{C}_{D}(\phi_{i} \rightarrow \psi_{i})$ for $i=2, \ldots, n-1$,
\item $g_{n} \in \mathcal{C}_{D}(\phi_{n} \rightarrow 0)$.
\end{enumerate}

If $(g_{i})$ additionally satisfies that $g_{1} \in \mathcal{C}_{D}(\phi_{1} \rightarrow \psi_{1})$, then we say that $(g_{i})_{i=1}^{n}$ is \emph{$D$-strongly marked with} $(\phi_{i})_{i=1}^{n}$ and $(\psi_{i})_{i=1}^{n-1}$.
\end{definition}

\begin{lem}\label{lem:singleDirection}
For each $j = 1, \ldots, 5$, we have the following.

Let $(\phi_{i})_{i}$, $(\psi_{i})_{i}$ be repulsive sequences and $(g_{i})_{i=1}^{n}$ be a sequence that is $\mathscr{D}_{j}$-marked with $(\phi_{i})_{i=2}^{n}, (\psi_{i})_{i=1}^{n-1}$. Let also $w: = g_{1} \cdots g_{n}$. Then $w$ is $\mathscr{D}_{j+1}$-witnessed by $[g_{1} \psi_{1}^{-1} x_{0}, g_{1}x_{0}]$. If, moreover, $(g_{i})_{i}$ is $\mathscr{D}_{j}$-strongly marked with $(\phi_{i})_{i=1}^{n}$, $(\psi_{i})_{i=1}^{n-1}$, then $w$ is $\mathscr{D}_{j+1}$-witnessed by $\phi_{1}$.
\end{lem}

\begin{proof}
Let $C_{Grom}$ be the constant obtained from Lemma \ref{lem:distantOrbits}, and let $K = \mathscr{C}_{0}(\epsilon_{j}) + \mathscr{D}_{0}(\epsilon_{j}) + 2\mathscr{D}_{j} + C_{Grom} + 1$. Recall that we have assumed in Condition \ref{eqn:largeJump} that $d(x_{0}, \phi^{\pm} x_{0}) \ge 2K + 2\mathscr{D}_{j+1}$.

We induct on the number of segments. For $n=1$, the conclusion follows from Lemma \ref{lem:singleSegment} since $\mathscr{D}_{j} \le \mathscr{D}_{j+1}$.

Now suppose that the theorem holds for $(g_{2}, \ldots, g_{n})$. By induction hypothesis, there exists $q_{2} \in [g_{1} x_{0}, wx_{0}]$ such that $[g_{1} x_{0}, q_{2}]$ and $[g_{1} x_{0}, g_{1} \phi_{2} x_{0}]$ $\mathscr{D}_{j+1}$-fellow travel. Note that the length of $[g_{1} x_{0}, q_{2}]$ is at least $d(x_{0}, \phi_{2} x_{0}) - 2\mathscr{D}_{j+1} \ge 2K$.

There exists $q_{1} \in [x_{0}, g_{1} x_{0}]$ such that $[g_{1} x_{0}, q_{1}]$ and $[g_{1} x_{0}, g_{1} \psi_1^{-1} x_{0}]$ $\mathscr{D}_{j}$-fellow travel. Then $[g_{1} x_{0}, q_{1}]$ is $\epsilon_{j}$-thick and its length is at least $d(x_{0}, \psi_1^{-1} x_{0}) - 2\mathscr{D}_{j} \ge 2K$.

Let $\eta = [y, y']$ be the subsegment of $[q_{1}, g_{1} x_{0}] \subseteq [x_{0}, g_{1} x_{0}]$ such that $d(y, y') = \mathscr{C}_{0}(\epsilon_{j})$ and $d(y', g_{1} x_{0}) = K$. By Theorem \ref{thm:Rafi2}, at least one of $[x_{0}, wx_{0}]$ and $[g_{1} x_{0}, wx_{0}]$ intersects the $\mathscr{D}_{0}(\epsilon_{j})$-neighborhood of $\eta$. 

Suppose that there exists a point $p \in [g_{1} x_{0}, wx_{0}]$ that belongs to the $\mathscr{D}_{0}(\epsilon_{j})$-neighborhood of $\eta$. We then have \[\begin{aligned}
d(g_{1} x_{0}, p) &\le d(g_{1} x_{0}, \eta) + \operatorname{diam}(\eta) + \mathscr{D}_{0}(\epsilon_{j})\\
&= K+ \mathscr{C}_{0}(\epsilon_{j}) + \mathscr{D}_{0}(\epsilon_{j}) \\
&< 2K\le d(g_{1} x_{0}, q_{2}).
\end{aligned}
\]
Hence, $d(g_{1}\psi_{1}^{-1}x_{0}, g_{1} \phi_{2} x_{0})$ is at most: \begin{equation}\label{eqn:singleDirection1}\begin{aligned}
& \le d(g_{1} \psi_{1}^{-1} x_{0},  q_{1}) + d(q_{1}, y) + d(y, p) + d(p, q_{2}) + d(q_{2}, g_{1}\phi_{2} x_{0}) \\
&= d(g_{1} \psi_{1}^{-1} x_{0}, q_{1}) + [d(q_{1}, g_{1} x_{0}) - d(y, g_{1} x_{0})] + d(y, p) \\
&  \quad + [ d(g_{1} x_{0}, q_{2}) - d(g_{1} x_{0}, p)] + d(q_{2}, g_{1} \phi_{2} x_{0}) \\
&\le  d(g_{1} \psi_{1}^{-1} x_{0}, q_{1}) + [d(q_{1}, g_{1} \psi_{1}^{-1} x_{0}) + d(g_{1} \phi_{1}^{-1} x_{0}, g_{1} x_{0})] - d(y, g_{1} x_{0}) + d(y, p) \\
& \quad + [d(g_{1} x_{0}, g_{1} \phi_{2} x_{0}) + d(g_{1}\phi_{2} x_{0}, q_{2})] - d(g_{1} x_{0}, p) + d(q_{2}, g_{1} \phi_{2} x_{0}) \\
& \le d(g_{1} \phi_{1}^{-1} x_{0}, g_{1} x_{0}) + d(g_{1} x_{0}, g_{1} \phi_{2} x_{0}) - 2d(y, g_{1} x_{0}) \\
& \quad + 2d(y, p) +  2d(q_{1}, g_{1}\psi_{1}^{-1} x_{0}) + 2d(q_{2}, g_{1} \phi_{2} x_{0}) \\
&\le d(g_{1} \phi_{1}^{-1} x_{0}, g_{1} x_{0}) + d(g_{1} x_{0}, g_{1} \phi_{2} x_{0}) - 2[K - \mathscr{D}_{0}(\epsilon_{j}) - \mathscr{C}_{0}(\epsilon_{j}) - 2\mathscr{D}_{j}].
\end{aligned}
\end{equation}
This implies that $(\phi_{2} x_{0}, \psi_{1}^{-1} x_{0})_{x_{0}} > C_{Grom}$, a contradiction.

Hence, we instead obtain a point $p \in [x_{0}, wx_{0}]$ that belongs to the $\mathscr{D}_{0}(\epsilon_{j})$-neighborhood of $\eta$. Then $p$ is within distance $\mathscr{D}_{0}(\epsilon_{j}) + \mathscr{C}_{0}(\epsilon_{j}) + K$ from $g_{1} x_{0}$. Then Corollary \ref{cor:Rafi1} tells us that $[x_{0}, g_{1} x_{0}]$ and $[x_{0}, p]$ $\mathscr{B}(\epsilon, 2K)$-fellow travel. Since $[g_{1}\psi_{1}^{-1} x_{0}, g_{1} x_{0}]$ $\mathscr{D}_{j}$-fellow travel with a terminal subsegment of $[x_{0}, g_{1} x_{0}]$, we conclude that $[g_{1} \psi_{1}^{-1}x_{0}, g_{1} x_{0}]$ and a terminal subsegment of $[x_{0}, p]$ $\mathscr{D}_{j+1}$-fellow travel. This establishes the first item.

If $[x_{0}, \phi_{1} x_{0}]$ $\mathscr{D}_{j}$-fellow travel with an initial subsegment of $[x_{0}, g_{1} x_{0}]$, we also conclude that an initial subsegment of $[x_{0}, p]$ and $[x_{0}, \phi_{1}x_{0}]$ $\mathscr{D}_{j+1}$-fellow travel. This establishes the second item.
\end{proof}

Several corollaries of Lemma \ref{lem:singleDirection} follow.

\begin{cor}\label{cor:distinctGuide}
Let $\phi_{1}, \phi_{2} \in \{\phi_{+}, \phi_{+}^{-1}, \phi_{-}, \phi_{-}^{-1}\}$ and $v\in G$. Suppose $v$ is  $\mathscr{D}_{5}$-witnessed by $\phi_{1}$ and $\phi_{2}$. Then $\phi_{1} = \phi_{2}$.
\end{cor}

\begin{proof}
Suppose not, i.e., $\phi_{1} \neq \phi_{2}$. This implies that the sequence $(v^{-1}, v)$ is $\mathscr{D}_{5}$-marked with $\phi_{1}$, $\phi_{2}$, which are repulsive. Lemma \ref{lem:singleDirection} then implies that $[x_{0}, x_{0}] = \{x_{0}\}$ is $\mathscr{D}_{6}$-witnessed by $[v^{-1}\phi_{2}^{-1}x_{0}, v^{-1}x_{0}]$, which is impossible since $d(x_{0}, \phi_{2} x_{0}) > 2\mathscr{D}_{6}$.
\end{proof}

\begin{cor}\label{cor:bothDirection}
Let $1\le j \le 4$, $(\phi_{i})$, $(\psi_{i})$ be repulsive sequences and $(g_{i})_{i=1}^{n}$ be a sequence that is $\mathscr{D}_{j}$-marked with $(\phi_{i})_{i=2}^{n}$ and $(\psi_{i})_{i=1}^{n-1}$. Let $w = g_{1} \cdots g_{n}$ and $i \in \{1, \ldots, n-1\}$. Then: \begin{enumerate}
\item $[x_{0}, w x_{0}]$ is $\mathscr{D}_{j+2}$-witnessed by $[g_{1} \cdots g_{i} x_{0}, g_{1} \cdots g_{i} \phi_{i+1} x_{0}]$;
\item $[x_{0}, wx_{0}]$ is $\mathscr{D}_{j+2}$-witnessed by $[g_{1} \cdots g_{i} \psi_{i}^{-1} x_{0}, g_{1} \cdots g_{i} x_{0}]$, and
\item $d(x_{0}, wx_{0}) \ge d(x_{0}, g_{1} \cdots g_{i} x_{0}) + d(x_{0}, g_{i+1} \cdots g_{n} x_{0}) - \mathscr{D}_{j+2}$. 
\end{enumerate}
\end{cor}

\begin{proof}
Note that $(g_{i+1}, \ldots, g_{n})$ is $\mathscr{D}_{j}$-strongly marked with $(\phi_{i+j})_{j=1}^{n-i}$, $(\psi_{i+j})_{j=1}^{n-i-1}$ and $(g_{i}^{-1}, \ldots, g_{1}^{-1})$ is $\mathscr{D}_{j}$-strongly marked with $(\psi_{i+1-j}^{-1})_{j=1}^{i}$, $(\phi_{i+1-j}^{-1})_{j=1}^{i-1}$. Thus, by Lemma \ref{lem:singleDirection}, $g_{i+1} \cdots g_{n}$ is $\mathscr{D}_{j+1}$-witnessed by $\phi_{i+1}$ and $g_{i}^{-1} \cdots g_{1}^{-1}$ is $\mathscr{D}_{j+1}$-witnessed by $\psi_{i}^{-1}$. In other words, $(g_{1} \cdots g_{i}, g_{i+1} \cdots g_{n})$ is $\mathscr{D}_{j+1}$-marked with $\phi_{i+1}$, $\psi_{i}$. We then apply Lemma \ref{lem:singleDirection} again to deduce the conclusion.
\end{proof}

The next corollary plays the role of Lemma \ref{lem:almostAdditive} on the Teichm{\"u}ller space.

\begin{cor}\label{cor:severalDirection}
Let $1\le j \le 4$, $(\phi_{i})_{i=2}^{n}$, $(\psi_{i})_{i=1}^{n-1}$ be repulsive sequences, $(g_{i})_{i=1}^{n}$ be a step sequence $\mathscr{D}_{j}$-marked with $(\phi_{i})$, $(\psi_{i})$, and $1=t_{0} < t_{1} < \cdots < t_{k}=n$. Then \[
\left|d(x_{0}, g_{1} \cdots g_{n} x_{0}) - \sum_{i=1}^{k} d(x_{0}, g_{t_{i-1}+1} \cdots g_{t_{i}} x_{0}) \right| \le (k-1)\mathscr{D}_{j+2}
\]
holds.
\end{cor}

\begin{proof}
Note that a subsequence of a step sequence marked with repulsive sequences is again marked by the corresponding repulsive subsequences. Thus, 
using induction on $k$, it suffices to prove the result for $k=2$. Then Corollary \ref{cor:bothDirection} applies.
\end{proof}

\subsection{Translation lengths of mapping classes}

So far, we observed that the directions witnessed by repulsive sequences of mapping classes are persistent in the final geodesic. Our next aim is to relate these recorded directions with the translation length of $w$ and analyze the effect of pivoting.

\begin{lem}\label{lem:transTech}
Let $h_{1}, h_{2} \in G$ and $\phi, \psi_{1}, \psi_{2} \in \{\phi_{+}, \phi_{+}^{-1}, \phi_{-}, \phi_{-}^{-1}\}$ be such that the following hold: \begin{itemize}
\item $d(x_{0}, h_{1} x_{0}) \ge d(x_{0}, h_{2} x_{0}) + Z$,
\item $h_{1}^{-1}$ is $\mathscr{D}_{3}$-witnessed by $\phi$, and
\item $\psi_{1} \neq \psi_{2}$.
\end{itemize}
Then the following hold: \begin{enumerate}
\item $h_{1}^{-1} h_{2} \psi_{i}$ is $\mathscr{D}_{4}$-witnessed by $\phi$ for each $i=1, 2$.
\item If $\psi_{1}^{-1} h_{2}^{-1} h_{1}$ is not $\mathscr{D}_{2}$-witnessed by $\psi_{1}^{-1}$, then $\psi_{2}^{-1} h_{2}^{-1} h_{1}$ is $\mathscr{D}_{2}$-witnessed by $\psi_{2}^{-1}$. 
\item If $\psi_{2}^{-1} h_{2}^{-1} h_{1}$ is not $\mathscr{D}_{2}$-witnessed by $\psi_{2}^{-1}$, then $\psi_{1}^{-1} h_{2}^{-1} h_{1}$ is $\mathscr{D}_{2}$-witnessed by $\psi_{1}^{-1}$. 
\end{enumerate}
\end{lem}

\begin{proof}
Let us first establish (1). The assumption tells us that $[h_{1}x_{0}, x_{0}]$ contains an initial subsegment $[h_{1}x_{0}, q_{1}]$ that $\mathscr{D}_{3}$-fellow travels with $[h_{1}x_{0}, h_{1}\phi x_{0}]$. This implies that $[h_{1} x_{0}, q_{1}]$ is $\epsilon_{3}$-thick. Moreover, the length of $[h_{1}x_{0}, q_{1}]$ is at least $d(x_{0}, \phi x_{0}) - \mathscr{D}_{3} \ge \mathscr{C}_{0}(\epsilon_{3})$ and at most $d(x_{0}, \phi x_{0}) + \mathscr{D}_{3}$.

Let $\eta = [y, q_{1}]$ be the subsegment of $[h_{1}x_{0}, q_{1}] \subseteq [h_{1}x_{0}, x_{0}]$ such that $d(y, q_{1}) = \mathscr{C}_{0}(\epsilon_{3})$. By Theorem \ref{thm:Rafi2}, at least one of $[h_{1} x_{0}, h_{2}\psi_{i} x_{0}]$ and $[h_{2}\psi_{i} x_{0}, x_{0}]$ intersects the $\mathscr{D}_{0}(\epsilon_{3})$-neighborhood of $\eta$.

If there exists a point $p \in [h_{2}\psi_{i} x_{0}, x_{0}]$ that belongs to the $\mathscr{D}_{0}(\epsilon_{3})$-neighborhood of $\eta$, we have \[ \begin{aligned}
d(x_{0}, h_{2} x_{0}) &\ge d(x_{0}, h_{2}\psi_{i} x_{0}) - d(h_{2} x_{0}, h_{2}\psi_{i} x_{0}) \\
&\ge d(x_{0}, p) - d(x_{0}, \psi_{i} x_{0}) \\
&\ge d(x_{0}, \eta) - \operatorname{diam}(\eta) - \mathscr{D}_{0}(\epsilon_{3}) - d(x_{0}, \psi_{i} x_{0}) \\
&\ge d(x_{0}, h_{1} x_{0})  - [d(x_{0}, \phi x_{0}) + \mathscr{D}_{3}] -\mathscr{C}_{0}(\epsilon_{3}) -\mathscr{D}_{0}(\epsilon_{3}) - d(x_{0}, \psi_{i} x_{0}).
\end{aligned}
\]
This contradicts the assumption that $d(x_{0}, h_{1} x_{0}) \ge d(x_{0}, h_{2} x_{0}) + Z$.

Hence, $[h_{1} x_{0}, h_{2}\psi_{i} x_{0}]$ contains a point $p$ that belongs to the $\mathscr{D}_{0}(\epsilon_{3})$-neighborhood of $\eta$. Then \[\begin{aligned}
d(p, h_{1} \phi x_{0}) &\le d(p, \eta) + \operatorname{diam}(\eta)  + d(q_{1}, h_{1} \phi x_{0}) \\
&\le \mathscr{D}_{0}(\epsilon_{3}) + \mathscr{C}_{0}(\epsilon_{3}) + \mathscr{D}_{3}
\end{aligned}
\]
and $[h_{1} x_{0}, p]$ $\mathscr{D}_{4}$-fellow travels with $[h_{1}x_0, h_{1}\phi x_{0}]$. 

Let us now establish (2). Let $K = 2\mathscr{C}_{0}(\epsilon) + 2\mathscr{D}_{0}(\epsilon) + C_{Grom}$. Recall that $[h_{2} x_{0}, h_{2} \psi_{1} x_{0}]$ is an $\epsilon$-thick geodesic whose length is at least $\mathscr{C}_{0}(\epsilon) + K$. 

Let $\eta = [y, y']$ be the subsegment of $[h_{2} x_{0}, h_{2}\psi_{1} x_{0}]$ with $d(y, y') = \mathscr{C}_{0}(\epsilon)$ and $d(h_{2}x_{0}, y) = K$. By Theorem \ref{thm:Rafi2}, either $[h_{1}x_{0}, h_{2}\psi_{1} x_{0}]$ or $[h_{1} x_{0}, h_{2}x_{0}]$ intersects the $\mathscr{D}_{0}(\epsilon)$-neighborhood of $\eta$. If there exists a point $p \in [h_{1}x_{0}, h_{2}\psi_{1} x_{0}]$ that belongs to $N_{\mathscr{D}_{0}(\epsilon)}(\eta)$, then $d(p, h_{2} x_{0}) \le \mathscr{C}_{0}(\epsilon) + \mathscr{D}_{0}(\epsilon) + K$. This implies that $[p, h_{2}\psi_{1} x_{0}]$ and $[h_{2} x_{0}, h_{2}\psi_{1} x_{0}]$ $\mathscr{D}_{2}$-fellow travel, which contradicts the assumption. Hence, we instead have a point $p \in [h_{1}x_{0}, h_{2}x_{0}]$ that belongs to $N_{\mathscr{D}_{0}(\epsilon)}(\eta)$.

We now consider a subsegment $\eta' = [z, z']$ of $[h_{2}x_{0}, h_{2} \psi_{2} x_{0}]$ with $d(z, z') = \mathscr{C}_{0}(\epsilon)$ and $d(h_{2}x_{0}, z) = K$.  By Theorem \ref{thm:Rafi2}, either $[h_{1}x_{0}, h_{2} \psi_{2} x_{0}]$ or $[h_{1}x_{0}, h_{2}x_{0}]$ intersects the $\mathscr{D}_{0}(\epsilon)$-neighborhood of $\eta$. 

Suppose that there exists a point $q \in [h_{1} x_{0}, h_{2}x_{0}]$ that belongs to $N_{\mathscr{D}_{0}(\epsilon)}(\eta')$. Then from the inequalities  \[\begin{aligned}
|d(h_{2}x_{0}, p) - d(h_{2}x_{0}, y)| &\le \operatorname{diam}(\eta) + \mathscr{D}_{0}(\epsilon), \\
|d(h_{2}x_{0}, q) - d(h_{2}x_{0}, z)| &\le \operatorname{diam}(\eta') + \mathscr{D}_{0}(\epsilon)
\end{aligned}
\]and the fact that $h_{2} x_{0}, p, q$ are on the same geodesic, we have \[\begin{aligned}
d(p, q) &= |d(h_{2} x_{0}, p) - d(h_{2} x_{0}, q)|\\
&\le \operatorname{diam}(\eta) + \operatorname{diam}(\eta') + 2\mathscr{D}_{0}(\epsilon).
\end{aligned}
\]
This in turn implies \[\begin{aligned}
d(h_{2} \psi_{1} x_{0},  h_{2} \psi_{2} x_{0}) &\le d(h_{2} \psi_{1} x_{0}, y) + d(y, p) + d(p, q) + d(q, z) + d(z, h_{2} \psi_{2} x_{0}) \\
&\le [d(h_{2} x_{0}, h_{2} \psi_{1} x_{0}) - K] + [\mathscr{C}_{0}(\epsilon) + \mathscr{D}_{0}(\epsilon)] + [2\mathscr{C}_{0}(\epsilon) + 2 \mathscr{D}_{0}(\epsilon)] \\
& \quad+  [\mathscr{C}_{0}(\epsilon) + \mathscr{D}_{0}(\epsilon)] +  [d(h_{2} x_{0}, h_{2} \psi_{2} x_{0}) - K] \\
&\le d(h_{2} x_{0}, h_{2}\psi_{1} x_{0}) + d(h_{2} x_{0}, h_{2} \psi_{2} x_{0}) - 2C_{Grom}.
\end{aligned}
\]
This contradicts the fact that $(\psi_{1} x_{0}, \psi_{2} x_{0})_{x_{0}} < C_{Grom}$. 

Hence, we instead have a point $q \in [h_{1}x_{0}, h_{2} \psi_{2} x_{0}]$ that is within distance $\mathscr{D}_{0}(\epsilon)$ from $\eta$. Then we have \[
d(q, h_{2} x_{0}) \le K + \mathscr{C}_{0}(\epsilon) + \mathscr{D}_{0}(\epsilon) \le 3\mathscr{C}_{0}(\epsilon) + 3\mathscr{D}_{0}(\epsilon) + C_{Grom}.
\]
Corollary \ref{cor:Rafi1} then tells us that $[q, h_{2}\psi_{2} x_{0}]$ and $[h_{2} x_{0}, h_{2} \psi_{2} x_{0}]$ $\mathscr{D}_{2}$-fellow travel as desired.

(3) is deduced from a similar argument.
\end{proof}

\section{Random walks on Teichm{\"u}ller space}

In this section, we adapt the proof of Theorem \ref{thm:A} to deal with Teichm{\"u}ller space. Before delving into details, we briefly sketch our plan. We will define persistent joints in sample paths and pivot the path at those joints as in Section \ref{section:translation}. (See Figure \ref{fig:pivoting}.) The basic philosophy of Section \ref{section:translation} was using Property \ref{property:Gromov} for hyperbolic spaces. More precisely, we applied Property \ref{property:Gromov} to the Gromov products \[
(x_{n \rightarrow 0}^{\sigma}, x_{n \rightarrow 0}^{\kappa})_{x_{0}} >> (x_{n \rightarrow 0}^{\kappa}, x_{n}^{\kappa})_{x_{0}} >> (x_{n}^{\kappa}, x_{n}^{\sigma})_{x_{0}}
\]
to conclude that $(x_{n}^{\sigma}, x_{n\rightarrow 0}^{\sigma})_{x_{0}}$ is small enough. This led to the lower bound on $\tau(w_{n}^{\sigma})$.

Our aim is to copy this phenomenon to the Teichm\"uller space $X$. However,  several issues arise due to the fact that $X$ is not Gromov hyperbolic: \begin{itemize}
\item Property \ref{property:Gromov} among the Gromov products may not hold;
\item small $(x_{n}^{\sigma}, x_{n \rightarrow 0}^{\sigma})_{x_{0}}$ may not lead to large translation length of $w_{n}^{\sigma}$.
\end{itemize}

In order to overcome these difficulties, we copy the following property of Gromov hyperbolicity: for a geodesic triangle with vertices $x$, $y$ and $z$, the edge $[x, y]$ fellow travels with either $[x, z]$ or $[z, x]$. Thanks to Rafi's theorems (Corollary \ref{cor:Rafi1} and Theorem \ref{thm:Rafi2}), we can partially guarantee such a phenomenon among certain Teichm{\"u}ller geodesics witnessed by pseudo-Anosov mapping classes.

We now begin our discussion. Recall that $\mu$ is a non-elementary probability measure on the mapping class group $G$. We have fixed two independent pseudo-Anosov mapping classes $\phi_{+}$, $\phi_{-}$ in $\supp \mu^{L}$ for some $L>0$. Here, $\phi_{\pm}$ are associated with constants $\epsilon, \mathscr{M}, C_{Grom}$ and $\mathscr{F}$ that satisfy Lemma \ref{lem:thick}, Lemma \ref{lem:distantOrbits} and Corollary \ref{cor:eventualWitness}. We have also defined constants $\mathscr{D}_{j}, \epsilon_{j}$ as in Display \ref{eqn:DEDfn} and assumed Inequality \ref{eqn:largeJump}.

Let $a_{i}, b_{i} \in \supp \mu$ be the letters for $\phi_{\pm}$ satisfying $\phi_{+}= b_{1} \cdots b_{L}$ and $\phi_{-}= a_{1} \cdots a_{L}$. As before, we fix the following notations: \[\begin{aligned}
p_{+} &=  \mu(a_{1}) \cdots \mu(a_{L}) , \\ p_{-} &= \mu(b_{1}) \cdots \mu(b_{L}), \\ 
P &= \max\left(\frac{p_{+}}{p_{+}+ p_{-}}, \frac{p_{-}}{p_{+} + p_{-}}\right). \end{aligned}
\]

This time, we define $\chi_{k}(\omega)$ as follows. $\chi_{k}(\omega) = 1$ if 
\begin{enumerate}
\item \[
(g_{3(k-1)L + 1}, \ldots, g_{3kL})= \left\{ \begin{array}{c} (b_{1}, \ldots, b_{L}, a_{1}, \ldots, a_{L}, b_{1}, \ldots, b_{L}) \\ \textrm{or} \\  (b_{1}, \ldots, b_{L}, b_{1}, \ldots, b_{L}, b_{1}, \ldots, b_{L})\end{array}\right.,
\]
\item $[x_{0}, x_{(3k-2)L \rightarrow n}]$ is $\mathscr{D}_{1}$-witnessed by $\phi_{+}^{-1}$ for $n \le 3(k-1)L$, and 
\item $[x_{0}, x_{(3k-1)L \rightarrow n}]$ is $\mathscr{D}_{1}$-witnessed by $\phi_{+}$ for $n \ge 3kL$,
\end{enumerate}
and $\chi_{k}(\w)=0$ otherwise.

We first observe that $\E(\chi_{1}(\w)) > 0$. The probability for condition (1) is $(p_{+} + p_{-}) p_{-}^{2} \neq 0$. Given (1) as the prior condition, (2) and (3) become independent events. Moreover, since $\mathscr{F} \le \mathscr{D}_{1}$, Corollary \ref{cor:eventualWitness} tells us that (2) and (3) hold for nonzero probability. Hence, we conclude that $\eta := \Prob(\chi_{1}(\w) =1) \neq 0$.

Note that $\chi_{k}(\w) = \chi_{1}(T^{3(k-1)L} \w)$. We define $W_{n} = \sum_{k=1}^{n} \chi_{k}(\w)$. Then $W_{n+m} = W_{n} + W_{m} \circ T^{3Ln}$ holds. Since $W_{1}$ is bounded, it has finite first moment. Applying Theorem \ref{thm:Woess}, we get almost everywhere convergence of $\frac{1}{n} W_{n}$ to an a.e. constant variable $W_{\infty}$. Since $\mathbb{E}(\frac{1}{n} W_{n}) = \eta > 0$, we have $W_{\infty} = \eta$ a.e. 

We also consider a modified version of $W_{n}$ as in Section \ref{section:translation}. Given positive integers $n \le m$, we say that $\mathcal{N}_{m, n} = \{n_1 < \cdots < n_k\} \subseteq 3L\Z$ is an $(m, n)$-set of pivots for $\vec{w} = (w_{1}, \cdots, w_{m})$ if the following holds:
\begin{enumerate}
\item $\mathcal{N}_{m, n} \subseteq \{1, \ldots, n\}$;
\item for each $i$, \[
\left(g_{n_{i}-3L+1}, \ldots, g_{n_{i}}\right)= \left\{ \begin{array}{c} \left(b_{1}, \ldots, b_{L}, a_{1}, \ldots, a_{L}, b_{1}, \ldots, b_{L}\right) \\ \textrm{or, } \\  \left(b_{1}, \ldots, b_{L}, b_{1}, \ldots, b_{L}, b_{1}, \ldots, b_{L}\right)\end{array}\right.,
\]
\item for each $i$, $[x_{0}, x_{n_{i} - 2L\rightarrow j}]$ is $\mathscr{D}_{1}$-witnessed by $\phi_{+}^{-1}$ for $n_{i-1} - L \le j \le n_{i} - 3L$, and 
\item for each $i$, $[x_{0}, x_{n_{i} - L\rightarrow j}]$ is $\mathscr{D}_{1}$-witnessed by $\phi_{+}$ for $n_{i} \le j \le n_{i+1} - 2L$.
\end{enumerate}
For convenience, we set $n_{0} = L$ and $n_{k+1} = m+2L$. 

As before, we denote the maximal $(m, n)$-set of pivots of $\vec{w}$ by $\mathcal{N}(\vec{w}) = \mathcal{N}_{m, n}(\vec{w})$, whose cardinality is denoted by $W_{n}^{m}$. Note that $W_{n}^{n} \ge W_{\lfloor n/3L \rfloor}$ always hold.

Now given a finite path $\vec{w} = (w_{1}, \ldots, w_{n})$ with $\mathcal{N}_{n, n}(\vec{w}) = \{n_{1} < \ldots < n_{k}\}$, we define the following: \[\begin{aligned}
A_{i}'(\vec{w}) &:= n_{i}(\vec{w}) - 3L,\\
\alpha_{i}'(\vec{w}) &:= n_{i}(\vec{w}) - 2L,\\
\beta_{i}'(\vec{w}) &:= n_{i}(\vec{w}) - L, \\
B_{i}'(\vec{w}) &= n_{i}(\vec{w}).
\end{aligned}
\]
Next, we set \[
(h_{2i-1}', h_{2i}') := (w_{\alpha_{i}'(\vec{w})}, w_{\beta_{i}'(\vec{w})})
\]
for $i=1, \ldots, k$ and $h_{0}' = id$, $h_{2k+1}' = w_{n}$. We can now discuss an analogue of Lemma \ref{lem:wellSevered}:

\begin{lem}\label{lem:wellSeveredTeich}
The sequence $(h_{i-1}'^{-1} h_{i}')_{i=1}^{2k+1}$ is $\mathscr{D}_{1}$-marked with repulsive sequences \[
(\phi_{+}, \phi_{1}', \phi_{+}, \phi_{2}', \ldots, \phi_{k}'), \,\, (\phi_{1}', \phi_{+}, \phi_{2}', \phi_{+}, \ldots, \phi_{+}),
\]
where $\phi_{i}' := h_{2i-1}'^{-1} h_{2i}'$ is either $\phi_{+}$ or $\phi_{-}$. We also have $h_{i}'^{-1} h_{j} \in \mathcal{C}_{\mathscr{D}_{2}}(\psi \rightarrow \psi')$ for $0 \le i < j \le 2k+1$, where \[
\psi = \left\{ \begin{array}{cc} \phi_{+} & \textrm{$i$ is odd} \\  \phi_{i/2}' & \textrm{$i$ is even} \end{array}\right.,\quad\psi' = \left\{ \begin{array}{cc} \phi_{(i+1)/2}' & \textrm{$i$ is odd} \\  \phi_{+} & \textrm{$i$ is even} \end{array}\right..
\]
Here, we set $\phi_{0}' = \phi_{k+1}' = id$ for convenience.

Moreover, we have \begin{align}\label{eqn:wellSeveredTeich1}
(h_{i}' x_{0}, h_{l}' x_{0})_{h_{j}' x_{0}} &\le \mathscr{D}_{3},\\ \label{eqn:wellSeveredTeich2}
d(h_{i}' x_{0}, h_{l}'x_{0}) &\ge d(h_{i}'x_0, h_{j}' x_{0}) + \mathscr{D}_{3}(l - j)
\end{align}
for each $0 \le i \le j \le l \le 2k+1$.
\end{lem}

\begin{proof}
For each $i = 1, \ldots, k$, $h_{2i-1}'^{-1} h_{2i}' = w_{\alpha_{i}'(\vec{w})}^{-1} w_{\beta_{i}'(\vec{w})} = \phi_{i}'$ for some $\phi_{i}' \in \{\phi_{+}, \phi_{-}\}$. Lemma \ref{lem:selfWitness} then tells us that $h_{2i-1}'^{-1} h_{2i}' \in \mathcal{C}_{\mathscr{D}_{1}}(\phi_{i}' \rightarrow \phi_{i}')$.

Moreover, for each $i=2, \ldots, k$, Condition (3) and (4) for an $(m,n)$-set of pivots tell us that:\begin{enumerate}
\item $[x_{0}, x_{n_{i} - 2L\rightarrow n_{i-1} - L}] = [x_{0}, x_{\alpha_{i}'(\vec{w})\rightarrow \beta_{i-1}'(\vec{w})}]$ is $\mathscr{D}_{1}$-witnessed by $\phi_{+}^{-1}$, and 
\item $[x_{0}, x_{n_{i-1} - L\rightarrow n_{i} - 2L}] = [x_{0}, x_{\beta_{i-1}'(\vec{w})\rightarrow \alpha_{i}'(\vec{w})}]$ is $\mathscr{D}_{1}$-witnessed by $\phi_{+}$.
\end{enumerate}
This means that $h_{2i-2}'^{-1} h_{2i-1}' = w_{\beta_{i-1}'(\vec{w})}^{-1} w_{\alpha_{i}'(\vec{w})}$ belongs to $\mathcal{C}_{\mathscr{D}_{1}}(\phi_{+} \rightarrow \phi_{+})$. Similarly, we observe $h_{1}' \in \mathcal{C}_{\mathscr{D}_{1}}(0 \rightarrow \phi_{+})$ and $h_{2k+1}' \in \mathcal{C}_{\mathscr{D}_{1}}(\phi_{+}\rightarrow 0)$. Since $\phi_{i}'$ is either $\phi_{+}$ or $\phi_{-}$, $\phi_{i}'$ and $\phi_{+}^{-1}$ are distinct for each $i$. This concludes that $(h_{i-1}'^{-1} h_{i}')_{i=1}^{2k+1}$ is $\mathscr{D}_{1}$-marked with repulsive sequences \[
(\phi_{+}, \phi_{1}', \phi_{+}, \phi_{2}', \ldots, \phi_{k}), \,\, (\phi_{1}', \phi_{+}, \phi_{2}', \phi_{+}, \ldots, \phi_{+}).
\]
Lemma \ref{lem:singleDirection} then tells us that $h_{i}'^{-1} h_{j} \in  \mathcal{C}_{\mathscr{D}_{2}} (\psi \rightarrow \psi')$ for each pair $(i, j)$ and the corresponding $\psi$, $\psi'$. Moreover, Corollary \ref{cor:bothDirection} implies Inequality \ref{eqn:wellSeveredTeich1} for each triple $(i, j, l)$. 

Finally, for each $i = 2, \ldots, 2k+1$, $h_{i-1}'^{-1} h_{i}'$ is $\mathscr{D}_{1}$-witnessed by either $\phi_{+}$ or $\phi_{-}$. Moreover, $h_{1}'^{-1}$ is $\mathscr{D}_{1}$-witnessed by $\phi_{+}^{-1}$. Hence, for $i = 1, \ldots, 2k+1$, we have \[
d(h_{i}' x_{0}, h_{i+1}' x_{0}) \ge d(x_{0}, \phi_{\pm} x_{0}) - 2\mathscr{D}_{1} \ge 4 \mathscr{D}_{3}.
\] Combining this with Inequality \ref{eqn:wellSeveredTeich1} yields Inequality \ref{eqn:wellSeveredTeich2}.
\end{proof}

We now set \[\begin{aligned}
M&>1000(1 + L + L/\mathscr{D}_{3} + \mathscr{D}_{3} + Z),\\
&D =10\eta/M,\quad 0.9 \le Q \le 1, \\
\end{aligned}
\]
and \[\begin{aligned}
F_{n} &:=\left\{ \vec{w} = (w_{1}, \cdots, w_{n}) : W_{n}^{n} \ge \frac{\eta n}{6L}+1\right\}, \\ 
G_{n}&:= \left\{ \vec{w} \in F_{n}: \tau(w_{n}) \le \left(2D-\frac{2\eta}{M}\right)n\right\} .
\end{aligned}
\]

As before, for $\vec{w} \in F_{n}$ and $Q> 0$ we also define \[\begin{aligned}
\mathcal{N}_{f}(\vec{w};Q) &:= \left\{ s \in \mathcal{N}(\vec{w}) : d(x_{0}, x_{s - L}) \le \frac{1}{2} d(x_{0}, x_{n}) -(D+Q\eta/M)n \right\},\\ 
\mathcal{N}_{b}(\vec{w};Q) &= \left\{ s \in \mathcal{N}(\vec{w}): d(x_{n}, x_{s - 2L}) \le \frac{1}{2} d(x_{0}, x_{n}) - (D+Q\eta/M)n\right\},\\
\mathcal{N}_{0}(\vec{w};Q) &:= \mathcal{N}(\vec{w})\setminus (\mathcal{N}_{f}(\vec{w};Q) \cup \mathcal{N}_{b}(\vec{w};Q)).
\end{aligned}
\]

\begin{lem}\label{lem:setsizeTeich}
For sufficiently large $n$, if $\vec{w} \in F_{n}$ and $Q \le 1$ then \[
|\mathcal{N}_{f}(\vec{w};Q)| \ge \frac{\eta n}{20 L}\,\,\textrm{or}\,\, |\mathcal{N}_{b}(\vec{w};Q)| \ge \frac{\eta n}{20 L}
\]
holds.
\end{lem}

\begin{proof}
If not, we have $|\mathcal{N}_{0}(\vec{w}; Q)| \ge \frac{\eta n}{15 L}+1$. As in the proof of Lemma \ref{lem:setsize}, we can take $t, t'$ such that $\mathcal{N}_{0}(\vec{w}; Q) = \{n_{t}, n_{t+1}, \ldots, n_{t'}\}$. Using Inequality \ref{eqn:wellSeveredTeich1} and \ref{eqn:wellSeveredTeich2}, we obtain \[
d(x_{\beta'_{t}}, x_{\alpha_{t'}'})  \ge \mathscr{D}_{3}(t' - t + 1) \ge \frac{1000L}{M} \cdot \frac{\eta n}{15 L} \ge \frac{60\eta n}{M} + 4\mathscr{D}_{3}
\]
for large enough $n$. Moreover, since $n_{t} \notin \mathcal{N}_{f}(\vec{w};Q)$ and $n_{t'} \notin \mathcal{N}_{b}(\vec{w};Q)$, we have \[
d(x_{0}, x_{\beta'_t}), d(x_{\alpha_{t'}'}, x_{n}) > \frac{1}{2} d(x_{0}, x_{n}) - (D + Q\eta / M)n.
\]
Combining these inequalities, we obtain
\[\begin{aligned}
d(x_{0}, x_{n}) &\ge d(x_{0}, x_{\beta_{t}'}) + d(x_{\beta_{t}'}, x_{\alpha_{t'}'}) + d(x_{\alpha_{t'}'}, x_{n}) - 4\mathscr{D}_{3} \\
&\ge 2\left(\frac{1}{2} d(x_{0}, x_{n}) - (D + Q\eta / M)n\right) + \frac{60\eta n}{M} \\
&\ge d(x_{0}, x_{n}) + \frac{\eta n}{M},
\end{aligned}
\]
a contradiction.
\end{proof}

Thus, when $Q \le 1$, \[\begin{aligned}
F_{n, f}(Q) & := \left\{ \vec{w} \in F_{n} : |\mathcal{N}_{f}(\vec{w}; Q)| \ge \frac{\eta n}{20 L} \right\},\\
F_{n, b}(Q) & := \left\{\vec{w} \in F_n : |\mathcal{N}_{b}(\vec{w}; Q)| \ge \frac{\eta n}{20L} \right\}
\end{aligned}
\]
cover entire $F_{n}$. We also define $$G_{n, f}(Q) := F_{n, f}(Q) \cap G_{n} \quad \mbox{and} \quad G_{n, b}(Q) := F_{n, b}(Q)\cap G_{n}.$$
From now on, we will focus on $\vec{w} \in F_{n, f}(Q)$; the argument for $\vec{w} \in F_{n, b}(Q)$ is analogous.

For each $\vec{w} \in F_{n, f}(Q)$, we fix an integer $N =N(\vec{w})$ between $\frac{\eta n}{50M^{2} L}$ and $\frac{\eta n}{20M^{2}L}$ as before and pick pivot indices $p_{1}(\vec{w}), \ldots, p_{N}(\vec{w})$ from $\mathcal{N}_{f}(\vec{w}; Q)$. We also bring the previous notation \[
\begin{aligned}
A_{i}(\vec{w}) &:= p_{i}(\vec{w}) - 3L,\\
\alpha_{i}(\vec{w}) &:= p_{i}(\vec{w}) - 2L,\\
\beta_{i}(\vec{w}) &:= p_{i}(\vec{w}) - L,\\
B_{i} (\vec{w}) &:= p_{i}(\vec{w}).
\end{aligned}
\]
We also let $\beta_{0}(\vec{w}) = B_{0}(\vec{w}) := 0$, $\alpha_{N+1}(\vec{w}) = A_{N+1}(\vec{w}) := n$,  \[
(h_{2i-1}, h_{2i}) := (w_{\alpha_{i}(\vec{w})}, w_{\beta_{i}(\vec{w})})
\]
for $i=1, \ldots, N$ and $h_{0} = id$, $h_{2k+1} = w_{n}$. Since $(\alpha_{i})_{i}$, $(\beta_{i})_{i}$ are subsequences of $(\alpha_{i}')_{i}$ and $(\beta_{i}')_{i}$, respectively, the marking information from Lemma \ref{lem:wellSeveredTeich} and Lemma \ref{lem:singleDirection} tell us the following.

\begin{lem}\label{lem:wellSeveredPivotTeich}
The sequence $(h_{i-1}^{-1} h_{i})_{i=1}^{2k+1}$ is $\mathscr{D}_{2}$-marked with repulsive sequences \[
(\phi_{+}, \phi_{1}, \phi_{+}, \phi_{2}, \ldots, \phi_{k}), \,\, (\phi_{1}, \phi_{+}, \phi_{2}, \phi_{+}, \ldots, \phi_{+}),
\]
where $\phi_{i} = h_{2i-1}^{-1} h_{2i}$ is either $\phi_{+}$ or $\phi_{-}$. Moreover, we have \begin{align}\label{eqn:wellSeveredPivotTeich1}
(h_{i} x_{0}, h_{l} x_{0})_{h_{j} x_{0}} &\le \mathscr{D}_{3},\\ \label{eqn:wellSeveredPivotTeich2}
d(h_{i}x_{0}, h_{l}x_{0}) &\ge d(h_{i}, h_{j} x_{0}) + \mathscr{D}_{3}(l - j)
\end{align}
for each $0 \le i \le j \le l \le 2k+1$.
\end{lem}

For each choice $\sigma \in \{0, 1\}^{N}$, we define the pivoted $\vec{w}^{\sigma}$ as before: we modify the type of joints that are marked by $\sigma$ only. Precisely speaking, the step sequences $(g_{i})_{i}$ and $(g_{i}^{\sigma})_{i}$ for $\vec{w}$ and $\vec{w}^{\sigma}$ coincide except at $\alpha_{j} + 1 \le i \le \beta_{j}$ for some $j$ such that $\sigma(j) = 1$. For $\sigma(j) = 1$ we set \[
(g^{\sigma}_{\alpha_j + 1}, \ldots, g^{\sigma}_{\beta_j}) := \begin{cases}

(a_1, \ldots, a_L) &\mbox{if } (g_{\alpha_j + 1}, \ldots, g_{\beta_j}) = (b_1, \ldots, b_L)\\
(b_1, \ldots, b_L) &\mbox{if } (g_{\alpha_j + 1}, \ldots, g_{\beta_j}) = (a_1, \ldots, a_L)

\end{cases}
\]
Other steps remain unchanged. Now, for \[
(h_{2i-1}^{\sigma}, h_{2i}^{\sigma}) := (w_{\alpha_{i}(\vec{w})}^{\sigma}, w_{\beta_{i}(\vec{w})}^{\sigma}),
\]
we have the following observation:

\begin{lem}\label{lem:wellSeveredPivotTeich2}
The sequence $((h_{i-1}^{\sigma})^{-1} h_{i}^{\sigma})_{i=1}^{2k+1}$ is $\mathscr{D}_{2}$-marked with repulsive sequences \[
(\phi_{+}, \phi_{1}^{\sigma}, \phi_{+}, \phi_{2}^{\sigma}, \ldots, \phi_{k}), \,\, (\phi_{1}^{\sigma}, \phi_{+}, \phi_{2}^{\sigma}, \phi_{+}, \ldots, \phi_{+}),
\]
where $\phi_{i}^{\sigma} \in \{\phi_{+}, \phi_{-}\}$ and $\phi_{i}^{\sigma} = \phi_{i}$ if and only if $\sigma(i) = 0$. Moreover, we have \begin{align}\label{eqn:wellSeveredPivotTeich1}
(h_{i}^{\sigma} x_{0}, h_{l}^{\sigma} x_{0})_{h_{j}^{\sigma} x_{0}} &\le \mathscr{D}_{3},\\ \label{eqn:wellSeveredPivotTeich2}
d(h_{i}^{\sigma}x_{0}, h_{l}^{\sigma}x_{0}) &\ge d(h_{i}^{\sigma}, h_{j}^{\sigma} x_{0}) + \mathscr{D}_{3}(l - j)
\end{align}
for each $0 \le i \le j \le l \le 2k+1$.
\end{lem}

\begin{proof}
Since $g_{i} = g_{i}^{\sigma}$ for all $i$ except for $\alpha_{j+1} \le i \le \beta_{j}$ where $\sigma(j) = 1$. In particular, we have $w_{\beta_{i-1}'(\vec{w}) \rightarrow \alpha_{i}'(\vec{w})}^{\sigma} \in \mathcal{C}_{\mathscr{D}_{1}}(\phi_{+} \rightarrow \phi_{+})$ for various $i$. Moreover, $w_{\alpha_{i}'(\vec{w}) \rightarrow \beta_{i}'(\vec{w})}^{\sigma} = (\phi_{i}'^{\sigma})^{F+1}\in \mathcal{C}_{\mathscr{D}_{1}}(\phi_{i}' \rightarrow \phi_{i}')$ for some $\phi_{i}'^{\sigma} \in \{\phi_{+}, \phi_{-}\}$. Then we have the marking information as in Lemma \ref{lem:wellSeveredTeich} (with constant $\mathscr{D}_{1}$), and the marking information as in Lemma \ref{lem:wellSeveredPivotTeich} (with constant $\mathscr{D}_{2}$) by taking subsequences. Finally, $\phi_{i}'^{\sigma} \neq \phi_{i}'$ if and only if $p_{i}'$ is chosen as a pivot index $p_{j}$ such that $\sigma(j) = 1$. This leads to the condition when $\phi_{i}^{\sigma}$ equals $\phi_{i}$.
\end{proof}

 We now prove an analogy of Lemma \ref{claim:step1} for sufficiently large $n$ in a different way, in the case of the Teichm\"uller space. 

\begin{lem}\label{claim:step1Teich}
Suppose that $n \ge (Z + 2 \mathscr{D}_{6})M/\eta$.	If $\kappa \neq \sigma \in \{0, 1\}^{N}$ and $\tau(w_{n}^{\kappa}) \le (2D - 2\eta /M)n$, then $$\tau(w_{n}^{\sigma}) \ge (2D + \eta/M)n.$$
\end{lem}
 \begin{proof}
 We first need the result of Claim \ref{claim:absolutediff}. Lemma \ref{lem:wellSeveredTeich} tells us that \[\begin{aligned}
\left|d(x_{0}, x_{n}^{\sigma}) - \left(\sum_{i=1}^{N+1} d(x_{\beta_{i-1}}^{\sigma}, x_{\alpha_{i}}^{\sigma}) +\sum_{i=1}^{N} d(x_{\alpha_{i}}^{\sigma}, x_{\beta_{i}}^{\sigma})\right) \right| \le 2N \mathscr{D}_{3}, \\
\left|d(x_{0}, x_{n}^{\kappa}) - \left(\sum_{i=1}^{N+1} d(x_{\beta_{i-1}}^{\kappa}, x_{\alpha_{i}}^{\kappa}) + \sum_{i=1}^{N} d(x_{\alpha_{i}}^{\kappa}, x_{\beta_{i}}^{\kappa})\right) \right| \le 2N \mathscr{D}_{3}.
\end{aligned}
\]
Recall that: \begin{enumerate}
\item $d(x_{\beta_{i-1}}^{\kappa}, x_{\alpha_{i}}^{\kappa}) = d(x_{\beta_{i-1}}^{\sigma}, x_{\alpha_{i}}^{\sigma})$ for all $i$, 
\item $d(x_{\alpha_{i}}^{\kappa}, x_{\beta_{i}}^{\kappa}) = d(x_{\alpha_{i}}^{\sigma}, x_{\beta_{i}}^{\sigma})$ for $i$ such that $\sigma(i) = 0$, and 
\item $|d(x_{\alpha_{i}}^{\kappa}, x_{\beta_{i}}^{\kappa}) - d(x_{\alpha_{i}}^{\sigma}, x_{\beta_{i}}^{\sigma})| \le Z$.
\end{enumerate}
From these, we deduce that \[
|d(x_{0}, x_{n}^{\sigma}) - d(x_{0}, x_{n}^{\kappa})| \le 4N(\mathscr{D}_{3} +Z) \le \frac{\eta n}{5M^{2}L}(\mathscr{D}_{3} + Z) \le \frac{\eta n}{40 M}.
\] For similar reasons, we also get \[
\begin{aligned}
|d(x_n, x_{\alpha_{i}'}^{\kappa}) - d(x_n, x_{\alpha_{i}'}^{\sigma})|& \le {\eta n \over 40M}, \\
|d(x_0, x_{\beta_{i}'}^{\kappa}) - d(x_0, x_{\beta_{i}'}^{\sigma})|& \le {\eta n \over 40M}
\end{aligned}
\]for $i=1, \cdots, N$ by considering partial sums.

Among the indices at which $\sigma$ and $\kappa$ differ, let $i$ and $j$ be the first and the last ones, respectively. We then have \begin{equation}\label{eqn:transTeichFirstLast1} \begin{aligned}
d(x_{0}, x_{\alpha(i)}^{\kappa}) + d(x_{0}, x_{\beta(j)}^{\kappa}) &\le 
d(x_{0}, x_{\alpha(i)}) + d(x_{0}, x_{\beta(j)}) + \frac{\eta n}{20 M} \\
&\le d(x_{0}, x_{n}) - 2(D + Q \eta / M)n + \frac{\eta n}{20 M} \\
&\le d(x_{0}, x_{n}^{\kappa}) - 2(D + 0.9 \eta /M) n \\
&\le d(x_{0}, x_{n}^{\kappa}) - Z
\end{aligned}
\end{equation}
for large enough $n$. We similarly have \begin{equation}\label{eqn:transTeichFirstLast2}
2d(x_{0}, x_{\beta(j)}^{\kappa}) \le d(x_{0}, x_{n}^{\kappa}) - 2(D + 0.95 \eta /M) n.
\end{equation}

Now Inequality \ref{eqn:transTeichFirstLast1} implies$$\begin{aligned}
	d(x_0, x^{\kappa}_{n \to \beta(j)}) & \ge d(x_0, x^{\kappa}_n) - d(x_0, x^{\kappa}_{\beta(j)}) \\
	& \ge d(x_0, x^{\kappa}_{\alpha(i)}) + Z.
	\end{aligned}
	$$
Moreover, $w_{\beta(j) \rightarrow n}^{\kappa}$ is $\mathscr{D}_{2}$-witnessed by $\phi_{+}$. Lemma \ref{lem:transTech} then tells us that $w_{\beta(j) \rightarrow n}^{\kappa} w_{0 \rightarrow \alpha(i)}^{\kappa} \phi_{\pm}$ is $\mathscr{D}_{4}$-witnessed by $\phi_{+}$. Also note that $w_{\beta(j) \rightarrow n}^{\kappa} = w_{\beta(j) \rightarrow n}^{\sigma}$ and $w_{0 \rightarrow \alpha(i)}^{\kappa} = w_{0 \rightarrow \alpha(i)}^{\sigma}$.

At this moment, if $(\phi_{i}^{\kappa})^{-1} w_{\alpha(i) \rightarrow 0}^{\kappa} w_{n \rightarrow \beta(j)}^{\kappa}$ is $\mathscr{D}_{2}$-witnessed by $(\phi_{i}^{\kappa})^{-1}$, then \[\begin{aligned}
& \quad \,\,(w_{\beta(j) \rightarrow n}^{\kappa} w_{0 \rightarrow \beta(i)}^{\kappa}\quad \,,w_{\beta(i) \rightarrow \beta(j)}^{\kappa}\, ,w_{\beta(j) \rightarrow n}^{\kappa} w_{0 \rightarrow \beta(i)}^{\kappa}\quad \, , \w_{\beta(i) \rightarrow \beta(j)}^{\kappa}\, ,\ldots) \\
&= (w_{\beta(j) \rightarrow n}^{\kappa} w_{0 \rightarrow \alpha(i)}^{\kappa} \phi_{i}^{\kappa}, w_{\beta(i) \rightarrow \beta(j)}^{\kappa}\, ,w_{\beta(j) \rightarrow n}^{\kappa} w_{0 \rightarrow \alpha(i)}^{\kappa} \phi_{i}^{\kappa}, w_{\beta(i) \rightarrow \beta(j)}^{\kappa}\, , \ldots)
\end{aligned}
\]
is $\mathscr{D}_{4}$-marked with repulsive sequences \[
(\phi_{+}, \phi_{+}, \ldots), \,\,(\phi_{i}^{\kappa}, \phi_{+}, \phi_{i}^{\kappa}, \phi_{+}, \ldots).
\]
Corollary \ref{cor:severalDirection} then tells us that $(v^{i}x_{0}, v^{k} x_{0})_{v^{j} x_{0}} \le \mathscr{D}_{6}$ for $v=w_{\beta(j) \rightarrow n}^{\kappa}w_{0 \rightarrow \beta(j)}^{\kappa}$ and $i < j < k$. This implies
\[
\begin{aligned}
\tau(w_{n}^{\kappa}) &= \tau((w_{\beta(j)}^{\kappa})^{-1} w_{n}^{\kappa} w_{\beta(j)}^{\kappa}) \\
&= \lim_{m \rightarrow \infty} \frac{1}{m} d(x_{0}, v^{m} x_{0}) \ge d(x_{0}, vx_{0}) - 2\mathscr{D}_{6} \\
&\ge d(x_{0}, w_{\beta(j) \rightarrow n}^{\kappa} x_{0}) - d(w_{\beta(j) \rightarrow n}^{\kappa}w_{0 \rightarrow \beta(j)}^{\kappa} x_{0}, w_{\beta(j) \rightarrow n}^{\kappa} x_{0}) - 2\mathscr{D}_{6}\\
&\ge d(x_{0}, x_{n}^{\kappa}) - 2d(x_{0}, x_{\beta(j)}^{\kappa}) - 2\mathscr{D}_{6} \\
&\ge 2(D+ 0.95 \eta/M) n - 2\mathscr{D}_{6},
\end{aligned}
\]
which contradicts the fact that $\tau(w_{n}^{\kappa}) \le (2D - 2\eta/M) n$.

Hence, $(\phi_{i}^{\kappa})^{-1} w_{\alpha(i) \rightarrow 0}^{\kappa} w_{n \rightarrow \beta(j)}^{\kappa}$ is not  $\mathscr{D}_{2}$-witnessed by $(\phi_{i}^{\kappa})^{-1}$. Instead, by Lemma \ref{lem:transTech}, $(\phi_{i}^{\sigma})^{-1} w_{\alpha(i) \rightarrow 0}^{\kappa} w_{n \rightarrow \beta(j)}^{\kappa}$ is $\mathscr{D}_{2}$-witnessed by $(\phi_{i}^{\sigma})^{-1}$. Then the above argument tells us that $\tau(w_{n}^{\sigma}) \ge (2D + 0.9\eta/M) n$. 
\end{proof}

In particular, for $\vec{w} \in G_{n, f}(Q)$, $\vec{w}^{\sigma} \notin G_{n}$ for any nontrivial $\sigma$. We now observe an analogue of Lemma \ref{claim:step2}.

\begin{lem}\label{claim:step2Teich}
Let  $Q=1$ and $n$ be a sufficiently large integer. Suppose that $\vec{w}, \vec{w}' \in G_{n, f}(Q=1)$ and the number of pivots $N(\vec{w})$, $N(\vec{w}')$ are $\lfloor \frac{\eta n}{40 M^{2}L}\rfloor$. Then for $\sigma, \sigma' \in \{0, 1\}^{N}$, $\vec{w}^{\sigma} = \vec{w}'^{\sigma'}$ if and only if $\vec{w}= \vec{w}'$ and $\sigma = \sigma'$.
\end{lem}

\begin{proof}

Let $\vec{v} = \vec{w}^{\sigma} = \vec{w}'^{\sigma'}$. As before, note that $$\mathcal{N}_{f}(\vec{w}; Q=1) \subseteq \mathcal{N}(\vec{w}) \subseteq \mathcal{N}(\vec{v}=\vec{w}^{\sigma}).$$ 
Moreover, for each $n_{i} \in \mathcal{N}_{f}(\vec{w}; Q=1)$ we have \[\begin{aligned}
d(x_{0}, x_{\beta_{i}'}^{\sigma})&\le d(x_{0}, x_{\beta_{i}'}) + \frac{\eta n}{40M}\\
&\le \frac{1}{2} d(x_{0}, x_{n})+ \frac{\eta n}{40M}- \left(D + \frac{\eta}{M}\right)n \\
&\le \frac{1}{2} d(x_{0}, x_{n}^{\sigma}) + \frac{3\eta n}{80M} - \left(D + \frac{\eta}{M}\right)n\\
&\le \frac{1}{2} d(x_{0}, x_{n}^{\sigma})- \left(D + \frac{0.9\eta}{M}\right)n
\end{aligned}
\]
from Claim \ref{claim:absolutediff}. Thus $n_{i} \in \mathcal{N}_{f}(\vec{v}; Q = 0.9)$. It follows that $$\mathcal{N}_{f}(\vec{w}; Q=1) \subseteq \mathcal{N}_{f}(\vec{v}; Q = 0.9)\quad \mbox{and} \quad  \vec{v} \in F_{n, f}(Q=0.9).$$ Similarly, $\mathcal{N}_{f}(\vec{w}'; Q=1) \subseteq \mathcal{N}_{f} (\vec{v}; Q=0.9)$.

Thus, we are able to pick forward pivots $p_{i}(\vec{w})$ and $p'_{i}(\vec{w}')$ of $\vec{w}$ and $\vec{w}'$ altogether for $\vec{v}$. (This will give $N(\vec{v}) \le \frac{\eta n}{20 M^{2}L}$ so we are safe.) Then Lemma \ref{claim:step1Teich} applied to $\vec{v} \in F_{n, f}(Q=0.9)$ yields a contradiction with $\vec{w}, \vec{w}' \in G_{n}$ unless $\sigma = \sigma'$.
\end{proof}

We are now ready to prove the second main theorem, Theorem \ref{thm:B}:

\mainthmTeich*

\begin{proof}
Lemma \ref{claim:step2Teich} tells us that $\Prob((\w_{i})_{i=1}^{n} \in G_{n, f}(Q=1)) \le P^{n}$ for sufficiently large $n$. Similarly we obtain $\Prob((\w_{i})_{i=1}^{n} \in G_{n, b}(Q=1)) \le P^{n}$. Since $G_{n, f}(Q=1) \cup G_{n, b}(Q=1) = G_{n}$, we deduce that $\Prob((\w_i)_{i=1}^n \in G_{n}) \le 2P^{n}$ for sufficiently large $n$. Thus, by the Borel-Cantelli lemma, almost every path avoids $G_{n}$ eventually. If $\w$ avoids $G_{n}$ eventually but $\tau(\w_{n}) \le (2D - 2\eta/M) n$ infinitely often, then $\w_{n}$ avoids $F_{n}$ infinitely often. However, since $W_{n}^{n} \ge W_{\lfloor n/3L\rfloor}$, the subadditive ergodic theorem implies that such path constitutes a measure zero set.

We now investigate the second assertion of Theorem \ref{thm:B}. When $\mu$ has finite first moment, the drift $\lambda$ of the random walk $\w$ is finite and strictly positive. Here, the strict positivity of $\lambda$ follows from the non-triviality of the Poisson boundary of non-elementary random walks on Teichm{\"u}ller space (cf. \cite[Theorem 2.3.2]{kaimanovich1996poisson}). As in the proof of the second assertion of Theorem \ref{thm:A}, we define the following set for $0<\epsilon < 1$ and $D = 0.5(1-\epsilon/2)\lambda$: $$\begin{aligned}
F_{n} & := \left\{ \vec{w} : \begin{matrix}
\# \mathcal{N}_{\lfloor \epsilon n/7 \rfloor, n}(\vec{w}) \ge {\eta \epsilon n \over 24L}, \,\,d(x_{0}, w_{n} x_{0}) \ge \left(1-\frac{\epsilon}{1000}\right) \lambda n,\\
d(x_{0}, w_{\lfloor \epsilon n / 7\rfloor} x_{0}) \le \frac{\epsilon \lambda n}{6} \end{matrix} \right\}\\
G_{n} & := \left\{ \vec{w} \in F_{n}: \tau(w_{n}) \le \left(2D - \frac{2\eta}{M} \right) n \right\}.
\end{aligned}$$

By the subadditive ergodic theorem, a.e. sample path $\w$ belongs to $F_{n}$ for all sufficiently large $n$. We then require $M > \frac{2000 \eta}{\lambda \epsilon} + {2000\eta \over \epsilon} + {2000 \over \epsilon} + 1$ and $0.9 \le Q \le 1$. Given $\vec{w} \in F_{n}$, we observe as before that \[
d(x_{0}, x_{n_{i}- L}) \le d(x_{0}, w_{\lfloor \epsilon n/7 \rfloor} x_{0}) \le \frac{1}{2}d(x_{0}, w_{n} x_{0}) - \left( D + \frac{\eta}{M}\right) n
\] 
for each $n_{i} \in \mathcal{N}_{\lfloor \epsilon n/7 \rfloor, n} (\vec{w})$. This implies that $\mathcal{N}_{\lfloor \epsilon n/7 \rfloor , n} (\vec{w}) \subseteq \mathcal{N}_{f}(\vec{w})$ and $\vec{w} \in F_{n, f}$. In other words, $F_{n, f}$ already covers $F_{n}$, which is stronger than what we hope in Lemma \ref{lem:setsizeTeich}. 

Given this, Lemma \ref{lem:wellSeveredPivotTeich}, Lemma \ref{claim:step1Teich} and Lemma \ref{claim:step2Teich} still hold true. Using them, we deduce that almost every path $\w$ does not fall into $G_{n}(Q=1)$ infinitely often. Note also that \[
2D - \frac{2 \eta}{M} \ge( 1-\epsilon) \lambda.
\]
Together with the observation that a.e. path does not avoid $F_{n}$ infinitely often, we deduce that $\liminf \frac{1}{n}\tau(\w_{n}) \ge (1-\epsilon) \lambda$ almost surely as desired.
\end{proof}

%
%

\medskip
\bibliographystyle{alpha}
\bibliography{translation_2}

\end{document}